\documentclass{article}

\usepackage{enumitem}
\usepackage{tikz}
\usepackage{cite}

\usepackage{graphicx}
\usepackage[utf8]{inputenc}
\usepackage{amssymb, amsmath, amsthm}
\usepackage[left=3cm, right=3cm]{geometry}
\usepackage{dsfont}
\usepackage{comment}
\usepackage{url}
\usepackage{multicol}
\usepackage{subcaption}

\usetikzlibrary{decorations.pathreplacing}

\title{Strain-Gradient Plasticity as the $\Gamma$-Limit of a Nonlinear Dislocation Energy with Mixed Growth}
\author{Janusz Ginster}
\date{}

\newcommand{\R}{\mathbb{R}}

\newcommand{\1}{\mathbf{1}}
\newcommand{\lm}{\left\|}
\renewcommand{\rm}{\right\|}

\newcommand{\del}{\partial}

\theoremstyle{definition}
\newtheorem{definition}{Definition}[section]
\newtheorem{remark}{Remark}[section]
\theoremstyle{theorem}
\newtheorem{theorem}{Theorem}[section]
\newtheorem*{theorem*}{Theorem}
\newtheorem{proposition}{Proposition}[section]
\newtheorem{lemma}{Lemma}[section]

\providecommand{\keywords}[1]{\textbf{\textit{Keywords: \,}} #1}
\providecommand{\msc}[1]{\textbf{\textit{2000 Mathematics Subject Classification: \,}} #1}

\begin{document}

\maketitle

\begin{abstract}
  In this paper a we derive by means of $\Gamma$-convergence a macroscopic strain-gradient plasticity from a semi-discrete model for dislocations in an infinite cylindrical crystal. In contrast to existing work, we consider an energy with subquadratic growth close to the dislocations. This allows to treat the stored elastic energy without the need to introduce an ad-hoc cut-off radius. As the main tool to prove a complementing compactness statement, we present a generalized version of the geometric rigidity result for fields with non-vanishing $\operatorname{curl}$. A main ingredient is a fine decomposition result for $L^1$-functions whose divergence is in certain critical Sobolev spaces.
\end{abstract}

\noindent
 \\
%
\msc{49J45, 58K45, 74C05}

\section{Introduction}

The permanent deformation of metals is caused by several defects of the atomic structure. 
Dislocations are topological defects of the metal lattice, \cite{Vo07,Po34,Ta34}, which play an important role in the effect of plastic slip i.e., the relative slip of atom layers which result in a permanent deformation of the metal lattice. Additionally to phenomenologically derived models (see, for example, \cite{FlHu93,KoCuOr02,Gu02} and references therein), there has been extensive research in the mathematical community to derive macroscopic plasticity models from microscopic or mesoscopic dislocation models (see \cite{GaLePo10, Po07, ScZe12, MuScZe14, CeLe05, dLGaPo12, MuGa06, MuGa05, CoGaMu11, CoGaMu16, CoGaMa15, CoGaOr15}).

Many of these dislocation models are formulated in a semi-discrete setting which means that the dislocations are modeled by discrete quantities whereas the elastic strains are averaged.
We consider the situation of an infinite cylindrical crystal with straight parallel edge dislocations.
The most natural setting for this situation is plane elasticity i.e., we restrict our analysis to a plane $\Omega$ which is orthogonal to the dislocation lines. The relevant quantities are then the in-plane components of the strain whereas the dislocations are characterized by their intersection with this plane and the Burgers vector, a vector-valued quantity describing the difference in relative slip, \cite{Bu95}. The presence of dislocations is then expressed by the incompatibility of the elastic strain $\beta: \Omega \rightarrow \R^{2\times 2}$, precisely
\begin{equation} \label{eq: disdens}
\operatorname{curl } \beta = \sum_i \xi_i \delta_{x_i},
\end{equation}
where $x_i$ are the intersection points between $\Omega$ and the dislocation lines and $\xi_i$ the corresponding Burgers vectors, see \cite{Ny53}.
In the linearized setting the stored elastic energy of an admissible planar deformation $\beta$ and a corresponding dislocation density $\mu = \operatorname{curl }\beta$ is given by
\[
\int_{\Omega_{\varepsilon}(\mu)} \frac12\mathcal{C}\beta :\beta \, dx,
\]
where $\mathcal{C} \in \R^{2\times2\times2\times2}$ is an elastic tensor and $\Omega_{\varepsilon}(\mu)$ is the domain given by $\Omega$ after removing discs of radius $\varepsilon$, the so--called core regions, whose size is comparable to the interatomic distance (see also \cite{GaLePo10, CeLe05, MuScZe14}).
This regularization is necessary as strains satisfying \eqref{eq: disdens} typically behave like $\frac1r$ around the dislocations.  Hence, the incompatibility \eqref{eq: disdens} leads to a logarithmic divergence of the quadratic energy around the dislocations.\\
In \cite{GaLePo10}, the authors argue that in the situation of approximately $|\log \varepsilon|$ dislocations the dislocation density and a corresponding elastic strain are of the same order. 
Precisely, they show that the suitably rescaled energy $\Gamma$-converges to a strain-gradient model (see, for example, \cite{FlHu01} and references therein) of the form
\begin{equation}\label{eq: straingrad}
\int_{\Omega} \frac12\mathcal{C} \beta : \beta \, dx + \int_{\Omega} \varphi\left(\frac{d\mu}{d|\mu|}\right) \, d\mu,
\end{equation}
where $\operatorname{curl }\beta = \mu$ and $\varphi$ is the limit self-energy per unit dislocation.
The analogue result for a nonlinear, rotationally invariant energy density $W(\beta) \sim \operatorname{dist}(\beta,SO(2))^2$ was established in \cite{MuScZe14}.
Note that in both papers the authors assume the separateness of the dislocations on an intermediate scale.
First results which do not need this additional assumption were established in \cite{dLGaPo12} in the subcritical regime with finitely many dislocations and in \cite{Gi17} in the regime with $|\log \varepsilon|$ dislocations.
\\
Additionally, it would be desirable to have an energetic description without the need to introduce the core-regions around the dislocations.
It is mainly caused by the use of a quadratic energy density.
At the same time, the use of a linearized elastic energy is justified a few atoms away from the dislocations as the distortion induced by a single dislocation is rather small far away from the dislocation.  
Within the radius of a few atoms around the dislocation it is at least discussible. 
In \cite{ScZe12} the authors consider a nonlinear elastic energy which behaves approximately like $\operatorname{dist}(\beta,SO(2))^2$ as long as $\beta$ is not too large, and like $|\beta|^p$, for some $1<p<2$, if $\beta$ is large, in particular close to the dislocations.
For this growth of the energy density, the contribution inside the cores is finite for a typical strain since $\int_{B_{\varepsilon}} \frac1{|x|^p} \,dx < \infty$. \\
In a regime with finitely many dislocations at given points $x_1,\dots,x_N \in \Omega$ with Burgers vectors $\varepsilon b_1, \dots, \varepsilon b_N$, encoded in the dislocation density $\mu_{\varepsilon} = \sum_{i=1}^N \varepsilon b_i \delta_{x_i}$, the authors prove that the suitably rescaled stored elastic energy $\int_{\Omega} W(\beta) \, dx$ $\Gamma$-converges to
\[
\int_{\Omega} \frac12\mathcal{C} \beta : \beta \, dx + \sum_{i=1}^N \psi(R^T b_i),
\]
where $\mathcal{C} = \frac{\del^2 W}{\del^2 F} W(Id)$, $\operatorname{curl }\beta = 0$, $R$ is a global rotation, which is the footprint of the rotational invariance of the energy, and $\psi$ the self energy per dislocation with a given Burgers vector.
Note that in this regime the limit variables $\beta$ and $\mu$ are completely decoupled, see also \cite{GaLePo10} and \cite{dLGaPo12}. \\
Our contribution will be to study a model with mixed growth as above in the critical regime with $|\log \varepsilon|$ dislocations and derive a strain-gradient energy similar to the one in \eqref{eq: straingrad}, Theorem \ref{theorem: critical}.
Mathematically, the transition from finitely many to infinitely many dislocations is far from trivial.
To ensure compactness, in \cite{GaLePo10} and \cite{MuScZe14} the authors develop generalizations of the classical Korn's inequality, \cite{Ko09}, and its nonlinear counterpart, \cite{FrJaMu02}, for fields which do not have a vanishing $\operatorname{curl}$.
Generalizations to dimensions higher than 2 can be found in \cite{LaLu17}. 
Analogously, we prove in our setting the following, Theorem \ref{theorem: generalizedrigidity}. 
Let $\Omega \subseteq \R^2$ be a simply connected, bounded domain with Lipschitz boundary. 
Then for every $\beta \in L^p(\Omega;\R^{2\times2})$ such that $\operatorname{curl} \beta$ is a bounded vector-valued Radon measure there exists a rotation $R$ such that
\begin{align*}
&\int_{\Omega} \min\{|\beta - R|^2,|\beta-R|\}^p\} \, dx \\ \leq &C \left( \int_{\Omega} \min\{\operatorname{dist}(\beta,SO(2))^2,\operatorname{dist}(\beta,SO(2))^p\} \, dx + |\operatorname{curl }\beta|(\Omega)^2 \right).
\end{align*}
If $\beta$ is a gradient then this estimates reduces to the geometric rigidity result in \cite{MuPa13}.
This estimate allows us to control the rotational invariance of the energy in order to obtain a compactness result for a sequence of suitably rescaled elastic strains which induce a uniformly bounded rescaled elastic energy, Theorem \ref{prop: compactness}. \\
The proof of the generalized (nonlinear) Korn's inequality in \cite{GaLePo10} and \cite{MuScZe14} rely on a fine estimate due to Bourgain and Br\'ezis (see \cite{BoBr03,BoBr07} and also \cite{BrVS07}). 
It states that an $L^1$-function in two dimensions, whose divergence is in $H^{-2}$, is already in $H^{-1}$ and
\[
\| f\|_{H^{-1}} \leq C \left( \| f \|_{L^1} + \| \operatorname{div }f\|_{H^{-2}} \right).
\]
For our mixed-growth situation we prove a corresponding result, namely we show that an $L^1$-function whose divergence is in the space $H^{-2} + W^{-2,p}$, for $1<p<2$, belongs to the space $H^{-1} + W^{-1,p}$.
Corresponding estimates hold, Theorem \ref{thm: bourgainbrezis}.\\ \newline
The paper is organized as follows. 
In Subsection \ref{sec: settingmixed} we introduce notation and the mathematical setting of the problem. 
The main results are presented in Subsection \ref{sec: mainresults}.
In Section \ref{sec: bourgainbrezis} we prove the generalized Bourgain-Br\'ezis decomposition result which we use in Section \ref{sec: generalizedrgidity} to show the generalized geometric rigidity result in the mixed growth case.
Finally, the proof of the $\Gamma$-convergence result can be found in Section \ref{sec: mixed}.

\subsection{Setting of the Problem}\label{sec: settingmixed}
Let $\Omega \subseteq \R^2$ be a simply-connected, bounded domain with Lipschitz boundary representing the cross section of an infinite cylindrical crystal.
The set of (normalized) minimal Burgers vectors for the given crystal is denoted by $S = \{b_1, b_2\}$ for two linearly independent vectors $b_1,b_2 \in \R^2$.
Moreover, set
\begin{equation*}
 \mathbb{S} = \operatorname{span}_{\mathbb{Z}} S = \{ \lambda_1 b_1 + \lambda_2 b_2: \lambda_1,\lambda_2 \in \mathbb{Z}\}
\end{equation*}
to be the the set of (renormalized) admissible Burgers vectors.

\noindent Let $\varepsilon  > 0$ the interatomic distance for the given crystal.
We define the set of admissible dislocation densities as a subset of the $\R^2$-valued Radon measures  $\mathcal{M}(\Omega;\R^2)$, namely
\begin{align*}
  X_{\varepsilon} = \bigg\{ &\mu \in \mathcal{M}(\Omega;\R^2): \mu = \sum_{i=1}^M \varepsilon \xi_i \delta_{x_i}, \, M \in \mathbb{N}, B_{\rho_{\varepsilon}}(x_i) \subseteq \Omega, \, |x_j - x_k| \geq 2 \rho_{\varepsilon}
  \\ &\text{ for } j \neq k, \, 0\neq\xi_i \in \mathbb{S} \bigg\},
 \end{align*}
where we assume that $\rho_{\varepsilon}$ satisfies
\begin{enumerate}
 \item $\lim_{\varepsilon \to 0} \rho_{\varepsilon} / \varepsilon^s = \infty$ for all fixed $s \in (0,1)$ and
 \item $\lim_{\varepsilon \to 0} |\log \varepsilon| \rho_{\varepsilon}^2 = 0$.
\end{enumerate}
This means that we assume the dislocations to be separated on an intermediate scale $\varepsilon \ll \rho_{\varepsilon} \to 0$. \\
Furthermore, we define the set of admissible strains generating $\mu \in X_{\varepsilon}$ by
\begin{equation}\label{eq: strains}
 \mathcal{A}\mathcal{S}_{\varepsilon}(\mu) = \left\{ \beta \in L^p(\Omega,\R^{2\times 2}): \operatorname{curl} \beta = \mu  \text{ in the sense of distributions} \right\}.
\end{equation}
We denote by $SO(2) = \{R \in \R^{2\times2}: R^T R = R R^T = Id, \, \operatorname{det}(R) > 0\}$ the set of rotations and by $|F| = \sqrt{\operatorname{tr}(F^T F)}$ the usual Euclidean norm for $F \in \R^{n\times n}$. \\ \newline
\noindent The energy density $W: \R^{2 \times 2} \rightarrow [0,\infty)$ satisfies the usual assumptions of nonlinear elasticity:
\begin{enumerate}
 \item $W \in C^0(\R^{2\times2})$ and $W \in C^2$ in a neighbourhood of $SO(2)$; \label{item: property1}
 \item stress-free reference configuration: $W(Id) = 0$; \label{item: property2}
 \item frame indifference: $W(RF) = W(F)$ for all $F \in \R^{2 \times 2}$ and $R \in SO(2)$. \label{item: property3}
\end{enumerate}
In addition, we assume that $W$ satisfies the following growth condition:
\begin{enumerate}\setcounter{enumi}{3} 
 \item there exists $\mathbf{1 < p < 2}$ and $0 < c \leq C$ such that for every $F \in \R^{2 \times 2}$ it holds 
 \begin{align}\label{item: property4} 
  c \big(\operatorname{dist}(F,SO(2))^2 \wedge & \operatorname{dist}(F,SO(2))^p\big) \\ \leq &W(F) \leq C \left(\operatorname{dist} (F,SO(2))^2 \wedge \operatorname{dist}(F,SO(2))^p\right). \nonumber
 \end{align} 
\end{enumerate}
Here, $\operatorname{dist}(F,SO(2)) = \min_{S \in SO(2)} |F - S|$ and for $a,b \in \R$ we write $a \wedge b = \min\{a,b\}$. \\ \newline
Finally, we define the critically rescaled energy for $\varepsilon >0$ by
\begin{equation} \label{definitition: energy}
E_{\varepsilon}(\mu,\beta) = \begin{cases}
                   \frac{1}{\varepsilon^2 |\log \varepsilon|^2} \int_{\Omega} W(\beta) \, dx &\text{ if } (\mu,\beta) \in X_{\varepsilon} \times \mathcal{A}\mathcal{S}_{\varepsilon}(\mu), \\
                   + \infty &\text{ else in } \mathcal{M}(\Omega;\R^2) \times L^p(\Omega;\R^{2\times2}).
                  \end{cases}
\end{equation}
We note here that the rescaling by $\varepsilon^2 |\log \varepsilon|^2$ corresponds to a system with $|\log \varepsilon|$ dislocations (see \cite{GaLePo10}). 
This is the only regime in which the strains and the dislocation densities are of the same order and therefore lead to a strain-gradient energy in the limit.\\
The topology which we will use for our $\Gamma$-convergence result is the following.
\begin{definition}\label{def: convcrit}
Let $\varepsilon \to 0$.
 We say that a sequence $(\mu_{\varepsilon},\beta_{\varepsilon}) \subseteq \mathcal{M}(\Omega;\R^2) \times L^p(\Omega;\R^{2\times2})$ converges to a triplet 
 $(\mu,\beta,R) \in  \mathcal{M}(\Omega;\R^2) \times L^p(\Omega;\R^{2\times2}) \times SO(2)$ if there exists a sequence $(R_{\varepsilon})_{\varepsilon} \subseteq SO(2)$ such that
 \begin{equation}
  \frac{\mu_{\varepsilon}}{\varepsilon |\log \varepsilon|} \stackrel{*}{\rightharpoonup} \mu \text{ in } \mathcal{M}(\Omega;\R^2),
 \end{equation}
\begin{equation}
 \frac{R_{\varepsilon}^T \beta_{\varepsilon} - Id}{\varepsilon |\log \varepsilon|} \rightharpoonup \beta \text{ in } L^p(\Omega;\R^{2\times 2}), \text{ and } R_{\varepsilon} \to R.
\end{equation}
\end{definition}
This definition will be justified by the compactness result, Theorem \ref{prop: compactness}.

\subsection{Main Results}\label{sec: mainresults}
In this paper, we prove the following $\Gamma$-convergence result.

\begin{theorem} \label{theorem: critical}
 The energy functionals $E_{\varepsilon}$ defined as in \eqref{definitition: energy} $\Gamma$-converge with respect to the notion of convergence given in Definition \ref{def: convcrit} to the functional $E^{crit}$ 
 defined on $\mathcal{M}(\Omega;\R^2) \times L^p(\Omega; \R^{2 \times 2}) \times SO(2)$ as
 \begin{equation*}
  E^{crit}(\mu,\beta,R) = \begin{cases}
                    \frac12 \int_{\Omega} \mathcal{C} \beta : \beta \,dx + \int_{\Omega} \varphi\left(R, \frac{d \mu}{d |\mu|}\right) d|\mu| &\text{if } \mu \in H^{-1}(\Omega;\R^2) \cap \mathcal{M}(\Omega;\R^2), \\ &\beta \in L^2(\Omega; \R^{2 \times 2}),
                    \operatorname{curl }\beta = R^T \mu, \\
                    +\infty &\text{otherwise },
                   \end{cases}
 \end{equation*}
where $\mathcal{C} = \frac{\del^2 W}{\del^2 F}(Id)$. 
The function $\varphi$ is the relaxed self-energy density per unit dislocation. 
It will be defined in \eqref{definition: phi}.
\end{theorem}
Note that the limit energy is the same as also found in \cite{GaLePo10} and \cite{MuScZe14}.
We complement the $\Gamma$-convergence result with a corresponding compactness statement.

\begin{theorem}[Compactness] \label{prop: compactness}
 Let $\varepsilon_j \rightarrow 0$. Let $(\beta_j, \mu_j)_j$ be a sequence in $L^p(\Omega, \R^{2\times2}) \times \mathcal{M}(\Omega;\R^2)$ such that $\sup_j E_{\varepsilon_j}(\mu_j,\beta_j) < \infty$. 
 Then there exist a sequence $(R_j) \subseteq SO(2)$, a rotation $R \in SO(2)$, a measure $\mu \in \mathcal{M}(\Omega;\R^2) \cap H^{-1}(\Omega;\R^2)$, a function $\beta \in L^2(\Omega;\R^{2\times2})$ and a subsequence (not relabeled) such that \\
 \centering
 \begin{tabular}{ll}
 \parbox{7cm}{
\begin{enumerate}
 \item $\frac{\mu_j}{\varepsilon_j |\log \varepsilon_j|} \stackrel{*}{\rightharpoonup} \mu$ in $\mathcal{M}(\Omega;\R^2)$,
  \item $\frac{R_j^T \beta_j - Id}{\varepsilon_j |\log \varepsilon_j|} \rightharpoonup \beta$ in $L^p(\Omega;\R^{2\times2})$,
 \end{enumerate}
}
 &
 \parbox{5cm}{
 \begin{enumerate}\setcounter{enumi}{2}
 \item $R_j \rightarrow R$,
  \item $\operatorname{curl} \beta = R^T \mu$.
 \end{enumerate}
}
\end{tabular} 
 \end{theorem}
The main tool to obtain the compactness result is a generalized rigidity result for energies with mixed growth. 
It generalizes the rigidity result in \cite{MuPa13} to fields whose $\operatorname{curl }$ is a vector-valued Radon measure.

\begin{theorem} \label{theorem: generalizedrigidity}
Let $1 < p<2$ and $\Omega \subseteq \R^2$ open, simply connected, bounded with Lipschitz boundary. 
 Then there exists a constant $C>0$ such that for every $\beta \in L^p(\Omega;\R^{2 \times 2})$ with $\operatorname{curl} \beta \in \mathcal{M}(\Omega;\R^2)$ there exists a rotation $R \in SO(2)$ such that
 \begin{equation*}
  \int_{\Omega} | \beta - R|^2 \wedge |\beta - R|^p \, dx \leq C \left( \int_{\Omega} \operatorname{dist}(\beta, SO(2))^2 \wedge \operatorname{dist}(\beta, SO(2))^p dx + | \operatorname{curl} \beta|(\Omega)^2 \right).
 \end{equation*}
\end{theorem}
In order to show the rigidity result we generalize a fine estimate from Bourgain and Br\'ezis (see \cite{BoBr03} and \cite{BoBr07}).
We prove that an $L^1$-function whose divergence is in the space $H^{-2} + W^{-2,p}$, for $1<p<2$, can be decomposed into a part in $H^{-1}$ and $W^{-1.p}$. 
Corresponding estimates hold.

\begin{theorem}[Bourgain-Br\'{e}zis type estimate] \label{thm: bourgainbrezis}
Let $1 < p < 2$ and $\Omega \subseteq \R^2$ open, simply connected, and bounded with Lipschitz boundary.
 Then there exists a constant $C>0$ such that for every $f \in L^1(\Omega;\R^2)$ satisfying $\operatorname{ div } f = a + b \in H^{-2}(\Omega) + W^{-2,p}(\Omega)$ there exist $A \in H^{-1}(\Omega;\R^2)$ and  $B \in W^{-1,p}(\Omega;\R^2)$ such that the following holds:
 \begin{enumerate}
  \item $f = A + B$,
  \item $\| A \|_{H^{-1}(\Omega;\R^2)} \leq C ( \| f \|_{L^1(\Omega;\R^2)} + \| a \|_{H^{-2}(\Omega)})$,
  \item $\| B \|_{W^{-1,p}(\Omega;\R^2)} \leq C \| b \|_{W^{-2,p}(\Omega)}$.
 \end{enumerate}
\end{theorem}

\begin{remark}
The original statement by Bourgain and Br\'ezis, \cite{BoBr07}, in our setting is recovered for $b=0$
but their statement holds in a much more general setting.
For $\Pi^n$ being the $n$-torus and $r \in \mathbb{N}$ it is shown in \cite[Theorem 10]{BoBr07} that it is sufficient that for an operator $S: W^{1,n}(\Pi^n;\mathbb{R}^r) \to X$ with closed range, where $X$ is a Banach space, there exists for each $1 \leq s \leq r$ an index $1 \leq i_s \leq d$ such that for all functions $f\in W^{1,n}(\Pi^n;\mathbb{R}^r)$ it holds that
\begin{equation*}
\lm Sf \rm \leq C \max_{1 \leq s \leq r} \max_{i \neq i_s} \lm \partial_i f_s \rm_{L^n}
\end{equation*}
to guarantee that for any $f \in W^{1,n}(\Pi^n;\mathbb{R}^r)$ there exists a function $g \in W^{1,n}(\Pi^n;\mathbb{R}^r) \cap L^{\infty}(\Pi^n;\mathbb{R}^r)$ satisfying $S(f)=S(g)$ and corresponding bounds.
Clearly, this condition holds true for the $\operatorname{div}$-operator.
\end{remark}

\section{A Bourgain-Br\'ezis-Inequality in Two Dimensions}\label{sec: bourgainbrezis}
In this section we prove the Bourgain-Br\'ezis-type decomposition result, Theorem \ref{thm: bourgainbrezis}. 
We start with the primal result on the torus which we will then localize and later extend to Lipschitz domains.
The main result then follows by dualization.
\subsection{The Case of a Torus}
In this section, we prove a primal version of the Bourgain-Br\'{e}zis type estimate on the two dimensional torus, which we simply denote by $\Pi\cong [-\pi,\pi)^2$ in the following.
To be precise, we show the following statement.
\begin{theorem} \label{theorem: iteration}
 Let $2 < q < \infty$. Then there exists a constant $C>0$ such that for all functions $f \in L^2(\Pi) \cap L^q(\Pi)$ satisfying $\int_{\Pi} f = 0$ 
 there exists a function $F \in L^{\infty}(\Pi;\R^2) \cap H^1(\Pi;\R^2) \cap W^{1,q}(\Pi;\R^2)$ such that \\
\begin{tabular}{ll}
 \parbox{5cm}{
\begin{enumerate}
  \item $\operatorname{div} F = f$,
  \item $\| F \|_{L^{\infty}} \leq C \| f \|_{L^2}$,
 \end{enumerate}
}
 &
 \parbox{5cm}{
 \begin{enumerate}\setcounter{enumi}{2}
  \item $\| F \|_{H^1} \leq C \| f \|_{L^2}$,
  \item $\| F \|_{W^{1,q}} \leq C \| f \|_{L^q}$.
 \end{enumerate}
}
\end{tabular} 
\end{theorem}

In the proof we will adapt the strategy from \cite{BoBr03}.
 \\
The main ingredient to prove Theorem \ref{theorem: iteration} is the following lemma which gives a first approximation to Theorem \ref{theorem: iteration}.
It shows that the equation $\operatorname{div} F = f$ can be almost solved by a function $F$ which satisfies estimates with a good linear term and a bad nonlinear term. 
\begin{lemma}[Nonlinear approximation] \label{lemma: nonlinear approximation}
 Let $2 < q < \infty$. Then there exists $c>0$ such that for all $f \in L^2(\Pi) \cap L^q(\Pi)$  satisfying $\| f \|_{L^2} \leq c$ and $\int_{\Pi} f = 0$ the following holds: \\ 
 For every $\delta > 0$ there exist $C_{\delta} > 0$ and $F \in L^{\infty}(\Pi;\R^2)\cap H^1(\Pi;\R^2) \cap W^{1,q}(\Pi;\R^2)$
 such that 
\begin{enumerate}
  \item $\| F \|_{L^{\infty}} \leq C_{\delta}$,
  \item $\| F \|_{H^1} \leq C_{\delta} \| f \|_{L^2}$,
  \item $\| \operatorname{div} F - f \|_{L^2} \leq \delta \| f \|_{L^2} + C_{\delta} \| f \|_{L^2}^2$,
  \item $\| F \|_{W^{1,q}} \leq C_{\delta} \| f \|_{L^q}$,
    \item $\| \operatorname{div} F - f \|_{L^q} \leq \delta \| f \|_{L^2} + C_{\delta} \| f \|_{L^2} \| f \|_{L^q}$. \label{eq: 1 5}
 \end{enumerate} 
\end{lemma}
\begin{proof}
 \begin{figure}[t]
 \centering
 \begin{tikzpicture}[scale = 0.4]
  \draw[->] (-10,0) -- (10,0);
  \draw[->] (0,-10) -- (0,10);
  \draw (10,0) node [anchor=west]{$n_1$};
  \draw (0,10) node [anchor=west]{$n_2$};
  
  \draw (1,-2) -- (1,2) -- (2,2) -- (2,-2) -- (1,-2);
  \draw (-1,-2) -- (-1,2) -- (-2,2) -- (-2,-2) -- (-1,-2);
  
  \draw (2,0) -- (2,4) -- (4,4) -- (4,-4) -- (2,-4) -- (2,0);
  \draw (-2,0) -- (-2,4) -- (-4,4) -- (-4,-4) -- (-2,-4) -- (-2,0);
  
  \draw (4,-8) -- (4,8) -- (8,8) -- (8,-8) -- (4,-8);
  \draw (-4,-8) -- (-4,8) -- (-8,8) -- (-8,-8) -- (-4,-8);
  
  \fill[draw=black] (1,0) circle[radius=0.1];
  \fill[draw=black] (2,0) circle[radius=0.1];
  \fill[draw=black] (4,0) circle[radius=0.1];
  \fill[draw=black] (8,0) circle[radius=0.1];
  \fill[draw=black] (-1,0) circle[radius=0.1];
  \fill[draw=black] (-2,0) circle[radius=0.1];
  \fill[draw=black] (-4,0) circle[radius=0.1];
  \fill[draw=black] (-8,0) circle[radius=0.1];
  \fill[draw=black] (0,2) circle[radius=0.1];
  \fill[draw=black] (0,4) circle[radius=0.1];
  \fill[draw=black] (0,8) circle[radius=0.1];
  \fill[draw=black] (0,-2) circle[radius=0.1];
  \fill[draw=black] (0,-4) circle[radius=0.1];
  \fill[draw=black] (0,-8) circle[radius=0.1];
  
  \draw (1,0) node [anchor=north,font=\tiny]{$2^{j-1}$};
  \draw (2,0) node [anchor=north,font=\tiny]{$2^{j}$};
  \draw (4,0) node [anchor=north,font=\tiny]{$2^{j+1}$};
  \draw (8,0) node [anchor=north,font=\tiny]{$2^{j+2}$};
  
  \draw (0,2) node [anchor=north,font=\tiny]{$2^{j}$};
  \draw (0,4) node [anchor=north,font=\tiny]{$2^{j+1}$};
  \draw (0,8) node [anchor=north,font=\tiny]{$2^{j+2}$};
  
  \draw (1.5,1) node [font=\tiny]{$\Lambda^1_j$};
  \draw (-1.5,1) node [font=\tiny]{$\Lambda^1_j$};
  
  \draw (3,1) node [font=\tiny]{$\Lambda^1_{j+1}$};
  \draw (-3,1) node [font=\tiny]{$\Lambda^1_{j+1}$};
  
   \draw (6,1) node [font=\tiny]{$\Lambda^1_{j+2}$};
  \draw (-6,1) node [font=\tiny]{$\Lambda^1_{j+2}$};
 \end{tikzpicture}
 \caption{Sketch of $\Lambda^1$.}
 \label{figure: Lambda^1}
\end{figure}
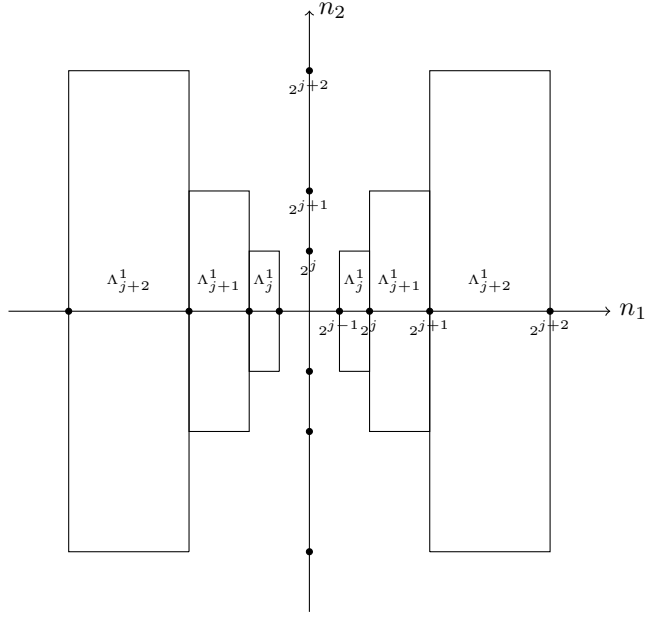
Let $f \in L^2(\Pi) \cap L^q(\Pi)$  such that $\int_{\Pi} f = 0$ and $\| f \|_{L^2} \leq c$ where $c>0$ will be fixed later.\\
\noindent Consider the following decomposition of $\mathbb{Z}^2\setminus\{0\}$, see Figure \ref{figure: Lambda^1},
 \begin{align} \label{eq: Lambda1jr}
  &\Lambda_j^1 = \left\{ 2^{j-1} < |n_1| \leq 2^j; |n_2| \leq 2^j \right\}\\  \text{ and } &\Lambda_j^2 = \left\{ 2^{j-1} < |n_2| \leq 2^{j}; |n_1| \leq 2^{j-1} \right\} \text{ for } j \in \mathbb{N}. \nonumber
 \end{align}
For $\alpha = 1,2$ set $\Lambda^{\alpha} = \bigcup_j \Lambda_j^{\alpha}$.
Correspondingly, let $f^{\alpha} = P_{\Lambda^{\alpha}}f = \sum_{n \in \Lambda^{\alpha}} \hat{f}(n) e^{in\cdot x}$ and decompose $f = f^1 + f^2$.
In the following, we construct functions $Y_{\alpha}: \Pi \rightarrow \R$ which satisfy 
\begin{enumerate}
  \item $\| Y_{\alpha} \|_{L^{\infty}} \leq C_{\delta}$, \label{item: Y1}
  \item $\| Y_{\alpha} \|_{H^1} \leq C_{\delta} \| f \|_{L^2}$,
  \item $\| \del_{\alpha} Y_{\alpha} - f^{\alpha} \|_{L^2} \leq \delta \| f \|_{L^2} + C_{\delta} \| f \|_{L^2}^2$, \label{item: Y3}
  \item $\| Y_{\alpha} \|_{W^{1,q}} \leq C_{\delta} \| f \|_{L^q}$, \label{item: Y4}
  \item $\| \del_{\alpha} Y_{\alpha} - f^{\alpha}\|_{L^q} \leq \delta \| f \|_{L^2} + C_{\delta} \| f \|_{L^2} \| f \|_{L^q}$. \label{item: Y5}
 \end{enumerate}
  Without loss of generality we may assume that $f=f^1$ and construct only $Y_1$.
  We follow the construction in \cite{BoBr03}. \\
Define
  \begin{equation*}
   f_j = P_{\Lambda_j^1} f = \sum_{n\in \Lambda_j^1} \hat{f}(n) \, e^{in\cdot x} \text{ and } F_j = \sum_n \frac1{n_1} \hat{f_j}(n) e^{in\cdot x} = \sum_{n \in \Lambda_j^1} \frac1{n_1} \hat{f}(n) e^{in\cdot x}.
  \end{equation*}
 Moreover, fix a small $\varepsilon > 0$ and subdivide $\Lambda^1_j$ in stripes of length $\sim \varepsilon 2^{j-1}$ by setting
  \begin{equation*}
   \Lambda_j^1 = \bigcup_{0 \leq r \leq 2 \lfloor \varepsilon^{-1} \rfloor + 1} \Lambda^1_{j,r},
  \end{equation*}
where for $0 \leq r \leq  \lfloor \varepsilon^{-1} \rfloor$ we set $\Lambda_{j,r}^1 = I_j^r \times [-2^j,2^j]$ whereas  for $\lfloor \varepsilon^{-1} \rfloor + 1 \leq r \leq 2 \lfloor \varepsilon^{-1} \rfloor + 1$ we set $\Lambda_{j,r}^1 = I_j^{r - \lfloor \varepsilon^{-1} \rfloor -1} \times [-2^j,2^j]$ where $I_k^r$ and $J_k^r$ are defined as 
\begin{align}
I^r_k &= (2^{k-1} + r \varepsilon 2^{k-1},2^{k-1} + (r+1) \varepsilon 2^{k-1}] \cap \mathbb{Z}, \label{eq: IkrJkr} \\
J^r_k &= [-2^{k-1} - (r+1) \varepsilon 2^k,-2^{k-1} - r \varepsilon 2^{k-1}) \cap \mathbb{Z} \text{ for } r=0, \dots, \lfloor \varepsilon^{-1} \rfloor - 1, \nonumber   \\
I_k^{\lfloor \varepsilon^{-1} \rfloor} &=  \left(2^{k-1} + \lfloor \varepsilon^{-1} \rfloor \varepsilon 2^{k-1}, 2^{k} \right] \cap \mathbb{Z}, \nonumber \\
\text{ and } J_k^{\lfloor \varepsilon^{-1} \rfloor} &= \left[-2^{k},-2^{k-1} - \lfloor \varepsilon^{-1} \rfloor  \varepsilon 2^{k-1}\right) \cap \mathbb{Z}. \nonumber
\end{align}
For a sketch of the situation, see Figure \ref{figure: stripes}.
  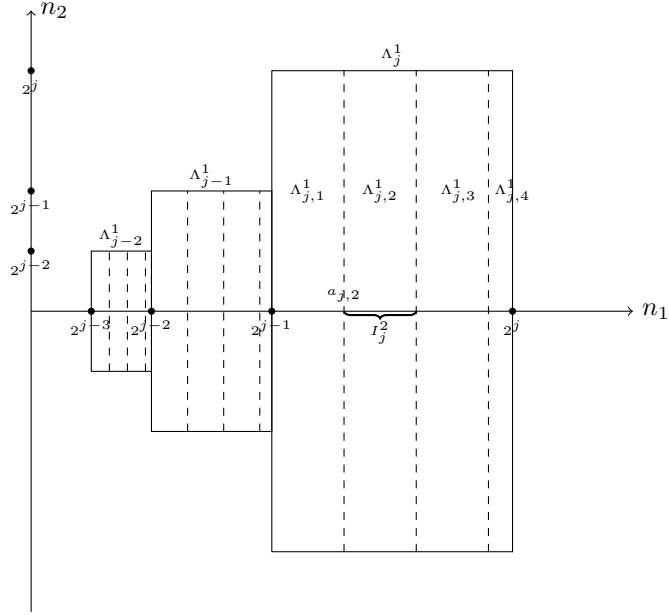
\begin{figure}
 \centering
 \begin{tikzpicture}[scale = 0.4]
  \draw[->] (0,0) -- (20,0);
  \draw[->] (0,-10) -- (0,10);
  \draw (20,0) node [anchor=west]{$n_1$};
  \draw (0,10) node [anchor=west]{$n_2$};
  
  \draw (2,-2) -- (2,2) -- (4,2) -- (4,-2) -- (2,-2);
  \draw[dashed] (2.6,-2) -- (2.6,2);
  \draw[dashed] (3.2,-2) -- (3.2,2);
  \draw[dashed] (3.8,-2) -- (3.8,2);

  \draw (4,0) -- (4,4) -- (8,4) -- (8,-4) -- (4,-4) -- (4,0);
  \draw[dashed] (5.2,-4) -- (5.2,4);
  \draw[dashed] (6.4,-4) -- (6.4,4);
  \draw[dashed] (7.6,-4) -- (7.6,4);
  
  \draw (8,-8) -- (8,8) -- (16,8) -- (16,-8) -- (8,-8);
  \draw[dashed] (10.4,-8) -- (10.4,8);
  \draw[dashed] (12.8,-8) -- (12.8,8);
  \draw[dashed] (15.2,-8) -- (15.2,8);
  
  \fill[draw=black] (2,0) circle[radius=0.1];
  \fill[draw=black] (4,0) circle[radius=0.1];
  \fill[draw=black] (8,0) circle[radius=0.1];
  \fill[draw=black] (16,0) circle[radius=0.1];  
  \fill[draw=black] (0,2) circle[radius=0.1];
  \fill[draw=black] (0,4) circle[radius=0.1];
  \fill[draw=black] (0,8) circle[radius=0.1];

  \draw (3,2.5) node [font=\tiny]{$\Lambda^1_{j-2}$};

  \draw (6,4.5) node [font=\tiny]{$\Lambda^1_{j-1}$};

   \draw (12,8.5) node [font=\tiny]{$\Lambda^1_{j}$};
  
  \draw [
    thick,
    decoration={
        brace,
        mirror,
        raise=0.2
    },
    decorate
] (10.4,0) -- (12.8,0) 
node [font=\tiny, pos=0.5,anchor=north] {$I_j^2$};
  
\draw (10.4,0) node[anchor=south,font=\tiny] {$a_{j,2}$};
  \draw (2,0) node [anchor=north,font=\tiny]{$2^{j-3}$};
  \draw (4,0) node [anchor=north,font=\tiny]{$2^{j-2}$};
  \draw (8,0) node [anchor=north,font=\tiny]{$2^{j-1}$};
  \draw (16,0) node [anchor=north,font=\tiny]{$2^{j}$};
  
  \draw (0,2) node [anchor=north,font=\tiny]{$2^{j-2}$};
  \draw (0,4) node [anchor=north,font=\tiny]{$2^{j-1}$};
  \draw (0,8) node [anchor=north,font=\tiny]{$2^{j}$};
  
  \draw (9.2,4) node [font=\tiny]{$\Lambda^1_{j,1}$};
  \draw (11.6,4) node [font=\tiny]{$\Lambda^1_{j,2}$};
  \draw (14.2,4) node [font=\tiny]{$\Lambda^1_{j,3}$};
  \draw (16,4) node [font=\tiny]{$\Lambda^1_{j,4}$};
  
 \end{tikzpicture}
 \caption{Sketch of the subdivision of $\Lambda^1$ into the stripes $\Lambda^1_{j,r}$ for positive $n_1$.}
 \label{figure: stripes}
\end{figure}
  Next, define 
  \begin{equation*}
   \tilde{F}_j (x) = \sum_r \left| \sum_{n \in \Lambda^1_{j,r}} \frac1{n_1} \hat{f}_j(n) e^{in\cdot x} \right|.
  \end{equation*}
The main property of $\tilde{F}_j$ is the smallness of its partial derivative in $x_1$-direction. 
In fact, we can rewrite $\tilde{F}_j (x) = \sum_r \left| \sum_{n \in \Lambda^1_{j,r}} \frac1{n_1} \hat{f}_j(n) e^{in\cdot x} e^{-ia_{j,r} x_1}\right|$ where $a_{j,r}$ is the left endpoint of $I_j^r$ and the right endpoint of $J_j^{r - \lfloor \varepsilon^{-1} \rfloor - 1}$, respectively. 
Differentiation leads to
\begin{equation}
|\del_1 \tilde{F}_j| = \sum_r \left| \sum_{n \in \Lambda^1_{j,r}} \frac{n_1 - a_{j,r}}{n_1} \hat{f}_j(n) e^{in\cdot x} \right|. \label{eq: tildeF}
\end{equation}
  As $0 \leq \frac{n_1 - a_{j,r}}{n_1} \leq \varepsilon$, we derive, using using Plancharel's identity and H\"older's inequality for the sum over $r$, that 
\begin{align*}
\| \del_1 \tilde{F}_j \|_{L^2} &\leq C \varepsilon^{-\frac12} \sum_r \lm \sum_{n \in \Lambda_{j,r}^1}  \frac{n_1 - a_{j,r}}{n_1} \hat{f}(n) e^{i n\cdot x} \rm_{L^2} \\
&\leq C \varepsilon^{\frac12} \| f_j \|_{L^2}.
\end{align*}  
  \noindent As we also need an appropriate localization in Fourier space of $\tilde{F}_j$, let us recall that the $n$-th one-dimensional F\'{e}jer-kernel is given by
  \begin{equation*}
   K_n(t) = \sum_{|k| < n} \frac{n - |k|}{n} e^{ikt} = \frac1n \frac{1 - \cos(nt)}{1 - \cos(t)} \geq 0.
  \end{equation*}
  If we define
  \begin{equation*}
   G_j = 9 \tilde{F}_j * \left( K_{2^{j+1}} \otimes K_{2^{j+1}} \right),
  \end{equation*}
  we obtain by the properties of the F\'{e}jer kernel that
  \begin{equation}\label{eq: localization Gj}
   \operatorname{supp} \hat{G}_j \subseteq [-2^{j+1},2^{j+1}] \times [-2^{j+1},2^{j+1}] \subseteq \{ |n| \leq C 2^j \} \text{ and } |F_j| \leq |\tilde{F}_j| \leq G_j.
  \end{equation}
  Moreover, in the proof of \cite[Theorem 1]{BoBr03} it is shown that
  \begin{align}
   &\| G_j \|_{L^{\infty}} \leq 9 \| \tilde{F}_j \|_{\infty} \leq C \| f_j \|_{L^2}, 
   &\| G_j \|_{L^2} \leq C \varepsilon^{-\frac12} 2^{-j} \| f_j \|_{L^2}, \label{eq: Gjest1} \\ 
   &\| \del_1 G_j \|_{L^2} \leq C \varepsilon^{\frac12} \| f_j \|_{L^2}, 
   &\| \nabla G_j \|_{L^2} \leq C \varepsilon^{-\frac12} \| f_j \|_{L^2}. \label{eq: Gjest3}
  \end{align}
As in \cite{BoBr03}, we define
  \begin{equation*}
   Y_1 = \sum_j F_j \prod_{k > j} (1-G_k).
  \end{equation*}
  By \eqref{eq: localization Gj} and \eqref{eq: Gjest1}, it holds $|F_j| \leq C \|f_j\|_{L^2} \leq C \lm f \rm_{L^2}$.
We assume that $\lm f \rm_{L^2}$, respectively $c$ in the formulation of the theorem, is so small that $C \lm f \rm_{L^2} < 1$.
 Then, it follows that
  \begin{equation}
   | Y_1 | \leq \sum_j | F_j | \prod_{k > j} (1 - | F_k |) \leq 1. \label{eq: bound Y_1}
  \end{equation}
  Another calculation, see \cite[equation (5.19)]{BoBr03}, shows that
  \begin{equation*}
   Y_1 = \sum_j F_j - \sum_j G_j H_j,
  \end{equation*}
  where 
  \begin{equation*}
   H_j = \sum_{k < j} F_k \prod_{k < l < j} (1 - G_l).
  \end{equation*}
  Thus,
  \begin{equation} \label{eq: del1Y}
   \del_1 Y_1 = \sum_j f_j - \sum_j \del_1(G_j H_j) = f - \sum_j \del_1(G_j H_j).
  \end{equation}
  Moreover, by definition of $H_j$ and $F_j$, and \eqref{eq: localization Gj} it can be seen that 
  \begin{align}\label{eq: localizationGjHj}
&|H_j|\leq 1, \, \,\,  \operatorname{supp} \hat{H}_j \subseteq \{|n| \leq C 2^{j}\}, \,\,\, P_k(G_j H_j) = 0 \text{ for all } k>j+m, \\&\text{ and } G_jH_j = \sum_{k \leq j + m} P_k(G_j H_j), \nonumber
  \end{align}
  where the $P_k$ are smooth Littlewood-Paley-projections on $\{|n| \sim 2^k\}$, and $m$ is independent of $j$.
  In \cite[proof of Theorem 1]{BoBr03}, Bourgain and Br\'ezis show, using \eqref{eq: localization Gj} - \eqref{eq: localizationGjHj}, that
  \begin{equation}
   \| \del_1 Y_1 - f\|_{L^2} \leq C \log(\varepsilon^{-1}) \left(\varepsilon^{\frac12} \| f \|_{L^2} + \varepsilon^{-\frac12} \| f \|_{L^2}^2\right), \text{ and } \| Y_1 \|_{H^1} \leq C_{\varepsilon} \lm f \rm_{L^2}. \label{eq: L2 estimate}
  \end{equation}
Hence, properties \ref{item: Y1}.--\ref{item: Y3}.~for $Y_1$ are already shown. 
  In what follows, we adopt the ideas of their proof to show the corresponding estimates in $L^q$, namely \eqref{item: Y4} and \eqref{item: Y5}.\\
  First, we estimate 
  \begin{equation} 
   \| \nabla Y_1 \|_{L^q} \leq \| \nabla \sum_j F_j \|_{L^q} + \| \nabla \sum_j G_j H_j \|_{L^q}. \label{eq: decomposenablaY}
  \end{equation}
  For the first term on the right hand side, we observe
  \begin{align}
   \left\| \sum_j \nabla F_j \right\|_{L^q} &= \left\| \sum_j \sum_{n \in \Lambda^1_j} \frac{n}{n_1} \hat{f}(n) e^{in\cdot x} \right\|_{L^q} \nonumber \\
   &\leq C \left\| \sum_j \sum_{n \in \Lambda^1_j} \hat{f}(n) e^{in\cdot x} \right\|_{L^q} \nonumber \\
   &= C \left\| f \right\|_{L^q}. \label{eq: nablaYeasy}
  \end{align}
 Here, we used for the inequality that $\frac{n}{n_1} \1_{\bigcup_j \Lambda_j^1}$ is an $L^p$-multiplier. 
This can be shown by multiplier transference and the Marcinkiewicz multiplier theorem (note that in $\Lambda_j$ the second variable $n_2$ is controlled by $2 n_1$).
  \noindent For the second term of the right hand side in \eqref{eq: decomposenablaY} we obtain, using classical Littlewood-Paley estimates, that
  \begin{align}
   \lm \sum_j \nabla (G_j H_j) \rm_{L^q} &\leq C \lm \left( \sum_k \left| P_{k} \sum_j \nabla(G_jH_j) \right|^2 \right)^{\frac12}  \rm_{L^q}. \label{eq: gradient GjHj1}\\
\intertext{As the operator $\nabla$ is a Fourier multiplication operator, it clearly commutes with the Littlewood-Paley projections $P_k$. In particular, the locality in Fourier space of $G_jH_j$ in \eqref{eq: localizationGjHj} also holds for $\nabla G_jH_j$.
Then the triangle inequality and the change of variables $j \to k+s$ yield }
&\leq C \sum_{s \geq -m} \lm \left( \sum_k \left| P_{k} \nabla(G_{k+s}H_{k+s}) \right|^2 \right)^{\frac12}  \rm_{L^q}. \\
\intertext{The change $k \to k-s$ and the Littlewood-Paley inequality for gradients lead to}
   &= C \sum_{s \geq -m} \lm \left( \sum_{k\geq s} |P_{k-s} \nabla (G_k H_k)|^2 \right)^{\frac12} \rm_{L^q}. \label{eq: inbetween}\\
   &\leq C \sum_{s \geq -m} 2^{-s} \lm \left( \sum_k |2^k G_k \underbrace{H_k}_{|\cdot|\leq 1}|^2 \right)^{\frac12} \rm_{L^q} \\
   &\leq C \sum_{s \geq -m} 2^{-s} \lm \left( \sum_j |2^k G_k|^2 \right)^{\frac12} \rm_{L^q} \label{eq: 2^jGj1}. \\ 
   \intertext{By definition, $G_k$ is the convolution of $\tilde{F_k}$ with a Fej\'{e}r kernel which can be bounded from above by the maximal operator. Vector-valued estimates for the maximal operator then give} \nonumber \\
   &\leq C \sum_{s \geq -m} 2^{-s} \lm \left( \sum_k | 2^k \tilde{F}_k |^2 \right)^{\frac12} \rm_{L^q}. \\
   &= C \sum_{s \geq -m} 2^{-s} \lm \left( \sum_k \left( \sum_{r \leq 2 \lfloor \varepsilon^{-1} \rfloor -1} \left| \sum_{n \in \Lambda^1_{k,r}} \frac{2^k}{n_1} \hat{f}(n) e^{in \cdot x} \right| \right)^2 \right)^{\frac12} \rm_{L^q}. \\   
   \intertext{Using H\"older's inequality for the sum over $r$ yields} \nonumber \\
   &\leq C \varepsilon^{-\frac12} \sum_{s \geq -m} 2^{-s} \lm \left( \sum_k \sum_{r \leq 2 \lfloor \varepsilon^{-1} \rfloor -1} \left| \sum_{n \in \Lambda^1_{k,r}} \frac{2^k}{n_1} \hat{f}(n) e^{in \cdot x} \right|^2  \right)^{\frac12} \rm_{L^q}. \label{eq: 20} \\
   \intertext{Now, we use a one-sided Littlewood-Paley-type inequality for non-dyadic decompositions which goes back to  Rubio de Francia, \cite[Theorem 8.1]{Ru85}.
 For the case of a torus, see \cite[Theorem 2.5]{GiTo04} or \cite{Bo85} for the dual statement. 
   The statement is the following: For $q>2$ there exists a constant $C> 0$ such that for all partitions of $\mathbb{Z}$ into intervals $(I_k)_k$ it holds that
   \begin{equation*}
    \lm ( \sum_k |S_k f|^2 )^{\frac12} \rm_{L^q(\Pi^1)} \leq C \lm f \rm_{L^q(\Pi^1)},
   \end{equation*}
   where $S_k f = \sum_{l \in I_k} \hat{f}(l) e^{il\cdot x}$.
   We use this inequality in the first variable for the decomposition of the $n_1$-axis given by $\Lambda_{k,r}^1$, more precisley by $I_k^r$ and $J_k^r$, } \nonumber \\
   \eqref{eq: 20}&\leq C \varepsilon^{-\frac12} \sum_{s \geq -m} 2^{-s} \lm  \sum_{k,r \leq 2 \lfloor \varepsilon^{-1} \rfloor - 1}  \sum_{n \in \Lambda^1_{k,r}} \frac{2^k}{n_1}\hat{f}(n) e^{in \cdot x}  \rm_{L^q}. \label{eq: 21}\\
   \intertext{Finally, we use that $\sum_k \frac{2^k}{n_1} \1_{\Lambda_k^1}$ is an $L^q$-multiplier to obtain }
   &\leq C \varepsilon^{-\frac12} \sum_{s \geq -m} 2^{-s} \bigg\|  \underbrace{\sum_{k, r \leq 2 \lfloor \varepsilon^{-1} \rfloor - 1}  \sum_{n \in \Lambda^1_{k,r}}}_{\sum_{n \in \Lambda^1}} \hat{f}(n) e^{in \cdot x}  \bigg\|_{L^q} \\
   &= C \varepsilon^{-\frac12} \sum_{s \geq -m} 2^{-s} \lm f \rm_{L^q} \label{eq: gradient HjHj last-1} \\
   &\leq C \varepsilon^{-\frac12} \lm f \rm_{L^q}. \label{eq: 2^jGjlast}
  \end{align}
 Collecting \eqref{eq: decomposenablaY}, \eqref{eq: nablaYeasy}, and  \eqref{eq: gradient GjHj1} - \eqref{eq: 2^jGjlast} leads to
  \begin{equation*}
   \lm \nabla Y_1 \rm_{L^q} \leq C \varepsilon^{-\frac12} \lm f \rm_{L^q}.
  \end{equation*}
  As we may assume without loss of generality that $\int_{\Pi} Y_1 = 0$, this implies
  \begin{equation}
   \lm Y \rm_{W^{1,q}} \leq C \varepsilon^{-\frac12} \lm f \rm_{L^q}. \label{eq: YW1q}
  \end{equation}
  Hence, it is left to prove property \ref{item: Y5}.~for $Y_1$.
  By \eqref{eq: del1Y} it remains to control  the term $\lm \del_1 \sum_j (G_j H_j) \rm_{L^q}$. 
  As in \eqref{eq: gradient GjHj1}--\eqref{eq: inbetween}, we can estimate
  \begin{equation*}
   \lm \del_1 \sum_j G_j H_j \rm_{L^q} \leq \sum_{s \geq -m} \lm \left(\sum_j \left|P_{j-s} \del_1 (G_j H_j) \right|^2 \right)^{\frac12} \rm_{L^q}. 
  \end{equation*}
  Now, fix $s_* \in \mathbb{N}$ and estimate for $s > s_*$ as in (\ref{eq: inbetween})--(\ref{eq: gradient HjHj last-1})
  \begin{equation}
   \lm \left(\sum_j \left|P_{j-s} \del_1 (G_j H_j) \right|^2 \right)^{\frac12} \rm_{L^q} \leq C \varepsilon^{-\frac12} 2^{-s} \lm f \rm_{L^q}. \label{eq: del1 GjHj slarge}
  \end{equation}
  For $s \leq s_{*}$ we estimate, using that $|H_j| \leq 1$,
  \begin{align}\label{eq: productrule ssmall}
   &\lm \left(\sum_j \left|P_{j-s} \del_1 (G_j H_j) \right|^2 \right)^{\frac12} \rm_{L^q} \\ \leq &C \lm \left( \sum_j |\del_1 G_j |^2 \right)^{\frac12} \rm_{L^q} 
   +  C \lm \left( \sum_j | G_j \del_1 H_j |^2 \right)^{\frac12} \rm_{L^q}. \nonumber
  \end{align}
 Again, as $G_j$ is the convolution of $\tilde{F}_j$ with a Fej\'{e}r kernel, we may estimate the first term on the right hand side by
  \begin{align}
   \lm \left( \sum_j |\del_1 G_j |^2 \right)^{\frac12} \rm_{L^q} &\leq C \lm \left( \sum_j |\del_1 \tilde{F}_j |^2 \right)^{\frac12} \rm_{L^q}. \nonumber \\
\intertext{Using equation \eqref{eq: tildeF} and H\"older's inequality for the sum over $r$, we find }
   &\leq C \varepsilon^{-\frac12} \lm \left( \sum_j \sum_{r \leq 2 \lfloor \varepsilon^{-1} \rfloor - 1} \left| \sum_{n \in \Lambda_{j,r}^1} \frac{n_1 - a_{j,r}}{n_1} \hat{f}(n) e^{in \cdot x} \right|^2 \right)^{\frac12} \rm_{L^q}. \nonumber \\
   \intertext{Using the Rubio-de-Francia-inequality for arbitrary intervals in the first variable as in \eqref{eq: 20}--\eqref{eq: 21} yields  } \nonumber \\
   &\leq C \varepsilon^{-\frac12} \lm  \sum_j \sum_{r \leq 2 \lfloor \varepsilon^{-1} \rfloor - 1} \sum_{n \in \Lambda_{j,r}^1}\frac{n_1 - a_{j,r}}{n_1} \hat{f}(n) e^{in \cdot x} \rm_{L^q}. \nonumber \\
   \intertext{By the improvement of the Marcinkiewicz multiplier theorem due to Coifman, de Francia, and Semmes (see \cite{CoFrSe88}) the function $\frac{n_1-a_{j,r}}{n_1}$ defines a multiplier whose associated operator-norm from $L^q$ to $L^q$ can be estimated by $C_t \varepsilon^{\frac{t-1}{t}}$ for any $t$ such that $|\frac12 - \frac1q| < \frac1t$. In particular, there exists $t > 2$ such that 
} \nonumber  \\
   &\leq C \varepsilon^{\frac{t-1}t - \frac12} \lm \sum_{n \in \Lambda^1} \hat{f}(n) e^{in \cdot x} \rm_{L^q} =  C \varepsilon^{\frac{t-1}t - \frac12} \lm f \rm_{L^q}. \label{eq: delGj}
  \end{align}
  For the second term of the right hand side of \eqref{eq: productrule ssmall}, note that in \cite{BoBr03} the authors show that
  \begin{equation*}
  \lm \nabla H_j \rm_{L^{\infty}} \leq 2^j \lm f \rm_{L^2}.
  \end{equation*}
  Consequently,
  \begin{equation*}
   \lm \left( \sum_j | G_j \del_1 H_j |^2 \right)^{\frac12} \rm_{L^q} \leq \lm \left( \sum_j |2^j G_j|^2 \right)^{\frac12} \rm_{L^q} \lm f \rm_{L^2}. \label{eq: del1 Hj}
  \end{equation*}
  The right hand side can now be treated as in \eqref{eq: 2^jGj1}--\eqref{eq: gradient HjHj last-1} to obtain 
  \begin{equation}
   \lm \left( \sum_j | G_j \del_1 H_j |^2 \right)^{\frac12} \rm_{L^q} \leq C \varepsilon^{-\frac12} \lm f \rm_{L^q} \lm f \rm_{L^2}. \label{eq: delHj final}
  \end{equation}
  Combining \eqref{eq: del1 GjHj slarge}, \eqref{eq: delGj}, and \eqref{eq: delHj final} yields
  \begin{equation}
   \lm \del_1 \sum_j G_j H_j \rm_{L^q} \leq C 2^{-s_{*}} \varepsilon^{-\frac12} \lm f \rm_{L^q} + \sum_{-m \leq s \leq s_*} \left(\varepsilon^{-\frac12} \lm f \rm_{L^q} \lm f \rm_{L^2} + \varepsilon^{\frac{t-1}t - \frac12} \lm f \rm_{L^q}\right). \label{eq: delGjHj final}
  \end{equation}
  Eventually, choose $s_*$ such that $2^{-s_*} \sim \varepsilon$. 
  Then 
  \begin{equation}
   \lm \del_1 \sum_j G_j H_j \rm_{L^q} \leq C \log (\varepsilon^{-1}) \left( \varepsilon^{\frac{t-1}t - \frac12} \lm f \rm_{L^q} + \varepsilon^{-\frac12} \lm f \rm_{L^q} \lm f \rm_{L^2} \right).
  \end{equation}
Here, we used that $\varepsilon^{\frac12} \leq \varepsilon^{\frac{t-1}{t}-\frac12}$ for $\varepsilon < 1$.
Recall that $t>2$ and therefore $\log (\varepsilon^{-1}) \varepsilon^{\frac{t-1}t - \frac12} \to 0$ as $\varepsilon \to 0$.
  Comparing with \eqref{eq: L2 estimate} and \eqref{eq: YW1q} shows that for given $\delta>0$ the properties \ref{item: Y1}.--\ref{item: Y5}.~for $Y_1$ can be achieved for $\varepsilon>0$ small enough.
This finishes the proof.
\end{proof}

From the nonlinear estimate we can now derive a linear estimate.

\begin{lemma}[Linear estimate]\label{lemma: linear estimate}
 Let $2 < q < \infty$. 
 Then for every $\delta > 0$ there exists a constant $C_{\delta}>0$ such that for every function $f \in L^2(\Pi) \cap L^q(\Pi)$ satisfying $\int_{\Pi} f = 0$ there exists $F \in L^{\infty}(\Pi;\R^2) \cap H^1(\Pi;\R^2) \cap W^{1,q}(\Pi;\R^2)$ such that \\
 \begin{tabular}{ll}
 \parbox{5cm}{
\begin{enumerate}
  \item $\| F \|_{L^{\infty}} \leq C_{\delta} \| f \|_{L^2}$,
  \item $\| F \|_{H^1} \leq C_{\delta} \| f \|_{L^2}$,
  \item $\| \operatorname{div} F - f \|_{L^2} \leq \delta \| f \|_{L^2}$,
 \end{enumerate}
}
 &
 \parbox{5cm}{
 \begin{enumerate}\setcounter{enumi}{3}
  \item $\| F \|_{W^{1,q}} \leq C_{\delta} \| f \|_{L^q}$,
  \item $\| \operatorname{div} F - f \|_{L^q} \leq \delta \| f \|_{L^q}$.
 \end{enumerate}
}
\end{tabular} 
\end{lemma}

\begin{proof}
As we want to prove a linear estimate, we may assume without loss of generality that $\|f \|_{L^2} = \delta C_{\delta}^ {-1} < c$ where $C_{\delta}$ and $c>0$ are the constants from Lemma \ref{lemma: nonlinear approximation}. 
The application of Lemma \ref{lemma: nonlinear approximation} for $\delta > 0$ directly implies the claim for $\tilde{\delta} = 2 \delta$ and $C_{\tilde{\delta}} = \delta^{-1} C_{\delta}^2$.
\end{proof}
Armed with this approximation we are now able to prove Theorem \ref{theorem: iteration} by iteration.

\begin{proof}[Proof of Theorem \ref{theorem: iteration}]
Let $f \in L^2(\Pi) \cap L^q(\Pi)$ such that $\int_{\Pi} f = 0$. 
 We apply Lemma \ref{lemma: linear estimate} for $\delta = \frac12$. Hence, there exists $F_1$ such that \\
 \begin{tabular}{ll}
 \parbox{6cm}{
 \begin{itemize}
  \item $\| F_1 \|_{L^{\infty}} \leq C_{\frac12} \| f \|_{L^2}$,
  \item $\| F_1 \|_{H^1} \leq C_{\frac12} \| f \|_{L^2}$,
  \item $\| \operatorname{div} F_1 - f \|_{L^2} \leq \frac12 \| f \|_{L^2}$,
  \end{itemize}
 }
 &
 \parbox{6cm}{
\begin{itemize}
  \item $\| F_1 \|_{W^{1,q}} \leq C_{\frac12} \| f \|_{L^q}$,
  \item $\| \operatorname{div} F_1 - f \|_{L^q} \leq \frac12 \| f \|_{L^q}$.
 \end{itemize} 
 }
 \end{tabular}

We define $F_i$ for $i \geq 2$ inductively: let $\tilde{f}_i = f -  \operatorname{div} \sum_{j = 1}^{i-1} F_j$. Note that by the periodicity of the $F_j$ it holds $\int_{\Pi} \tilde{f}_i = 0$. The inductive application of Lemma \ref{lemma: linear estimate} for $\delta = \frac12$ and
$\tilde{f}_i$ provides the existence of $F_i$ such that
\begin{enumerate}
 \item $\| F_i \|_{L^{\infty}} \leq C_{\frac12} \| f - \operatorname{div} \sum_{j = 1}^{i-1} F_j \|_{L^2} \leq C_{\frac12} (\frac12)^{i-1} \| f \|_{L^2}$,
  \item $\| F_i \|_{H^1} \leq  C_{\frac12} \| f - \operatorname{div} \sum_{j = 1}^{i-1} F_j \|_{L^2} \leq C_{\frac12} (\frac12)^{i-1} \| f \|_{L^2}$,
  \item $\| \operatorname{div} F_i + \operatorname{div} \sum_{j=1}^{i-1} F_j - f \|_{L^2} \leq \frac12 \| \operatorname{div} \sum_{j=1}^{i-1} F_j - f \|_{L^2} \leq (\frac12)^i \| f\|_{L^2} $,
  \item $\| F_i \|_{W^{1,q}} \leq C_{\frac12} \| f - \operatorname{div} \sum_{j = 1}^{i-1} F_j \|_{L^q} \leq C_{\frac12} (\frac12)^{i-1} \| f \|_{L^q}$,
  \item $\| \operatorname{div} F_i + \operatorname{div} \sum_{j=1}^{i-1} F_j - f \|_{L^q} \leq \frac12 \| \operatorname{div} \sum_{j=1}^{i-1} F_j - f \|_{L^q} \leq (\frac12)^i \| f\|_{L^q}$.
\end{enumerate}
Define $F = \sum_{j=1}^{\infty} F_j$. Then, $\operatorname{div} F = f$ and the claimed estimates follow by the triangle inequality with $C = 2 C_{\frac12}$.
\end{proof}
\subsection{Lipschitz Domains}

In \cite{BoBr03}, the authors prove that for $f \in L^2(\Pi)$ such that $\int_{\Pi} f \,dx = 0$ there exists $Y \in H^1(\Pi)$ satisfying $\operatorname{div } Y = f$ and $\| Y \|_{L^{\infty}} + \| Y \|_{H^1} \leq C \| f \|_{L^2}$.
Moreover, in \cite[Section 7]{BoBr03} the authors present an argumentation to transfer this result to Lipschitz domains.
One can adopt their strategy with minor changes to find

\begin{theorem}\label{theorem: lipschitz}
 Let $2 < q < \infty$ and $\Omega \subseteq \R^2$ open, bounded with Lipschitz boundary. 
 Then there exists a constant $C > 0$ such that for every $f \in L^2(\Omega) \cap L^q(\Omega)$ satisfying $\int_{\Omega} f \,dx = 0$ there exists $Y \in L^{\infty}(\Omega;\R^2) \cap H_0^1(\Omega;\R^2) \cap W^{1,q}_0(\Omega;\R^2)$ such that $  \operatorname{div }Y = f$,
 \begin{equation*}
\| Y \|_{L^{\infty}(\Omega;\R^2)} + \| Y \|_{H_0^1(\Omega;\R^2)} \leq C \| f \|_{L^2(\Omega)}, \text{ and } \| Y \|_{W_0^{1,q}(\Omega;\R^2)} \leq C \| f \|_{L^q(\Omega)}.
 \end{equation*}
\end{theorem}

We use this result to prove the primal version of Theorem \ref{thm: bourgainbrezis}, namely we 
establish a decomposition result for functions in $H^1_0 \cap W^{1,q}_0$ into a bounded part and a gradient. 

  \begin{theorem}[Primal result]\label{theorem: dual}
 Let $2 < q < \infty$ and $\Omega \subseteq \R^2$ open, simply connected, bounded with Lipschitz boundary.
Then there exists a constant $C>0$ such that for every $\varphi \in H^1_0(\Omega;\R^2) \cap W^{1,q}_0(\Omega;\R^2)$ there exist $h \in H^2_0(\Omega) \cap W^{2,q}_0(\Omega)$ and $g \in L^{\infty}(\Omega;\R^2) \cap H^1_0(\Omega;\R^2) \cap W^{1,q}_0(\Omega;\R^2)$ satisfying
 \begin{enumerate}
  \item $\varphi = g + \nabla h$,
  \item $\| g \|_{L^{\infty}(\Omega;\R^2)} + \| g \|_{H_0^1(\Omega;\R^2)} + \| h \|_{H_0^2(\Omega)} \leq C \| \varphi \|_{H^1_0(\Omega;\R^2)}$,
  \item $\| g \|_{W_0^{1,q}(\Omega;\R^2)} + \| h \|_{W^{2,q}_0(\Omega)} \leq C \| \varphi \|_{W^{1,q}_0(\Omega;\R^2)}$.
 \end{enumerate}
\end{theorem}

\begin{proof}
Let $\varphi \in H_0^1(\Omega;\R^2) \cap W_0^{1,q}(\Omega)$. 
Then $\int_{\Omega} \operatorname{curl }\varphi \, dx = 0$.
The application of Theorem \ref{theorem: lipschitz} to $\operatorname{curl }\varphi \in L^2(\Omega) \cap L^q(\Omega)$ provides a function $Y \in L^{\infty}(\Omega;\R^2) \cap H_0^1(\Omega;\R^2) \cap W^{1,q}_0(\Omega;\R^2)$ such that $\operatorname{div }Y = \operatorname{curl }\varphi$
and
 \begin{align*}
  &\| Y \|_{L^{\infty}(\Omega;\R^2)} + \| Y \|_{H_0^1(\Omega;\R^2)} \leq C \| \operatorname{curl }\varphi \|_{L^2(\Omega)} \leq C \| \varphi \|_{H^1(\Omega;\R^2)}, \\ \text{ and } &\| Y \|_{W_0^{1,q}(\Omega;\R^2)} \leq C \| \varphi \|_{W^{1,q}(\Omega;\R^2)}.
  \end{align*}
  Set $g = Y^{\perp} = (-Y_2,Y_1)$.
  Then $g$ satisfies the same bounds as $Y$ and $\operatorname{curl } g = \operatorname{div }Y = \operatorname{curl }\varphi$.
 As $\Omega$ is simply-connected, by the Hodge decomposition there exists a vector field $h \in H^2(\Omega) \cap W^{2,q}(\Omega)$ such that $\varphi - g = \nabla h$,
  \begin{equation*}
   \| h \|_{H^2(\Omega)} \leq C \| g - \varphi \|_{H^1_0(\Omega;\R^2)} \leq C \| \varphi \|_{H^1(\Omega;\R^2)}, \text{ and } \| h \|_{W^{2,q}(\Omega)} \leq  C \| \varphi \|_{W_0^{1,q}(\Omega;\R^2)}.
  \end{equation*}
  Moreover, $\nabla h = \varphi - g = 0$ on $\del \Omega$. 
  Therefore, $h$ is constant on the boundary of $\Omega$ and we may assume it is zero. 
  Hence, $h \in H_0^2(\Omega) \cap W^{2,q}_0(\Omega)$.
\end{proof} 
\begin{remark}
From Theorem \ref{theorem: dual} we derive the corresponding dual statement, i.e., a function $f \in L^1(\Omega;\R^2)$ satisfying $\operatorname{ div } f = a + b \in H^{-2}(\Omega) + W^{-2,p}(\Omega)$, $1< p < 2$, is 
an element of the space $H^{-1}(\Omega;\R^2) + W^{-1,p}(\Omega;\R^2)$ and 
\begin{equation*} \| f \|_{H^{-1}(\Omega;\R^2) + W^{-1,p}(\Omega;\R^2)} \leq C \left( \left\| f \right\|_{L^1\left(\Omega;\R^2\right)} + \left\| a \right\|_{H^{-2}\left(\Omega\right)} 
  + \left\| b \right\|_{W^{-2,p}\left(\Omega\right)}\right).
\end{equation*}
Indeed, let $\varphi \in H_0^1(\Omega;\R^2) \cap W_0^{1,p'}(\Omega;\R^2)$ where $\frac1p + \frac1{p'} = 1$. 
We use the decomposition $\varphi = g + \nabla h$ from Theorem \ref{theorem: dual} and estimate
\begin{align*}
&\int_{\Omega} f \varphi \,dx = \int_{\Omega} f (g + \nabla h) \,dx \\
\leq &C\left( \lm f \rm_{L^1(\Omega;\R^2)} \lm g \rm_{L^{\infty}(\Omega;\R^2)} + \lm a \rm_{H^{-2}(\Omega)} \lm h \rm_{H^2_0(\Omega)} + \lm b \rm_{W^{-2,p}(\Omega)} \lm h \rm_{W^{2,p'}_0(\Omega)}\right) \\
\leq &C \left(\lm f \rm_{L^1(\Omega;\R^2)} + \lm a \rm_{H^{-2}(\Omega)} + \lm b \rm_{W^{-2,p'}(\Omega)}\right) \max \left(\lm \varphi \rm_{H_0^1(\Omega;\R^2)},\lm \varphi \rm_{W^{1,p}_0(\Omega;\R^2)} \right).
\end{align*} 
In particular, $f \in (H^1_0(\Omega;\R^2) \cap W^{1,p'}_0(\Omega;\R^2))' = H^{-1}(\Omega;\R^2) + W^{-1,p}(\Omega;\R^2)$.
Hence, it can be written as $f = A + B \in  H^{-1}(\Omega;\R^2) + W^{-1,p}(\Omega;\R^2)$.
The difference to the Theorem \ref{thm: bourgainbrezis} is that $A$ and $B$ only satisfy a combined estimate, precisely
\begin{equation}\label{eq: remarkestimate}
\lm A \rm_{H^{-1}(\Omega;\R^2)} + \lm B \rm_{W^{-1,p}(\Omega;\R^2)} \leq C ( \lm f \rm_{L^1(\Omega;\R^2)} + \lm a \rm_{H^{-2}(\Omega)} + \lm b \rm_{W^{-2,p}(\Omega)} ).
\end{equation}
\end{remark}

We will use a scaling argument to separate the combined estimate \eqref{eq: remarkestimate}. \\
 The classical $W^{k,p}$-norm and the homogeneous $W_0^{k,p}$-norm are equivalent norms on the space $W_0^{k,p}$.
So far, it has not been important which of these norms we use on $W_0^{k,p}$.
Next, we are interested in the scaling of the optimal constant in Theorem \ref{theorem: dual}.
With respect to the homogeneous $W_0^{k,p}$-norms, i.e., $\| f \|_{W^{k,p}_0(\Omega,\R^m)} = \sum_{|\alpha| = k} \lm D^{\alpha} f \rm_{L^p(\Omega;\R^m)}$, this constant is scaling invariant.

 \begin{proposition} \label{prop: scaling}
 Let $2 < q < \infty$ and $\Omega \subseteq \R^2$ open, simply connected, bounded with Lipschitz boundary.
  Let $R > 0$ and $\Omega_R = R \cdot \Omega$. 
  If we denote by $C(\Omega)$ and $C(\Omega_R)$ the optimal constant of Theorem \ref{theorem: dual} for the domain $\Omega$ and $\Omega_R$, respectively, then $C(\Omega) = C(\Omega_R)$.
 \end{proposition}
 
\begin{proof}
The proof follows easily by scaling.
\end{proof}
Using Proposition \ref{prop: scaling}, we can now prove Theorem \ref{thm: bourgainbrezis}.

\begin{theorem*}[Bourgain-Br\'{e}zis type estimate]
Let $1 < p < 2$ and $\Omega \subseteq \R^2$ open, simply connected, and bounded with Lipschitz boundary.
 Then there exists a constant $C>0$ such that for every $f \in L^1(\Omega;\R^2)$ satisfying $\operatorname{ div } f = a + b \in H^{-2}(\Omega) + W^{-2,p}(\Omega)$ there exist $A \in H^{-1}(\Omega;\R^2)$ and  $B \in W^{-1,p}(\Omega;\R^2)$ such that 
 \begin{enumerate}
  \item $f = A + B$,
  \item $\| A \|_{H^{-1}(\Omega;\R^2)} \leq C ( \| f \|_{L^1(\Omega;\R^2)} + \| a \|_{H^{-2}(\Omega)})$,
  \item $\| B \|_{W^{-1,p}(\Omega;\R^2)} \leq C \| b \|_{W^{-2,p}(\Omega)}$.
 \end{enumerate}
\end{theorem*}
\begin{proof}
 Let $f \in L^1(\Omega;\R^2)$, $R>0$, and $\Omega_R = R \cdot \Omega$. \\ 
 Define the function $f_R: \Omega_R \to \R^2$ by $f_R(x) = f\left(\frac{x}{R}\right)$ for $x \in \Omega_R$ . 
 Now, consider a test function $\varphi \in H_0^1\left(\Omega_R;\R^2\right) \cap W_0^{1,p'}\left(\Omega_R;\R^2\right)$ where $\frac1p + \frac1{p'} = 1$.
 By Theorem \ref{theorem: dual}, there exist functions $h \in H^2_0\left(\Omega_R\right) \cap W^{2,p'}_0\left(\Omega_R\right)$ and 
 $g \in L^{\infty}\left(\Omega_R;\R^2\right) \cap H^1_0\left(\Omega_R;\R^2\right) \cap W^{1,p'}_0\left(\Omega_R;\R^2\right)$ such that
  $\varphi = g + \nabla h$, 
\begin{align*} &\| g \|_{L^{\infty}\left(\Omega_R;\R^2\right)} + \| g \|_{H_0^1\left(\Omega_R;\R^2\right)} + \| h \|_{H_0^2\left(\Omega_R\right)} \leq C \| \varphi \|_{H_0^1\left(\Omega_R;\R^2\right)}, \\
 \text{and } &\| g \|_{W_0^{1,p'}\left(\Omega_R;\R^2\right)} + \| h \|_{W_0^{2,p'}\left(\Omega_R\right)} \leq C \| \varphi \|_{W_0^{1,p'}(\Omega_R;\R^2)}.
\end{align*}
  Note that by Proposition \ref{prop: scaling} the constant $C$ does not depend on $R$.
  Next, observe that
  \begin{align}\label{eq: expansion}
   &\int_{\Omega_R} f_R \cdot  \varphi \, dx   = \int_{\Omega_R} f_R \cdot (g + \nabla h) \,dx  \\
 = &< f_R, g>_{L^1(\Omega_R;\R^2),L^{\infty}(\Omega_R;\R^2)} - < a_R, h >_{H^{-2}(\Omega_R),H^2_0(\Omega_R)} \nonumber \\& - < b_R, h >_{W^{-2,p}(\Omega_R),W^{2,p'}_0(\Omega_R)}, \nonumber
  \end{align}
 where we define the distributions $a_R$ and $b_R$ by 
 \begin{align*} <a_R,h>_{H^{-2}(\Omega_R),H^2_0(\Omega_R)} &= R <a,h_{R^{-1}}>_{H^{-2}(\Omega),H^2_0(\Omega)} \\ \text{ and } <b_R,h>_{W^{-2,p}(\Omega_R),W^{2,p'}_0(\Omega_R)} &= R <b,h_{R^{-1}}>_{W^{-2,p}(\Omega),W^{2,p'}_0(\Omega)},
 \end{align*}
  respectively, for $h_{R^{-1}}(x) = h(Rx)$.
  By scaling we have
  \begin{equation} \label{eq: scaling negative}
   \left\| a_R \right\|_{H^{-2}\left(\Omega_R\right)} = R^2 \left\| a \right\|_{H^{-2}(\Omega)}, \text{ and } \left\| b_R \right\|_{W^{-2,p}\left(\Omega_R\right)} = R^{1 + \frac2p} \left\| b \right\|_{W^{-2,p}(\Omega)}.
  \end{equation}
  Moreover, from \eqref{eq: expansion} we derive that
  \begin{align*}
  &\left| \int_{\Omega_R} f_R \cdot \varphi \,dx \right| \\ \leq &C \left( \left\| f_R \right\|_{L^1\left(\Omega_R;\R^2\right)} + \left\| a_R \right\|_{H^{-2}\left(\Omega_R\right)} 
  + \left\| b_R \right\|_{W^{-2,p}\left(\Omega_R\right)}\right) \\ 
&\max \left(\left\| \varphi \right\|_{H^{1}\left(\Omega_R;\R^2\right) }, \left\| \varphi \right\|_{W^{1,p'}\left(\Omega_R;\R^2\right)} \right).
  \end{align*}
  The dual space of $H_0^1(\Omega_R;\R^2) \cap W^{1,p'}_0(\Omega_R;\R^2)$ equipped with the norm 
\begin{equation*}
\lm \varphi \rm_{H_0^1(\Omega_R;\R^2) \cap W^{1,p'}_0(\Omega_R;\R^2)} = \max \left( \lm \varphi \rm_{H_0^1(\Omega_R;\R^2)},\lm \varphi \rm_{W^{1,p'}_0(\Omega_R;\R^2)}\right)
\end{equation*}
 is isomorphic to the space $H^{-1}(\Omega_R;\R^2) + W^{-1,p}(\Omega_R;\R^2)$ endowed with the norm 
\begin{equation*}
\lm F \rm_{H^{-1}(\Omega_R;\R^2) + W^{-1,p}(\Omega_R;\R^2)} = \inf \{ \lm F_1 \rm_{H^{-1}(\Omega_R;\R^2)} + \lm F_2 \rm_{W^{-1,p}(\Omega_R;\R^2)}: F_1 + F_2 = F \}.
\end{equation*}
Hence,  $f_R \in H^{-1}\left(\Omega_R;\R^2\right) + W^{-1,p}\left(\Omega_R;\R^2\right)$ and
  \begin{equation*}
   \left\| f_R \right\|_{ H^{-1}\left(\Omega_R;\R^2\right) + W^{-1,p}\left(\Omega_R;\R^2\right)} \leq C 
   \left( \left\| f_R \right\|_{L^1\left(\Omega_R;\R^2\right)} + \left\| a_R \right\|_{H^{-2}\left(\Omega_R\right)} + \left\| b_R \right\|_{W^{-2,p}\left(\Omega_R\right)}\right).
  \end{equation*}
  In particular, there exist $A_R \in H^{-1}\left(\Omega_R;\R^2\right), B_R \in W^{-1,p}\left(\Omega_R;\R^2\right)$ such that $f_R = A_R + B_R$ and
  \begin{equation}\label{eq: estimate sumspace}
   \left\| A_R \right\|_{H^{-1}\left(\Omega_R\right)} + \left\| B_R \right\|_{W^{-1,p}\left(\Omega_R\right)} \leq C 
   \left( \left\| f_R \right\|_{L^1\left(\Omega_R\right)} + \left\| a_R \right\|_{H^{-2}\left(\Omega_R\right)} + \left\| b_R \right\|_{W^{-2,p}\left(\Omega_R\right)}\right).
  \end{equation}
 We define $A \in H^{-1}(\Omega;\R^2)$ and $B \in W^{-1,p}(\Omega;\R^2)$ by  
  \begin{align*}
   &< A, \varphi >_{H^{-1}(\Omega;\R^2),H_0^1(\Omega;\R^2)} = R^{-2} < A_R, \varphi_R >_{H^{-1}(\Omega_R;\R^2),H_0^1(\Omega_R;\R^2)} \\ \text{ and } &<B, \varphi>_{W^{-1,p}(\Omega;\R^2),W_0^{1,p'}(\Omega;\R^2)} = R^{-2} <B_R, \varphi_R>_{W^{-1,p}(\Omega_R;\R^2),W_0^{1,p'}(\Omega_R;\R^2)} 
  \end{align*}
  for every $\varphi \in C_c^{\infty}(\Omega)$ and $\varphi_R \in C_c^{\infty}(\Omega_R)$ given by $\varphi_{R}(x) = \varphi\left(\frac{x}{R}\right)$.
  It follows for every $\varphi \in C^{\infty}_c(\Omega)$ 
\begin{align*}
\int_{\Omega} f \cdot \varphi \, dx &= R^{-2} \int_{\Omega_R} f_R \cdot \varphi_R \, dx \\ &= < A, \varphi >_{H^{-1}(\Omega;\R^2),H_0^1(\Omega;\R^2)} + <B, \varphi>_{W^{-1,p}(\Omega;\R^2),W_0^{1,p'}(\Omega;\R^2)}.
\end{align*}
Consequently, $f= A + B$.
  Moreover, by \eqref{eq: scaling negative} and \eqref{eq: estimate sumspace} we see that
  \begin{align*}
   \left\| A \right\|_{H^{-1}(\Omega;\R^2)} &= R^{-2} \left\| A_R \right\|_{H^{-1}\left(\Omega_R;\R^2\right)} \\ &\leq C 
   \left( \left\| f \right\|_{L^1\left(\Omega;\R^2\right)} + \left\| a \right\|_{H^{-2}\left(\Omega\right)} + R^{\frac2p - 1} \left\| b \right\|_{W^{-2,p}\left(\Omega\right)}\right) \\ 
   \left\| B \right\|_{W^{-1,p}(\Omega;\R^2)}& = R^{-1-\frac2p} \left\| B_R \right\|_{W^{-1,p}\left(\Omega_R;\R^2\right)} \\ &\leq C 
   \left( R^{1-\frac2p} \left(\left\| f \right\|_{L^1\left(\Omega;\R^2\right)} + \left\| a \right\|_{H^{-2}\left(\Omega\right)}\right) + \left\| b \right\|_{W^{-2,p}\left(\Omega\right)}\right).
  \end{align*}
  Choosing $R$ such that $R^{1-\frac2p} = \frac{\left\| b \right\|_{W^{-2,p}\left(\Omega\right)}}{\left\| f \right\|_{L^1\left(\Omega;\R^2\right)} + \left\| a \right\|_{H^{-2}\left(\Omega\right)}}$
  finishes the proof.
\end{proof}
\begin{remark}
Let us remark here that by the same argumentation this result also holds for Radon measures.
\end{remark}

\section{A Generalized Rigidity Estimate with Mixed Growth}\label{sec: generalizedrgidity}

The goal of this section is to prove a rigidity estimate for fields with non-vanishing $\operatorname{curl}$ in the case of a nonlinear energy density with mixed growth. 
Precisely, we show the following theorem.
\begin{theorem*}
Let $1 < p<2$ and $\Omega \subseteq \R^2$ open, simply connected, bounded with Lipschitz boundary. 
 There exists a constant $C>0$ such that for every $\beta \in L^p(\Omega;\R^{2 \times 2})$ with $\operatorname{curl} \beta \in \mathcal{M}(\Omega;\R^2)$ there exists a rotation $R \in SO(2)$ such that
 \begin{equation*}
  \int_{\Omega} | \beta - R|^2 \wedge |\beta - R|^p \, dx \leq C \left( \int_{\Omega} \operatorname{dist}(\beta, SO(2))^2 \wedge \operatorname{dist}(\beta, SO(2))^p dx + | \operatorname{curl} \beta|(\Omega)^2 \right).
 \end{equation*}
\end{theorem*}

We start by observing the following easy triangle-inequality.
\begin{lemma}\label{lemma: triangle}
 Let $m \in \mathbb{N}$ and $1 < p< 2$. There exists a constant $C > 0$ such that for all $a,b \in \R^{m}$ it holds 
 \begin{equation*}
  \left|a + b\right|^2 \wedge \left|a + b \right|^p \leq C \left( |a|^2 \wedge |a|^p + |b|^2 \wedge |b|^p \right).
 \end{equation*}
\end{lemma}
\begin{proof}
The result can easily be obtained by distinguishing the cases $|a+b| \leq 1, |a+b| >1$.
\end{proof}

 Next, we prove a simple decomposition result which we need in the proof of a preliminary weak-type rigidity estimate, Proposition \ref{prop weaktyperigidity}.

\begin{lemma} \label{lemma: decompose}
 Let $U \subseteq \R^n$ and $1 < p < 2$. 
Then for every $k>0$ there exists a constant $C(k)>0$ such that for every two nonnegative functions $f \in L^{2,\infty}(U), \, g \in L^p(U)$ there exist functions $\tilde f \in L^{2,\infty}(U)$ and $\tilde g \in L^p(U)$ such that $f + g = \tilde f + \tilde g$, $\tilde g \in \{0\} \cup (k,\infty]$, $\tilde f \leq k$, and 
 \begin{equation}
 \lm \tilde f \rm_{L^{2,\infty}(U)}^2 + \lm \tilde g \rm_{L^p(U)}^p \leq C(k) \left(\lm f \rm_{L^{2,\infty}(U)}^2 + \lm g \rm_{L^p(U)}^p\right).
 \end{equation}
\end{lemma}

\begin{proof}
 Let $k > 0$. 
Define $\tilde f = ( f+ g ) \1_{\{f+g\leq k\}}$ and $\tilde g = (f+g) \1_{\{f + g > k\}}$. 
 Then the first two properties are clearly satisfied. 
 Moreover, we can estimate
 \begin{align*} 
 \lm \tilde f \rm_{L^{2,\infty}(U)}^2 &\leq 4 \lm f \rm_{L^{2,\infty}(U)}^2 + 4 \lm g \1_{\{g \leq k\}} \rm_{L^2(U)}^2 
\\ &\leq 4 \lm f \rm_{L^{2,\infty}(U)}^2 + 4 k^{2-p} \lm g \1_{\{g \leq k\}} \rm_{L^p(U)}^p 
\\ &\leq C(k) \left(\lm f \rm_{L^{2,\infty}(U)}^2 + \lm g \rm_{L^p(U)}^p\right).
 \end{align*}
For $\tilde g$, notice that $f \1_{\{f + g > k\}} = f \1_{\{f + g > k\}} \1_{\{ f \leq \frac k2\}} + f \1_{\{f + g > k\}} \1_{\{ f > \frac k2\}} \leq g + f \1_{\{f > \frac k2\}}$.
Thus, we can conclude that $\lm \tilde g \rm_{L^p(U)}^p \leq C \left(\lm f \1_{\{f > \frac k2\}} \rm_{L^p(U)}^p + \lm g \rm_{L^p(U)}^p\right)$.
In addition,
\begin{align}
 \lm \1_{\{f > \frac k2\}} f \rm_{L^p(U)}^p = &\int_0^{\infty} p t^{p-1} \mathcal{L}^n(\{ \1_{\{f > \frac k2\}} f > t \}) \, dt \label{eq: weak2strongp}\\
 =&\int_0^{\frac k2} p t^{p-1} \mathcal{L}^n\left(\left\{ \1_{\{f > \frac k2\}} f > t \right\}\right) \, dt \\ &+ \int_{\frac k2}^{\infty} p t^{p-1} \mathcal{L}^2\left(\left\{ \1_{\{f > \frac k2\}} f > t \right\}\right) \, dt \nonumber \\
 \leq &\left(\frac{k}2\right)^p \mathcal{L}^n\left(\left\{ f > \frac k2 \right\}\right) + \int_{\frac k2}^{\infty} p t^{p-3} \lm f \rm_{L^{2,\infty}(U)}^2 \, dt \nonumber\\
 \leq &\left(\frac{k}2\right)^p \mathcal{L}^n\left(\left\{ f > \frac k2 \right\}\right) + \frac2{2-p} \left( \frac k2\right)^{p-2} \lm f \rm_{L^{2,\infty}(U)}^2 \nonumber\\
 \leq &C(k) \lm f \rm_{L^{2,\infty}(U)}^2. \nonumber
\end{align}
\end{proof}

 As a second ingredient for the proof of the preliminary mixed-growth rigidity result we recall the following truncation argument from \cite[Proposition A.1]{FrJaMu02}.

\begin{proposition} \label{prop: cutoff}
 Let $U \subseteq \R^n$ be a bounded Lipschitz domain and $m \geq 1$. Then there is a constant $c_1 = c_1(U)$ such that for all $u \in W^{1,1}\left(U,\R^m \right)$ and all $\lambda > 0$ there exists 
 a measurable set $E \subseteq U$ such that
 \begin{enumerate}
  \item $u$ is $c_1\lambda$-Lipschitz on $E$,
  \item $\mathcal{L}^n(U \setminus E) \leq \frac{c_1}{\lambda} \int_{\{|\nabla u| > \lambda\}} |\nabla u| \, dx$ \label{item: cutoff2}.
 \end{enumerate}
\end{proposition}

\noindent With the use of this result, we are now able to prove the preliminary weak-type rigidity estimate.
In \cite{CoDoMu14}, the authors prove rigidity estimates for fields whose distance to $SO(n)$ is either the sum of an $L^p$- and an $L^q$-function, or in a weak space $L^{p,\infty}$.
Our result is a combination of these results.

\begin{proposition} \label{prop weaktyperigidity}
 Let $1<p < 2$, $n\geq 2$, and $U \subseteq \R^n$ open and bounded with Lipschitz boundary.
 Let $u \in W^{1,1}(U;\R^{n \times n})$ such that there exist $f \in L^{2,\infty}(U)$ and $g \in L^p(U)$ satisfying
 \begin{equation*}
  \operatorname{dist}(\nabla u,SO(n)) = f + g.
 \end{equation*}
Then there exist  matrix fields $F \in L^{2,\infty}(U;\R^{n\times n})$ and $G \in L^p(U;\R^{n\times n})$ and a rotation $R \in SO(n)$ such that
\begin{equation*}
 \nabla u = R + G + F
\end{equation*}
and 
\begin{equation*}
 \lm F \rm_{L^{2,\infty}(U;\R^{n\times n})}^2 + \lm G \rm_{L^p(U;\R^{n\times n})}^p \leq C (\| f \|_{L^{2,\infty}(U)}^2 + \| g \|_{L^p(U)}^p).
\end{equation*}
The constant $C$ does not depend on $u,f,g$.
\end{proposition}

\begin{proof}
 Without loss of generality we may assume that $f$ and $g$ are nonnegative. 
 According to Lemma \ref{lemma: decompose}, we may also assume that $f \leq k$ and $g \in \{0\} \cup (k,\infty)$ where $k$ will be fixed later.  \\
First, we apply Proposition \ref{prop: cutoff} for $\lambda = 2n$ to obtain a measurable set $E \subseteq U$ such that $u$ is Lipschitz continuous on $E$ with Lipschitz constant $M = 2c_1n$.
Let $u_M$ be a Lipschitz continuous extension of $u_{|E}$ to $U$ with the same Lipschitz constant. 
In particular, $u_M = u$ on $E$.
Set $k = 2M$. 
Then we obtain
\begin{equation}\label{eq: distum}
 \operatorname{dist}( \nabla  u_M, SO(2)) \leq f + 2 M \1_{U \setminus E}.
\end{equation}
Indeed, notice that (we may assume that $c_1 \geq 1$)
\begin{equation}\label{eq: duM small}
 \operatorname{dist}(\nabla u_M, SO(2)) \leq 2c_1n + \sqrt{n} \leq 2M.
\end{equation}
Hence, we derive $\operatorname{dist}(\nabla u_M, SO(2)) \leq 2M$ on $U \setminus E$.
On $E$, we obtain that 
\begin{equation*}
 \operatorname{dist}(\nabla u_M, SO(2)) = \operatorname{dist}(\nabla u, SO(2)) = f + g.
\end{equation*}
As $g \in \{0\} \cup (2M,\infty]$, in view of equation \eqref{eq: duM small}, we find $\operatorname{dist}(\nabla u_M, SO(2)) = f$ on $E$. 
This shows \eqref{eq: distum}.\\
By applying the $L^{2,\infty}$ rigidity estimate from \cite[Corollary 4.1]{CoDoMu14}, we find a rotation $R \in SO(2)$ such that
\begin{align}
 \lm \nabla  u_M - R \rm_{L^{2,\infty}(U;\R^{n \times n})}^2 &\leq C \lm \operatorname{dist}(\nabla  u_M, SO(2)) \rm_{L^{2,\infty}(U)}^2 \nonumber \\
 &\leq 4 C \lm f \rm_{L^{2,\infty}(U)}^2 + 16 C M^2 \lm \1_{U \setminus E} \rm_{L^{2,\infty}(U)}^2. \label{eq: weak estimate}
\end{align}
Next, note that if $|\nabla u| > 2n$, then 
\begin{equation} \label{eq: duestimate}
 |\nabla u| \leq \sqrt{n} + \operatorname{dist}(\nabla u,SO(n)) \leq 2 \operatorname{dist}(\nabla u,SO(n)) =2 (f + g) \leq 4 \max \{f,g\}.
\end{equation}
Using propert \ref{item: cutoff2}.~from Proposition \ref{prop: cutoff} and \eqref{eq: duestimate}, we estimate
\begin{align*}
 \mathcal{L}^n(U \setminus E) &\leq \frac {c_1} {2n} \int_{\{|\nabla u| > 2n \}} |\nabla u| \, dx \\
 &\leq \frac {c_1} {2n} \int_{\{4f \geq n\}} 4f \,dx + \frac{c_1}{2n} \int_{\{4g \geq n\}} 4g \,dx \\
 &\leq \frac {c_1} {2 n^p} \int_{\{4f \geq n\}} 4^pf^p \,dx + \frac{c_1}{2 n^p} \int_{\{4g \geq n\}} 4^p g^p \,dx \\
 &\leq C \left( \lm f \1_{\{4f \geq n\}}\rm_{L^p(U)}^p + \lm g \rm_{L^p(U)}^p\right) \\
 &\leq C \left( \lm f \rm_{L^{2,\infty}(U)}^2 + \lm g \rm_{L^p(U)}^p\right), 
\end{align*}
where we used a similar estimate as in \eqref{eq: weak2strongp} for the last inequality.
In particular, it follows from \eqref{eq: weak estimate} that
\begin{equation*}
 \lm \nabla u_M - R \rm_{L^{2,\infty}(U;\R^{n \times n})}^2 \leq C (\| f \|_{L^{2,\infty}(U)}^2 + \| g \|_{L^p(U)}^p).
\end{equation*}
Hence, writing $\nabla u - R = \nabla u - \nabla u_M + \nabla u_M - R$, it remains to control $\nabla u - \nabla u_M$.
Clearly, we only have to consider $\nabla u - \nabla u_M$ on $U \setminus E$. 
On $U \setminus E$, it holds the pointwise estimate
\begin{equation*}
 |\nabla u - \nabla u_M| \leq |\nabla u| + 2c_1n \leq \operatorname{dist}(\nabla u,SO(2)) + 2M \1_{U \setminus E} = f + g + 2M \1_{U \setminus E}.
\end{equation*}
As before, we have that $\lm \1_{U \setminus E} \rm_{L^{2,\infty}(U)}^2 \leq C (\| f \|_{L^{2,\infty}(U)}^2 + \| g \|_{L^p(U)}^p)$.
This finishes the proof.
\end{proof}

 Armed with this weak-type rigidity estimate we are now able to prove Theorem \ref{theorem: generalizedrigidity}.
The proof is similar to the one of the corresponding statement in \cite[Theorem 3.3]{MuScZe14} but uses  quantities with mixed growth instead of quantities in $L^2$, in particular the Bourgain-Br\'ezis type estimate for mixed growth, Theorem  \ref{thm: bourgainbrezis}.

\begin{proof}[Proof of Theorem \ref{theorem: generalizedrigidity}]
 First we define \[
 \delta = \int_{\Omega} \operatorname{dist}(\beta, SO(2))^2 \wedge \operatorname{dist}(\beta, SO(2))^p dx 
 + | \operatorname{curl} \beta|(\Omega)^2.\] 
 As $1<p<2$, the embedding $\mathcal{M}(\Omega;\R^2) \hookrightarrow W^{-1,p}(\Omega;\R^2)$ is bounded.  
Hence, there exists a unique solution $v$ to the problem
 \begin{equation} \label{eq: laplace}
 \begin{cases}
 \Delta v = \operatorname{curl} \beta, \\
 v \in W^{1,p}_0(\Omega;\R^{2}).
 \end{cases}
 \end{equation}
 Define $\tilde \beta =\nabla v J$ where
 \begin{equation*}
 J = 
 \begin{pmatrix}
  0 & -1 \\ 1 & 0
 \end{pmatrix}.
 \end{equation*} 
The optimal regularity estimate for elliptic equations with measure valued right hand side yields (see, for example, \cite{DoHuMu00})
 \begin{equation} \label{eq: estimatebeta}
  \lm \tilde \beta \rm_{L^{2,\infty}(U;\R^{2\times 2})} \leq C \left| \operatorname{curl} \beta \right|(\Omega).
 \end{equation}
 In addition, we have that $\operatorname{curl} \tilde \beta = \operatorname{curl} \beta $. 
Hence, there exists a function $u \in W^{1,p}(\Omega;\R^2)$ such that $\nabla u = \beta - \tilde \beta$. 
Clearly,
\begin{equation}\label{eq: traingle}
 \left| \operatorname{dist}(\nabla u, SO(2)) \right| \leq |\tilde \beta| + \left| \operatorname{dist}(\beta,SO(2)) \right|.
\end{equation}
Moreover, observe that  for $f_1 = \operatorname{dist}(\beta, SO(2)) \1_{\{|\operatorname{dist}(\beta, SO(2))|\leq 1\}}$ and $f_2 = \operatorname{dist}(\beta, SO(2)) \1_{\{|\operatorname{dist}(\beta, SO(2))|> 1\}}$ we have
\begin{align*}
\operatorname{dist}(\beta, SO(2)) = f_1 + f_2
 \end{align*}
 and
$\lm f_1 \rm_{L^2(\Omega)}^2 \leq \delta$ and $\lm f_2 \rm_{L^p(\Omega)}^p \leq \delta$.
Combining this decomposition with \eqref{eq: estimatebeta}, and \eqref{eq: traingle} proves the existence of functions $g_1 \in L^{2,\infty}(\Omega)$ and $g_2 \in L^p(\Omega)$ such that $\operatorname{dist}(\nabla u, SO(2)) = g_1 + g_2$,
 \begin{align*}
\lm g_1 \rm_{L^{2,\infty}(\Omega)}^2 \leq 4 \lm \tilde \beta \rm_{L^{2,\infty}(\Omega)}^2 + 4 \lm f_1 \rm_{L^{2,\infty}(\Omega)}^2 \leq C \delta \text{ and }
\lm g_2 \rm_{L^p(\Omega)}^p \leq \lm f_2 \rm_{L^p(\Omega)}^p \leq C \delta. 
\end{align*}
By Proposition \ref{prop weaktyperigidity}, we derive the existence of $Q \in SO(2)$, $G_1 \in L^{2,\infty}(\Omega;\R^{2 \times 2})$ and $G_2 \in L^p(\Omega;\R^{2 \times 2})$ such that
\begin{equation*}
 \nabla u - Q= G_1 + G_2, \, \lm G_1 \rm_{L^{2,\infty}}^2 \leq C \delta \text{ and } \lm G_2 \rm_{L^{p}}^p \leq C \delta.
\end{equation*}
Without loss of generality we may assume that $Q = Id$ (otherwise replace $\beta$ by $Q^T \beta$). \\
Next, let $\vartheta: \Omega \rightarrow [-\pi,\pi)$ be a measurable function such that the corresponding rotation
\begin{equation*}
 R(\vartheta) = \begin{pmatrix}
                 \cos(\vartheta) & -\sin(\vartheta) \\ \sin(\vartheta) & \cos(\vartheta)
                \end{pmatrix}
\end{equation*}
satisfies
\begin{equation}
 |\beta(x) - R(\vartheta(x))| = \operatorname{dist}(\beta,SO(2)) \text{ for almost every } x \in \Omega. \label{eq: def R}
\end{equation}
Now, let us decompose
\begin{align}
 R(\vartheta(x)) - Id &= R(\vartheta(x)) - \beta + \beta - \nabla u + \nabla u - Id \nonumber \\
 &= R(\vartheta(x)) - \beta + \tilde \beta + G_1 + G_2. \label{eq: R-Id}
\end{align}
As $SO(2)$ is a bounded set, it is true that $| Id - R(\vartheta(x)) |^2 \leq C | Id - R(\vartheta(x)) |^p \wedge | Id - R(\vartheta(x)) |^2$. 
In addition, $|R(\vartheta(x)) - Id| \geq \frac{|\vartheta(x)|}{2}$.
Hence, by \eqref{eq: def R}, \eqref{eq: R-Id}, and the triangle inequality in Lemma \ref{lemma: triangle}, we may estimate
\begin{align*}
 \frac{|\vartheta(x)|^2}{4} \leq &|R(\vartheta(x)) - Id |^2 \\ \leq &C \left( \operatorname{dist}(\beta,SO(2))^2 \wedge \operatorname{dist}(\beta,SO(2))^p + |\tilde \beta |^2 + |G_1|^2 + |G_2|^p \right).
\end{align*}
Taking the $L^{1,\infty}$-quasinorm on both sides of the inequality we obtain
\begin{equation}
 \lm \vartheta \rm_{L^{2,\infty}(\Omega)}^2 \leq C \delta. \label{eq: weakestimatetheta}
\end{equation}
Following \cite[Theorem 3.3]{MuScZe14}, we define
\begin{equation*}
 R_{lin}(\vartheta) = \begin{pmatrix}
                       1 & -\vartheta \\ \vartheta & 1
                      \end{pmatrix}.
\end{equation*}
Using  \cite[Lemma 3.2]{MuScZe14}, we derive from \eqref{eq: weakestimatetheta} that
\begin{equation*}
 \lm R(\vartheta) - R_{lin}(\vartheta) \rm_{L^2}^2 \leq C \delta.
\end{equation*}
Thus, there exist functions $h_1 \in L^2(\Omega;\R^{2 \times 2})$ and $h_2 \in L^p(\Omega;\R^{2 \times 2})$ such that
\begin{equation}\label{eq: hs}
 \beta - R_{lin}(\vartheta) = \underbrace{\beta - R(\vartheta)}_{\in L^p(\Omega;\R^{2 \times 2}) + L^2(\Omega;\R^{2 \times 2})} + \underbrace{R(\vartheta) - R_{lin}(\vartheta)}_{\in L^2(\Omega;\R^{2 \times 2})} = h_1 + h_2,
\end{equation}
and $\lm h_1 \rm_{L^p(\Omega;\R^{2 \times 2})}^p, \, \lm h_2 \rm_{L^2(\Omega;\R^{2 \times 2})}^2 \leq C \delta$. 
By definition, we see that $\operatorname{curl} R_{lin}(\vartheta) = - \nabla \vartheta$. Hence,
\begin{equation*}
 \operatorname{curl} \beta = -\nabla \vartheta + \operatorname{curl} h_1 + \operatorname{curl} h_2,
\end{equation*}
which implies
\begin{equation*}
 \operatorname{div} \left((\operatorname{curl} \beta)^{\perp}\right) = \underbrace{\operatorname{div} \left((\operatorname{curl} h_1)^{\perp}\right)}_{\in W^{-2,p}(\Omega)} + \underbrace{\operatorname{div} \left((\operatorname{curl} h_2)^{\perp}\right)}_{\in H^{-2}(\Omega)}.
\end{equation*}
Therefore, we can apply Theorem \ref{thm: bourgainbrezis} to obtain two distributions $A \in H^{-1}(\Omega;\R^2)$ and $B \in W^{-1,p}(\Omega;\R^2)$ such that $(\operatorname{curl} \beta)^{\perp} = A + B$, 
\begin{align}\label{eq: decomositioncurlbeta}
&\lm A \rm_{H^{-1}(\Omega;\R^2)}^2 \leq C \left(|\operatorname{curl }\beta|(\Omega)^2 + \lm \operatorname{div} (\operatorname{curl} h_1)^{\perp}\rm^2_{H^{-2}(\Omega)}\right),  \\ \text{and }& \lm B \rm_{W^{-1,p}(\Omega;\R^2)}^p \leq C \lm \operatorname{div} (\operatorname{curl} h_2)^{\perp} \rm^p_{W^{-2,p}(\Omega)}.  \nonumber
\end{align}
In particular, it follows from \eqref{eq: hs} and \eqref{eq: decomositioncurlbeta} that
\begin{equation}\label{eq: estimateA}
 \lm A \rm_{H^{-1}(\Omega;\R^2)}^2 \leq C \delta \text{ and } \lm B \rm_{W^{-1,p}(\Omega;\R^2)}^p \leq C \delta.
\end{equation}
Clearly, the same holds for $\operatorname{curl }\beta$, $-A^{\perp}$ and $-B^{\perp}$. \\
Now, as $v$ is the unique solution to the linear problem \eqref{eq: laplace}, in view of \eqref{eq: decomositioncurlbeta} and \eqref{eq: estimateA} there exists a decomposition $v = v_1 + v_2$ satisfying $\lm v_1 \rm_{H^1(\Omega;\R^2)}^2 \leq C \delta$ and 
$\lm v_2 \rm_{W^{1,p}(\Omega;\R^2)}^p \leq C \delta$. 
Following the notation from the beginning of the proof, set
\begin{equation*}
 \tilde \beta_1 = \nabla v_1 J \text{ and } \tilde \beta_2 = \nabla v_2 J.
\end{equation*}
Then $\nabla u = \beta - \tilde \beta = \beta - \tilde \beta_1 - \tilde \beta_2$. Using the classical mixed growth rigidity estimate from \cite[Proposition 2.3]{MuPa13}, there exists a rotation $R \in SO(2)$ such that
\begin{equation*}
 \int_{\Omega} |\nabla u - R|^2 \wedge |\nabla u - R|^p \text{ d}x \leq C \int_{\Omega} \operatorname{dist}(\nabla u,SO(2))^2 \wedge \operatorname{dist}(\nabla u,SO(2))^p \text{ d}x.
\end{equation*}
With the use of Lemma \ref{lemma: triangle} we obtain eventually the following chain of inequalities
\begin{align*}
 &\int_{\Omega} |\beta - R|^2 \wedge |\beta - R|^p \text{ d}x \\
 &\leq C \left( \int_{\Omega} |\nabla u - R|^2 \wedge |\nabla u - R|^p \text{ d}x + \lm \tilde \beta_1 \rm_{L^2}^2 + \lm \tilde \beta_2 \rm_{L^p}^p\right)  \\
 &\leq C \left( \int_{\Omega} \operatorname{dist}(\nabla u,SO(2))^2 \wedge \operatorname{dist}(\nabla u,SO(2))^p \text{ d}x + \delta\right) \\
 &\leq C \left( \int_{\Omega} \operatorname{dist}(\beta,SO(2))^2 \wedge \operatorname{dist}(\beta,SO(2))^p \text{ d}x+ \lm \tilde \beta_1 \rm_{L^2}^2 + \lm \tilde \beta_2 \rm_{L^p}^p + \delta \right) \\
 &\leq C \delta,
\end{align*}
which finishes the proof.
\end{proof}
\section{Proof of the $\Gamma$-Limit Result}\label{sec: mixed}

In this section, we prove Theorem \ref{theorem: critical} and Theorem \ref{prop: compactness}.
We start with the compactness statement. 

\subsection{Compactness}

The main ingredient in the proof will be the generalized rigidity estimate from Theorem \ref{theorem: generalizedrigidity}.

\begin{proof}[Proof of Theorem \ref{prop: compactness}]
We prove the result in three steps.\\
 \textbf{Step 1.} \emph{Weak convergence of the scaled dislocation measures.} \\
In this step our objective is to show that there exists a constant $C>0$ such that 
 \begin{equation*}
  \frac{| \mu_j| (\Omega)}{\varepsilon_j | \log \varepsilon_j|}  \leq C.
 \end{equation*}
 Then the existence of a weakly*-converging subsequence is immediate.
 Let us fix $\alpha \in (0,1)$. 
 By the finiteness of $E_{\varepsilon_j}(\mu_{j},\beta_{j})$, it follows $\mu_{j} \in X_{\varepsilon_j}$. We write $\mu_{j} = \sum_{i=1}^{M_j} \varepsilon_j \xi_{i,j} \, \delta_{x_{i,j}}$ for appropriate $\xi_{i,j} \in \mathbb{S}$ and $x_{i,j} \in \Omega$.
As $\rho_{\varepsilon_j} \gg \varepsilon^{\alpha}$ we find for $j$ large enough that
 \begin{align}
 C &\geq \frac1{\varepsilon_j^2 |\log \varepsilon_j|^2} \sum_{i = 1}^{M_j} \int_{B_{\varepsilon_j^{ \alpha}}(x_{i,j})\setminus B_{\varepsilon_j}(x_{i,j})} W(\beta_j) \,dx. \label{eq: boundannuli}
\end{align}
 Although $\beta_j$ is not a gradient on $B_{\varepsilon_j^{ \alpha}}(x_{i,j})\setminus B_{\varepsilon_j}(x_{i,j})$, using a covering by overlapping simply connected domains, we can still use the rigidity estimate from \cite[Proposition 2.3]{MuPa13} to find rotations $R_{i,j} \in SO(2)$ such that for all $1 \leq i \leq M_j$ and $j \in \mathbb{N}$ it holds
 \begin{align} \label{eq: mixedrigidity} 
  &\int_{B_{\varepsilon_j^{ \alpha}}(x_{i,j})\setminus B_{\varepsilon_j}(x_{i,j})} |\beta_j - R_{i,j}|^2 \wedge |\beta_j - R_{i,j}|^p \, dx \\ \leq C &\int_{B_{\varepsilon_j^{ \alpha}}(x_{i,j})\setminus B_{\varepsilon_j}(x_{i,j})} \operatorname{dist }(\beta_j,SO(2))^2 \wedge \operatorname{dist }(\beta_j,SO(2))^p \, dx. \nonumber
 \end{align}
Note that as the relative thickness of the annuli $B_{\varepsilon_j^{ \alpha}}(x_{i,j})\setminus B_{\varepsilon_j}(x_{i,j})$ is uniformly bounded from below, we can choose the constant $C$ in the estimate above uniformly in $i$ and $j$.
Furthermore, using Jensen's inequality, we have
\begin{align}
 &\int_{B_{\varepsilon_j^{ \alpha}}(x_{i,j})\setminus B_{\varepsilon_j}(x_{i,j})} |\beta_j - R_{i,j}|^2 \wedge |\beta_j - R_{i,j}|^p \, dx \label{eq: estimatemixed} \\
 \geq &\int_{\varepsilon_j}^{\varepsilon_j^{\alpha}} 2\pi t \, \frac{1}{2\pi t}\int_{\partial B_{t}(x_{i,j})} \frac{|(\beta_j - R_{i,j})\cdot \tau|^2}{2} \wedge \frac{|(\beta_j - R_{i,j})\cdot \tau|^p}{p} \, d\mathcal{H}^1\, dt \nonumber \\
 \geq &\int_{\varepsilon_j}^{\varepsilon_j^{\alpha}} 2\pi t \left( \frac12 \left| \frac{1}{2\pi t} \int_{\partial B_{t}(x_{i,j})} (\beta_j - R_{i,j})\cdot \tau d\mathcal{H}^1 \right|^2 \wedge \frac1p \left| \frac{1}{2\pi t} \int_{\partial B_{t}(x_{i,j})} (\beta_j - R_{i,j})\cdot \tau d\mathcal{H}^1 \right|^p\right) \,dt \nonumber \\
 \geq &\int_{\varepsilon_j}^{\varepsilon_j^{\alpha}} \pi t \left(\left| \frac{\varepsilon_j \, \xi_{i,j}}{2 \pi t}\right|^2 \wedge \left| \frac{\varepsilon_j \, \xi_{i,j}}{2 \pi t} \right|^p \right) \,dt. \nonumber
\end{align} \\
Here, $\tau$ denotes the tangent to $\partial B_t(x_{i,j})$.\\ \newline
\noindent \emph{Claim: Let $\alpha < \gamma < 1$. Then it holds $\varepsilon_j |\xi_{i,j}| \leq \varepsilon_j^{\gamma}$ for all $1 \leq i \leq M_j$ and $j \in \mathbb{N}$ large enough.} \\
Assume this is not the case i.e., there exists a subsequence (not relabeled) and indices $1 \leq i_j \leq M_j$ such that $\varepsilon_j |\xi_{i_j,j}| \geq \varepsilon_j^{\gamma}$.
Combining \eqref{eq: boundannuli}, \eqref{eq: mixedrigidity}, and \eqref{eq: estimatemixed}, this implies for $j$ large enough that
\begin{align*}
 C &\geq \frac{1}{\varepsilon_j^2 |\log \varepsilon_j|^2}  \int_{\varepsilon_j}^{\varepsilon_j^{\alpha}} \pi t \left(\left| \frac{\varepsilon_j \, \xi_{i_j,j}}{2 \pi t}\right|^2 \wedge \left| \frac{\varepsilon_j \, \xi_{i_j,j}}{2 \pi t} \right|^p\right) \,dt \\
 &\geq \frac{1}{\varepsilon_j^2 |\log \varepsilon_j|^2} \int_{\varepsilon_j}^{\frac{\varepsilon_j^{\gamma}}{2\pi}} \pi t \left| \frac{\varepsilon_j \, \xi_{i_j,j}}{2 \pi t} \right|^p \, dt \\
 &= \frac{1}{\varepsilon_j^2 |\log \varepsilon_j|^2} \varepsilon_j^p |\xi_{i_j,j}|^p 2^{-p} \pi^{1-p} (2-p)^{-1} \left( \frac{\varepsilon_j^{(2-p)\gamma}}{(2\pi)^{2-p}} - \varepsilon_j^{(2-p)} \right) \\
\end{align*}
As we assume that $\varepsilon_j |\xi_{i_j,j}| \geq \varepsilon_j^{\gamma}$, we derive from the estimate above
\begin{equation*}
C \geq  2^{-p} \pi^{1-p} (2-p)^{-1} \frac{1}{|\log \varepsilon_j|^2} \left( \frac{\varepsilon_j^{2(\gamma - 1)}}{(2\pi)^{2-p}} - \varepsilon_j^{p(\gamma - 1)} \right) \rightarrow \infty
\end{equation*} 
since $2(\gamma - 1) < p (\gamma - 1) < 0$.
Contradiction! \\ \newline
Fix $\alpha < \gamma < 1$. 
The claim above, \eqref{eq: boundannuli}, \eqref{eq: mixedrigidity}, and \eqref{eq: estimatemixed} imply that
\begin{align}
 C &\geq  \sum_{i=1}^{M_j} \frac{1}{\varepsilon_j^2 |\log \varepsilon_j|^2}  \int_{\varepsilon_j}^{\varepsilon_j^{\alpha}} \pi t \left(\left| \frac{\varepsilon_j \, \xi_{i,j}}{2 \pi t}\right|^2 \wedge \left| \frac{\varepsilon_j \, \xi_{i_j,j}}{2 \pi t} \right|^p\right) \,dt \nonumber \\
 &\geq  \frac{1}{\varepsilon_j^2 |\log \varepsilon_j|^2} \sum_{i=1}^{M_j} \int_{\varepsilon^{\gamma}_j}^{\varepsilon_j^{\alpha}} \pi t \left| \frac{\varepsilon_j \, \xi_{i,j}}{2 \pi t} \right|^2 \, dt \nonumber \\
 &= \frac{1}{4 \pi |\log \varepsilon_j|^2} \sum_{i=1}^{M_j} |\xi_{i,j}|^2 (\gamma - \alpha) |\log \varepsilon_j| \label{eq: estimatemuj}.
\end{align}
As the non-zero elements of $\mathbb{S}$ are bounded away from zero, it follows directly from \eqref{eq: estimatemuj} that
\begin{equation*}
 C \geq \frac{1}{|\log \varepsilon_j|} \sum_{i=1}^{M_j} |\xi_{i,j}| = \frac{|\mu_j|(\Omega)}{\varepsilon_j |\log \varepsilon_j|}.
\end{equation*}
 
\noindent \textbf{Step 2.} \emph{Weak convergence of the scaled strains.} \\
Our assumptions imply directly that $\beta_j \in \mathcal{A}\mathcal{S}_{\varepsilon_j}(\mu_{j})$, in particular $\operatorname{curl }\beta_j = \mu_j$.
The generalized rigidity estimate, Theorem \ref{theorem: generalizedrigidity}, yields the existence of rotations $R_j \in SO(2)$ such that 
\begin{equation*}
 \int_{\Omega} |\beta_j - R_j|^2 \wedge |\beta_j - R_j|^p \,dx \leq C \left( \int_{\Omega} \operatorname{dist}(\beta_j,R)^2 \wedge \operatorname{dist}(\beta_j,R)^p \,dx + |\mu_j|(\Omega)^2 \right).
\end{equation*}
From the lower bound on $W$ (see \ref{item: property4}.~ in Section \ref{sec: settingmixed}) and step 1 it follows
\begin{equation}\label{eq: compactnessrigidity}
\int_{\Omega} |\beta_j - R_j|^2 \wedge |\beta_j - R_j|^p \,dx \leq C \varepsilon_j^2 |\log \varepsilon_j|^2.
\end{equation}
Set $G_j = \frac{R_j^T \beta_j - Id}{\varepsilon_j |\log \varepsilon_j|}$. 
Then
\begin{equation}
 \int_{\Omega} |G_j|^2 \wedge \frac{|G_j|^p}{\varepsilon_j^{2-p} |\log \varepsilon_j|^{2-p}} \,dx \leq C. \label{eq: boundGj}
\end{equation}
This implies that $(G_j)_j$ is a bounded sequence in $L^p(\Omega;\R^{2\times2})$. 
Hence, there exists a subsequence (again denoted by $G_j$) which converges weakly in $L^p(\Omega;\R^{2\times2})$ to some function $\beta \in L^p(\Omega;\R^{2\times2})$. \\
Next, we show that $\beta \in L^2(\Omega;\R^{2\times2})$. \\
Consider the decomposition of $\Omega$ into the two sets
\begin{align*}
 &A_j^2 = \left\{x \in \Omega: |G_j(x)|^2 \leq \frac{|G_j(x)|^p}{\varepsilon_j^{2-p} |\log \varepsilon_j|^{2-p}} \right\} \\  \text{ and }
 &A_j^p = \left\{x \in \Omega: |G_j(x)|^2 > \frac{|G_j(x)|^p}{\varepsilon_j^{2-p} |\log \varepsilon_j|^{2-p}} \right\}.
\end{align*}
By \eqref{eq: boundGj}, the sequence $|G_j| \1_{A_j^2}$ is bounded in $L^2(\Omega;\R^{2\times2})$.
Consequently, up to taking a further subsequence, the sequence converges weakly in $L^2(\Omega;\R^{2\times2})$ to a function $\tilde \beta \in L^2(\Omega;\R^{2\times 2})$.
It suffices to show that $\beta = \tilde \beta$. \\
By the definition of $G_j$, one verifies that
\begin{equation*}
 A_j^p = \left\{ x \in \Omega: |\beta_j - R_j|^2 > |\beta_j - R_j|^p \right\} = \{x \in \Omega: |\beta_j - R_j| > 1\}.
\end{equation*}
Then \eqref{eq: compactnessrigidity} implies
\begin{equation*}
 \left| A_j^p \right| \leq \int_{A_j^p} |\beta_j - R_j|^p \,dx \leq C \varepsilon_j^2 |\log \varepsilon_j|^2 \rightarrow 0.
\end{equation*}
Thus, $\1_{A_j^2} \rightarrow 1$ boundedly in measure which implies directly that also 
\begin{equation*}
 G_j \1_{A_j^2} \rightharpoonup \beta \text{ in } L^p(\Omega;\R^{2\times2}).
\end{equation*}
Hence, $\beta = \tilde \beta \in L^2(\Omega;\R^{2\times2})$. \\

\noindent \textbf{Step 3.} \emph{$\mu \in H^{-1}(\Omega;\R^2)$ and $\operatorname{curl} \beta = R^T \mu$.} \\
As $\beta \in L^2(\Omega;\R^{2\times 2})$, it is clear that $\operatorname{curl} \beta \in H^{-1}(\Omega;\R^2)$. 
Moreover, one computes for $\varphi \in C^{\infty}_c(\Omega;\R^2)$ and $J= \begin{pmatrix} 0 &-1 \\1 &0 \end{pmatrix}$ that
\begin{align*}
 &< \mu, \varphi >_{\mathcal{D}',\mathcal{D}} = \lim_j \frac{1}{\varepsilon_j |\log \varepsilon_j|} <\mu_j, \varphi >_{\mathcal{D}',\mathcal{D}} \\
 = &\lim_j \frac{1}{\varepsilon_j |\log \varepsilon_j|} <\operatorname{curl} (\beta_j - R_j), \varphi >_{\mathcal{D}',\mathcal{D}}
 = - \lim_j \frac{1}{\varepsilon_j |\log \varepsilon_j|} <\beta_j - R_j, J \nabla \varphi >_{\mathcal{D}',\mathcal{D}} 
\\ = &- < R \beta, J \nabla \varphi>_{\mathcal{D}',\mathcal{D}} = < \operatorname{curl} (R \beta), \varphi >_{\mathcal{D}',\mathcal{D}}.
\end{align*}
As $\operatorname{curl} (R \beta) = R \operatorname{curl} \beta$, it follows that $R^T \mu = \operatorname{curl} \beta$. 
\end{proof}

\subsection{The Self-Energy}\label{section: self-energy}
In this subsection we define the self-energy per dislocation, which appears in the limit, and recall briefly its most important properties.
For proofs of the statements in this section, we refer to \cite{GaLePo10}. \\
\noindent Let $0 < r_1 < r_2$ and $\xi \in \R^2$. 
Define
\begin{align*}
\mathcal{A}\mathcal{S}_{r_1,r_2}(\xi) = \left\{ \eta \in L^2\left( B_{r_2}(0) \setminus B_{r_1}(0);\R^{2\times 2} \right): \operatorname{curl }\eta = 0  \text{ and } \int_{\del B_{r_1}(0)} \eta \cdot t = \xi \right\}.
\end{align*}
Here, $\tau$ denotes the unit tangent to $\partial B_{r_1}(0)$.
The circulation condition has to be understood in the sense of traces.
For a function $\eta \in L^2\left( B_{r_2}(0) \setminus B_{r_1}(0);\R^{2\times 2}\right)$ which is $\operatorname{curl }$-free the tangential boundary values are well-defined in $H^{-\frac12}\left( B_{r_2}(0) \setminus B_{r_1}(0);\R^{2}\right)$ (see \cite[Theorem 2]{DaLi88}). 
The integral is then understood as testing with the constant $1$-function. \\
Note that this definition of admissible strains $\mathcal{A}\mathcal{S}_{r_1,r_2}(\xi)$ is defined by a circulation condition and not by a $\operatorname{curl}$-condition as in the definition of $\mathcal{A}\mathcal{S}_{\varepsilon}$ in Section \ref{sec: settingmixed}.
Clearly, the two formulations are linked via Stoke's theorem. \\
Next, set
\begin{equation}\label{def:psirR}
\psi_{r_1,r_2}(\xi) = \min \left\{ \frac12 \int_{B_{r_2}(0) \setminus B_{r_1}(0)} \mathcal{C}\eta:\eta \, dx: \, \eta \in \mathcal{A}\mathcal{S}_{r_1,r_2}(\xi) \right\},
\end{equation} 
where $\mathcal{C} = \frac{\del^2 W}{\del^2 F}(Id)$. 
Note that by scaling it holds that $\psi_{r_1,r_2}(\xi) = \psi_{\frac{r_1}{r_2},1}(\xi)$.
The special case $r_2=1$ will be denoted by
\begin{equation}\label{definition: psi}
 \psi(\xi,\delta) = \min \left\{ \frac12 \int_{B_1(0) \setminus B_{\delta}(0)} \mathcal{C}\eta:\eta \, dx: \, \eta \in \mathcal{A}\mathcal{S}_{1,\delta}(\xi) \right\}.
\end{equation}

\begin{proposition}[Corollary 6 and Remark 7 in \cite{GaLePo10}]\label{prop: convpsi}
 Let $\xi \in \R^2$, $\delta \in (0,1)$ and let $\psi(\xi,\delta)$ be defined as in \eqref{definition: psi}. 
Then for every $\xi \in \R^2$ it holds
 \begin{equation*}
  \lim_{\delta \to 0} \frac{\psi(\xi,\delta)}{|\log \delta|} = \psi(\xi),
 \end{equation*}
where $\psi: \R^2 \rightarrow [0,\infty)$ is defined by
\begin{equation}\label{def: selfenergy}
 \psi(\xi) = \lim_{\delta \to 0} \frac1{|\log \delta|} \frac12 \int_{B_1(0)\setminus B_{\delta}(0)} \mathcal{C} \eta_0: \eta_0 \, dx
\end{equation}
and $\eta_0: \R^2 \rightarrow \R^{2 \times 2}$ is a fixed distributional solution to
\begin{equation*}
 \begin{cases}
  \operatorname{curl }\eta_0 = \xi \delta_0 &\text{ in } \R^2, \\
  \operatorname{div }\eta_0 = 0 &\text{ in } \R^2. 
 \end{cases}
\end{equation*} 
In particular, both limits exist.
Moreover, there exists a constant $K>0$ such that for all $\delta >0$ small enough and $\xi \in \R^2$ it holds
\begin{equation*}
\left| \frac{\psi(\xi,\delta)}{|\log \delta|} - \psi(\xi)\right| \leq K \frac{|\xi|^2}{|\log \delta|}.
\end{equation*}
\end{proposition}
\begin{remark}
Note that the functions $\psi(\cdot,\delta)$ and $\psi$ are $2$-homogeneous and convex.
\end{remark}
\begin{remark} \label{remark: psiext}
In \cite[Proposition 8]{GaLePo10}, the authors show the following extension of the result above. 
Let $0 < r_{\delta} \to 0$ such that $\frac{ \log (r_{\delta})}{\log  (\delta)} \to 0$. 
Define for $\xi \in \R^2$ the function $\tilde{\psi}(\cdot,\delta)$ by
\begin{equation*}
\tilde{\psi}(\xi,\delta) = \min \left\{ \int_{B_{r_{\delta}}(0) \setminus B_{\delta}(0)} \frac12 \mathcal{C} \eta:\eta \,dx: \eta \in \mathcal{A}\mathcal{S}_{r_{\delta},\delta}(\xi) \right\}.
\end{equation*}
Then $\frac{\tilde{\psi}(\xi,\delta)}{|\log \delta|} = \frac{ \psi(\xi,\delta)}{|\log \delta|} (1 + o(1))$ where $o(1) \to 0$ as $\delta \to 0$. 
\end{remark}

The function $\psi$ is the (renormalized) limit self-energy of a single dislocation with Burgers vector $\xi$.
 The well-separateness condition on the dislocations does not prevent dislocations from merging to a single dislocation in the limit.
This could lead to a smaller limit energy per dislocation than $\psi$. 
The right way to capture this energetic behavior is to define the limit self-energy density $\varphi$ through a relaxation procedure.

\begin{definition}
We define the function $\varphi: SO(2) \times \R^2 \rightarrow [0,\infty)$ by
\begin{equation}\label{definition: phi}
 \varphi(R,\xi) = \min \left\{ \sum_{k=1}^{M} \lambda_k \psi(R^T \xi_k): \sum_{k=1}^M \lambda_k \xi_k, \, M \in \mathbb{N}, \, \lambda_k \geq 0, \, \xi_k \in \mathbb{S} \right\}.
\end{equation}
\end{definition}
\begin{remark}
Indeed, it can be seen by the $2$-homogeneity of $\psi$ that the $\min$ in the definition of $\varphi$ exists.
\end{remark}
\begin{remark}
 The function $\varphi(R,-)$ is convex and $1$-homogenous.
\end{remark}


\subsection{The $\Gamma$-Convergence Result}
Finally, we prove the $\Gamma$-convergence result for the energy $E_{\varepsilon}$ as defined in \eqref{definitition: energy}, Theorem \ref{theorem: critical}.
 The proof will be subdivided in Propositions \ref{prop: liminf} and \ref{prop: limsup}.

\begin{proposition}[The $\Gamma$-$\liminf$-inequality]\label{prop: liminf}
Let $\varepsilon_j \to 0$. 
Let $(\mu_{j},\beta_{j})$ be a sequence in the space $\mathcal{M}(\Omega;\R^2) \times L^p(\Omega;\R^{2\times2})$ that converges to a triplet $(\mu,\beta,R) \in  \mathcal{M}(\Omega;\R^2) \times L^p(\Omega;\R^{2\times2}) \times SO(2)$ in the sense of Definition \ref{def: convcrit}. 
Then
 \begin{equation*}
  \liminf_{j \to \infty} E_{\varepsilon_j}(\mu_{j},\beta_{j}) \geq E^{crit}(\mu, \beta,R).
 \end{equation*}
\end{proposition}

\begin{proof} 
 We may assume that $\liminf_{j \to \infty} E_{\varepsilon_j}(\mu_j,\beta_j) = \lim_{j \to \infty} E_{\varepsilon_j}(\mu_j,\beta_j)$. 
 Moreover, we may assume that $\sup_{j} E_{\varepsilon_j}(\mu_j,\beta_j) < \infty$.
 This implies that $\mu_j \in X_{\varepsilon_j}(\Omega)$ and $\beta_j \in \mathcal{A}\mathcal{S}_{\varepsilon_j}(\mu_j)$ for all $j$.
 In particular, the dislocation density $\mu_j$ is of the form $\mu_j = \sum_{i = 1}^{M_j} \varepsilon_j \xi_{i,j} \, \delta_{x_{i,j}}$ for some $0 \neq \xi_{i,j} \in \mathbb{S}$ and $x_{i,j} \in \Omega$.
A straightforward computation shows that the rotations provided by the application of the generalized rigidity estimate in the proof of the compactness result also converge to $R$. 
We will assume that the $R_j$ are those from the compactness result.
Then it follows that $\beta \in L^2(\Omega;\R^{2 \times 2})$, $\mu \in H^{-1}(\Omega;\R^2) \cap \mathcal{M}(\Omega;\R^2)$, and $\operatorname{curl }\beta = R^T \mu$. \\
 
 \noindent In order to prove the lower bound, we subdivide the energy $E_{\varepsilon_j}(\mu_{j},\beta_{j})$ into a part far from the dislocations and a contribution close to the dislocations (see also Figure \ref{fig: liminf}), precisely
 \begin{equation*}
  E_{\varepsilon_j}(\mu_j,\beta_j) = \frac1{\varepsilon_j^2 |\log \varepsilon_j|^2} \int_{\Omega_{\rho_{\varepsilon_j}}(\mu_j)} W(\beta_j) \, dx
  + \frac1{\varepsilon_j^2 |\log \varepsilon_j|^2} \sum_{i=1}^{M_j} \int_{B_{\rho_{\varepsilon_j}}(x_{i,j})} W(\beta_j) \, dx   ,
 \end{equation*}
 where we define for $r > 0$ the set $\Omega_r(\mu_j) = \Omega \setminus \bigcup_{i=1}^{M_j} B_r(x_{i,j})$.
The two contributions will be treated separately.\\
The first term on the right hand side will asymptotically include the linearized elastic interaction energy of the dislocations, the second term on the right hand side includes the self-energies of the dislocations. \\ \newline
 \noindent \textbf{\emph{Lower bound far from the dislocations. }}We will perform a second order Taylor expansion at scale $\varepsilon_j |\log \varepsilon_j|$ of the function $W$.
As the energy density has a minimum at the identity matrix, there exists a function $\sigma: \R^{2 \times 2} \rightarrow \R$ such that for all $F \in \R^{2 \times 2}$ it holds
 \begin{equation*}
  W(Id + F) = \frac12 \mathcal{C} F:F + \sigma(F),
 \end{equation*}
where $\sigma(F)/|F|^2 \to 0$ as $|F| \to 0$.
 Set $\omega(t) = \sup_{|F| \leq t} |\sigma(F)|$.
Note that $\omega(t)/t^2 \to 0$ as $t \to 0$.
 Then for all $F \in \R^{2 \times 2}$ 
  \begin{equation}\label{eq: taylor}
  W(Id + \varepsilon_j |\log \varepsilon_j| F) \geq \frac12 \varepsilon_j^2 |\log \varepsilon_j|^2 \mathcal{C} F:F - \omega(\varepsilon_j |\log \varepsilon_j| |F|).
 \end{equation}
Next, define 
\begin{equation*}
 G_j = \frac{R_{j}^T \beta_j - Id}{\varepsilon_j |\log \varepsilon_j|},
\end{equation*}
and
\begin{equation*}
 A_{\varepsilon_j}^2 = \left\{x \in \Omega: |G_j(x)|^2 \leq \frac{|G_j|^p}{\varepsilon_j^{2-p} |\log \varepsilon_j|^{2-p}} \right\}.
\end{equation*}
As in the proof of the compactness (Theorem \ref{prop: compactness}), it can be shown that $G_j \1_{A_{\varepsilon_j}^2} \rightharpoonup \beta$  in $L^2(\Omega;\R^{2\times 2})$ and 
$\1_{A_{\varepsilon_j}^2} \to 1$ boundedly in measure. 
Furthermore, define the set
\begin{equation*}
 B_{\varepsilon_j} = \left\{x \in \Omega: |G_j| \leq \varepsilon_j^{-\frac12} \right\}.
\end{equation*}
The boundedness of the sequence $(G_j)_{j}$ in $L^p(\Omega;\R^{2\times2})$  implies that $\1_{B_{\varepsilon_j}} \to 1$ boundedly in measure. \\
As the non-zero elements of $\mathbb{S}$ are bounded away from zero, we derive from the bound $\frac{|\mu_j|(\Omega)}{\varepsilon_j |\log \varepsilon_j|} \leq C$ that
\begin{equation*}
M_j \leq C \sum_{i=1}^{M_j} |\xi_{i,j}| = C \frac{|\mu_j|(\Omega)}{\varepsilon_j} \leq C |\log \varepsilon_j|.
\end{equation*}
Hence, by the assumptions on $\rho_{\varepsilon_j}$ we have
\begin{equation*}
|\Omega \setminus \Omega_{\rho_{\varepsilon_j}}(\mu_j)| \leq C |\log \varepsilon_j| \rho_{\varepsilon_j}^2 \to 0.
\end{equation*}
Consequently, $\1_{\Omega_{\rho_{\varepsilon_j}}(\mu_j)} \to 1$ boundedly in measure. \\
Eventually, define the function
\begin{equation*}
 \chi_{\varepsilon_j}(x) = \begin{cases}
                          1 &\text{if } x \in \Omega_{\rho_{\varepsilon_j}}(\mu_j) \cap A_{\varepsilon_j}^2 \cap B_{\varepsilon_j}, \\
                          0 &\text{else}.
                         \end{cases}
\end{equation*}
By the considerations above, we conclude that $\chi_{\varepsilon_j} \to 1$ boundedly in measure. 
Moreover, as $G_j \1_{A_{\varepsilon_j}^2} \rightharpoonup  \beta$ in $L^2(\Omega;\R^{2\times 2})$, we derive that also
\begin{equation}
G_j \chi_{\varepsilon_j} = G_j \1_{A_{\varepsilon_j}^2} \, \chi_{\varepsilon_j} \rightharpoonup \beta \text{ in } L^2(\Omega;\R^{2\times2}). \label{eq: L2convergence}
\end{equation}
Using the frame indifference of $W$ and \eqref{eq: taylor}, we estimate
\begin{align}
 &\frac1{\varepsilon_j^2 |\log \varepsilon_j|^2} \int_{\Omega_{\rho_{\varepsilon_j}}(\mu_j)} W(\beta_j) \, dx \nonumber \\= &\frac1{\varepsilon_j^2 |\log \varepsilon_j|^2} \int_{\Omega_{\rho_{\varepsilon_j}}(\mu_j)} W(R_{j}^T \beta_j) \, dx \nonumber \\
 \geq &\frac1{\varepsilon_j^2 |\log \varepsilon_j|^2} \int_{\Omega} \chi_{\varepsilon_j} W(R_{j}^T \beta_j) \, dx  \nonumber\\
 = &\frac1{\varepsilon_j^2 |\log \varepsilon_j|^2} \int_{\Omega} \chi_{\varepsilon_j} W(Id + \varepsilon_j |\log \varepsilon_j| G_j) \, dx\nonumber \\
 \geq &\int_{\Omega} \frac12 \mathcal{C} \left(\chi_{\varepsilon_j}G_j\right):\left(\chi_{\varepsilon_j}G_j\right) - \chi_{\varepsilon_j} \frac{\omega(\varepsilon_j |\log \varepsilon_j| |G_j|)}{\varepsilon_j^2 |\log \varepsilon_j|^2}\, dx \nonumber \\
 =&\int_{\Omega} \frac12  \mathcal{C}\left( \chi_{\varepsilon_j}G_j\right):\left(\chi_{\varepsilon_j}G_j\right) -  |\chi_{\varepsilon_j} G_j|^2 \, \frac{\omega(\varepsilon_j |\log \varepsilon_j| |G_j|)}{|G_j|^2 \varepsilon_j^2 |\log \varepsilon_j|^2}\, dx. \label{eq: liminflast} 
\end{align}
Now, recall \eqref{eq: L2convergence} and notice that the first term in \eqref{eq: liminflast} is lower semi-continuous with respect to weak convergence in $L^2(\Omega;\R^{2 \times 2})$.
For the second term in \eqref{eq: liminflast}, note that $(\chi_{\varepsilon_j} G_j)_{\varepsilon_j}$ is a bounded sequence in $L^2(\Omega;\R^{2\times2})$.
Moreover, note that whenever $\chi_{\varepsilon_j}(x) = 1$ we have $\varepsilon_j |\log \varepsilon_j| |G_j(x)| \leq \varepsilon_j^{\frac12} |\log \varepsilon_j| \rightarrow 0$.
Hence, by the properties of $\omega$ we find that
\begin{equation*}
 \chi_{\varepsilon_j} \frac{\omega(\varepsilon_j |\log \varepsilon_j| |G_j|)}{|G_j|^2 \varepsilon_j^2 |\log \varepsilon_j|^2} \rightarrow 0 \text{ in } L^{\infty}(\Omega).
\end{equation*}
Thus, 
\begin{equation*}
 \int_{\Omega} |\chi_{\varepsilon_j} G_j|^2 \frac{\omega(\varepsilon_j |\log \varepsilon_j| |G_j|)}{|G_j|^2 \varepsilon_j^2 |\log \varepsilon_j|^2}\, dx \rightarrow 0 \text{ as } \varepsilon_j \to 0.
\end{equation*}
Eventually, we derive from \eqref{eq: liminflast} 
\begin{equation*}
 \liminf_{j \to \infty} \frac1{\varepsilon_j^2 |\log \varepsilon_j|^2} \int_{\Omega_{\rho_{\varepsilon_j}}(\mu_j)} W(\beta_j) \, dx \geq 
 \int_{\Omega} \frac12 \mathcal{C} \beta: \beta \, dx.
\end{equation*} \\
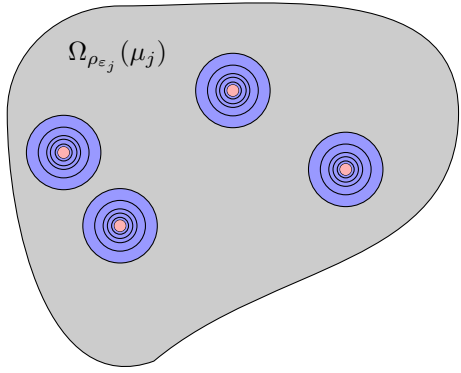
\begin{figure}
\centering
\begin{tikzpicture}[scale =1.5]
\fill[gray!40!white] (0,0) to[out = 90,in = 180] (1,0.95) to[out=0, in=90] (4,0) to[out = 270, in=40] (1.3,-2.2) to[out=200,in=270] (0,0);
\draw (0,0) to[out = 90,in = 180] (1,0.95) to[out=0, in=90] (4,0) to[out = 270, in=40] (1.3,-2.2) to[out=200,in=270] (0,0);
\draw (1.5,0.5) node[anchor=east] {$\Omega_{\rho_{\varepsilon_j}}(\mu_j)$};

\fill[blue!40!white] (1,-1) circle (0.33cm);
\draw (1,-1) circle (0.33cm);
\draw (1,-1) circle (0.225cm);
\draw (1,-1) circle (0.15cm);
\draw (1,-1) circle (0.1125cm);
\draw (1,-1) circle (0.075cm);
\draw (1,-1) circle (0.05cm);
\fill[red!30!white] (1,-1) circle (0.05cm);

\fill[blue!40!white] (3,-0.5) circle (0.33cm);
\draw (3,-0.5) circle (0.33cm);
\draw (3,-0.5) circle (0.225cm);
\draw (3,-0.5) circle (0.15cm);
\draw (3,-0.5) circle (0.1125cm);
\draw (3,-0.5) circle (0.075cm);
\draw (3,-0.5) circle (0.05cm);
\fill[red!30!white] (3,-0.5) circle (0.05cm);

\fill[blue!40!white] (0.5,-0.35) circle (0.33cm);
\draw (0.5,-0.35) circle (0.33cm);
\draw (0.5,-0.35) circle (0.225cm);
\draw (0.5,-0.35) circle (0.15cm);
\draw (0.5,-0.35) circle (0.1175cm);
\draw (0.5,-0.35) circle (0.075cm);
\draw (0.5,-0.35) circle (0.05cm);
\fill[red!30!white] (0.5,-0.35) circle (0.05cm);

\fill[blue!40!white] (2,0.2) circle (0.33cm);
\draw (2,0.2) circle (0.33cm);
\draw (2,0.2) circle (0.225cm);
\draw (2,0.2) circle (0.15cm);
\draw (2,0.2) circle (0.1175cm);
\draw (2,0.2) circle (0.075cm);
\draw (2,0.2) circle (0.05cm);
\fill[red!30!white] (2,0.2) circle (0.05cm);
\end{tikzpicture}
\caption{Sketch of the situation in the proof of the $\liminf$-inequality. The reduced domain $\Omega_{\rho_{\varepsilon_j}}(\mu_j)$ is drawn in gray, the balls around the dislocations with radius $\delta \varepsilon_j^{\alpha}$ are drawn in red. The annuli $B_{\rho_{\varepsilon_j}}(x_{i,j}) \setminus B_{\delta \varepsilon_j^{\alpha} }(x_{i,j})$ are drawn in blue and subdivided into annuli with constant ratio $\delta ^{-1}$.}
\label{fig: liminf}
\end{figure}
\noindent \textbf{\emph{Lower bound close the dislocations. }} Fix $\alpha, \delta \in (0,1)$. 
We subdivide for each $i \in \{1,\dots,M_j\}$ the annulus $B_{\rho_{\varepsilon_j}}(x_{i,j}) \setminus B_{\delta \varepsilon_j^\alpha}(x_{i,j})$ around the dislocation point $x_{i,j}$ into annuli with constant ratio $\delta^{-1}$ (see Figure \ref{fig: liminf}), namely define 
\begin{equation}
 C_j^{k,i} = B_{\delta^{k-1} \rho_{\varepsilon_j}}(x_{i,j})\setminus B_{\delta^{k} \rho_{\varepsilon_j}}(x_{i,j}) \label{eq: defcjki}
\end{equation}
for $k \in \{1,\dots, \tilde k_{j}\}$ where 
\begin{equation}\label{eq: deftildek}
 \tilde k_j = \left\lfloor \alpha \frac{|\log \varepsilon_j|}{|\log \delta|} - \frac{|\log \rho_{\varepsilon_j}|}{|\log \delta|} \right\rfloor + 1.
\end{equation}
Notice that for $k \leq \tilde k_j$ it holds $\delta^k \rho_{\varepsilon_j} \geq \delta^{\tilde k_j} \rho_{\varepsilon_j} \geq \delta \varepsilon^{\alpha}_j $.
Hence, for every $j \in \mathbb{N}$ and $i \in \{1, \dots, M_j\}$ we have
\begin{equation}
 \int_{B_{\rho_{\varepsilon_j}}(x_{i,j})\setminus B_{\delta \varepsilon_j^{\alpha}}(x_{i,j})} \frac{W(\beta_j)}{\varepsilon_j^2} \geq \sum_{k=1}^{\tilde k_j} \int_{C^{k,i}_j} \frac{W(\beta_j)}{\varepsilon_j^2} \,dx. \label{eq: decompositiondeltaannuli}
\end{equation}
Similar to the proof of \cite[Proposition 3.11]{ScZe12}, one can use a contradiction argument to show that there exists a sequence $\sigma_j \stackrel{j\to \infty}{\rightarrow} 0$ such that for all $j \in \mathbb{N}$, $i \in \{1, \dots ,M_j\}$, and $k \in \{1, \dots, \tilde k_j\}$ it holds
 \begin{equation}
  \int_{C_j^{k,i}} \frac{W(\beta_j)}{\varepsilon_j^2} \,dx \geq  \psi(R^T\xi_{i,j},\delta) - \sigma_j |\xi_{i,j}|^2, \label{eq: energydeltaannuli}
 \end{equation}
where $\psi(\cdot,\delta)$ is defined as in \eqref{definition: psi}.

\noindent Combining \eqref{eq: deftildek}, \eqref{eq: decompositiondeltaannuli}, and \eqref{eq: energydeltaannuli} yields
\begin{align}
&\frac1{|\log \varepsilon_j|^2} \sum_{i=1}^{M_j} \int_{B_{\rho_{\varepsilon_j}}(x_{i,j})} \frac{W(\beta_j)}{\varepsilon_j^2} \nonumber \\
 \geq &\frac1{|\log \varepsilon_j|^2} \sum_{i = 1}^{M_j} \tilde k_j \left(\psi(R^T \xi_{i,j},\delta) - \sigma_j |\xi_{i,j}|^2\right) \nonumber\\
 \geq &\frac1{|\log \varepsilon_j|} \sum_{i = 1}^{M_j} \left(\alpha - \frac{| \log \rho_{\varepsilon_j}|}{|\log \varepsilon_j|} \right)  \left(\frac{\psi(R^T \xi_{i,j},\delta)}{|\log \delta|} - \frac{\sigma_j |\xi_{i,j}|^2}{|\log \delta|}\right).
 \label{eq: energysum1}
\end{align}
From Proposition \ref{prop: convpsi} we know that there exists $K>0$ (which does not depend on $\delta$) such that for every $\xi \in \R^2$ it holds
\begin{equation*}
\left| \frac{\psi(\xi, \delta)}{|\log \delta|} - \psi(\xi) \right| \leq \frac{K |\xi|^2}{|\log \delta|}.
\end{equation*}
Hence,
\begin{align*}
 &\frac1{|\log \varepsilon_j|^2} \sum_{i=1}^{M_j} \int_{B_{\rho_{\varepsilon_j}}(x_{i,j})} \frac{W(\beta_j)}{\varepsilon_j^2} 
\\
 \geq &\frac1{|\log \varepsilon_j|} \sum_{i = 1}^{M_j} \left(\alpha - \frac{| \log \rho_{\varepsilon_j}|}{|\log \varepsilon_j|}\right)
 \left(\psi(R^T \xi_{i,j}) - \frac{K|\xi_{i,j}|^2}{|\log \delta|} - \frac{\sigma_j |\xi_{i,j}|^2}{|\log \delta|}\right). 
\end{align*} 
Arguing as in the proof of the compactness, we find similarly to \eqref{eq: estimatemuj} that
\begin{equation*}
  \frac1{|\log \varepsilon_j|} \sum_{i=1}^{M_j} |\xi_{i,j}|^2 \leq C
\end{equation*}
and thus
\begin{align}
 \frac1{|\log \varepsilon_j|^2} \sum_{i=1}^{M_j} \int_{B_{\rho_{\varepsilon_j}}(x_{i,j})} \frac{W(\beta_j)}{\varepsilon_j^2} \geq 
 &\frac1{|\log \varepsilon_j|} \left[ \sum_{i=1}^{M_j} \left(\alpha - \frac{|\log \rho_{\varepsilon_j}|}{|\log \varepsilon_j|}\right) \psi(R^T \xi_{i,j}) \right] \label{eq: liminfclose2} \\  
 &- C \left(\alpha - \frac{|\log \rho_{\varepsilon_j}|}{|\log \varepsilon_j|}\right) \left( \frac{1+ \sigma_j}{|\log \delta|} \right). \label{eq: liminfclose3}
\end{align}

\noindent Now, write $\tilde{\mu}_{j} = \frac{\mu_j}{\varepsilon_j | \log \varepsilon_j |}$. 
Using $\psi \geq \varphi$ and the $1$-homogeneity of $\varphi$ we estimate 
\begin{align}
 \frac1{|\log \varepsilon_j|} \sum_{i=1}^{M_j}\left(\alpha - \frac{|\log \rho_{\varepsilon_j}|}{| \log \varepsilon_j |} \right) \psi(R^T \xi_{i,j})  
 \geq \left(\alpha - \frac{|\log \rho_{\varepsilon_j}|}{| \log \varepsilon_j |}\right) \int_{\Omega} \varphi\left( R, \frac{d \tilde{\mu}_{j}}{d |\tilde{\mu}_{j}|} \right) \,d|\tilde{\mu}_{j}|.\label{eq: liminfclose3.5}
\end{align}
By the definition of the convergence of $(\mu_j,\beta_j)$ to the triple $(\beta,\mu,R)$, it holds in particular that $\tilde{\mu}_{\varepsilon_j} \stackrel{*}{\rightharpoonup} \mu$ in $\mathcal{M}(\Omega;\R^2)$.
As $\varphi$ is a continuous, convex, and $1$-homogeneous function, we may apply Reshetnyak's theorem to conclude
\begin{equation}
 \liminf_{j \to \infty} \frac1{|\log \varepsilon_j|} \sum_{i=1}^{M_j}\left(\alpha - \frac{|\log \rho_{\varepsilon_j}|}{| \log \varepsilon_j |}\right) \psi(R^T \xi_{i.j}) 
 \geq \alpha \int_{\Omega} \varphi\left(R, \frac{d \mu}{d|\mu|}\right) \,d |\mu|. \label{eq: liminfclose4}
\end{equation}
Combining \eqref{eq: liminfclose3} and \eqref{eq: liminfclose4} yields
\begin{align*}
 &\liminf_{j \to \infty} \frac1{|\log \varepsilon_j|^2} \sum_{i=1}^{M_j} \int_{B_{\rho_{\varepsilon_j}}(x_{i,j})} \frac{W(\beta_j)}{\varepsilon_j^2} \\ \geq
 &\alpha \int_{\Omega} \varphi\left(R, \frac{d \mu}{d|\mu|}\right) \,d |\mu| 
 - \limsup_{j \to \infty} C \left(\alpha - \frac{|\log \rho_{\varepsilon_j}|}{|\log \varepsilon_j|}\right) \left( \frac{1+ \sigma_j}{|\log \delta|} \right) \\
 = &\alpha \int_{\Omega} \varphi\left(R, \frac{d \mu}{d|\mu|}\right) \,d |\mu| - \frac{C\alpha}{|\log \delta|}.
\end{align*}
Letting $\alpha \to 1$ and $\delta \to 0$ finishes the proof of the lower bound close to the dislocations. \\
Combining the estimates close and far from the dislocations shows the claimed $\liminf$-inequality.
\end{proof}
Finally, we prove the existence of a recovery sequence for the energy $E^{crit}$.
\begin{proposition}[The $\Gamma$-$\limsup$-inequality]\label{prop: limsup}
Let $\varepsilon_j \to 0$.
 Let $R\in SO(2)$, $\beta \in L^2(\Omega; \R^{2\times2})$ such that $\operatorname{curl} \beta = R^T \mu  \in \mathcal{M}(\Omega;\R^2)$. 
 Then there exists a sequence of dislocation measures and associated strains $(\mu_{j},\beta_{j})_{j} \subseteq \mathcal{M}(\Omega; \R^2) \times  L^p(\Omega,\R^{2\times2})$ converging to $(\mu,\beta,R)$ in the sense of Definition \ref{def: convcrit} such that
 \begin{equation*}
  \limsup_{j \to \infty} E_{\varepsilon_j}(\mu_{j},\beta_{j}) \leq E^{crit}(\mu,\beta,R).
 \end{equation*}
\end{proposition}

\begin{proof}
We will use that the limit energy $E^{crit}$ is the same as in \cite{GaLePo10} and \cite{MuScZe14}.
In particular, we make use of the density result in \cite{GaLePo10} that allows us to restrict ourselves to the case that $\mu$ is locally constant and absolutely continuous with respect to the Lebesgue measure. \\ \newline
 \noindent \textbf{Step 1. }\emph{$\mu = \xi \,dx$ for some $\xi \in \R^2$.} \\
 Let $\lambda_1 , \dots, \lambda_M > 0$ and $\xi_1, \dots, \xi_M \in \mathbb{S}$ such that $\xi = \sum_{k=1}^M \lambda_k \xi_k$ and $\varphi(R,\xi) = \sum_{k=1}^M \lambda_k \psi(R^T \xi_k)$.
 Moreover, set
 \begin{equation*}
  \Lambda = \sum_{k=1}^M \lambda_k \text { and } r_{\varepsilon_j} = \frac{1}{2 \sqrt{\Lambda |\log \varepsilon_j|}}.
 \end{equation*}
Then, by the assumptions on $\rho_{\varepsilon_j}$, it holds that $\frac{\rho_{\varepsilon_j}}{r_{\varepsilon_j}} = 2 \sqrt{\Lambda} \sqrt{\rho_{\varepsilon_j}^2 |\log \varepsilon_j|} \to 0$.
According to  \cite[Lemma 11]{GaLePo10}, there exists a sequence of measures $\mu_{j} = \sum_{k=1}^M \varepsilon_j \xi_k \mu_{j}^k$ with $\mu_{j}^k$ of the type $\sum_{l=1}^{M_j^k} \delta_{x^k_{l,j}}$ for some $x^k_{l,j} \in \Omega$ such that 
  for all $x,y \in \operatorname{supp}(\mu_j)$ it holds $B_{r_{\varepsilon_j}}(x) \subseteq \Omega$ and $|x-y| \geq 2 r_{\varepsilon_j}$.
  Moreover,
  \begin{equation}
   \frac{|\mu_{j}^k|(\Omega)}{|\log \varepsilon_j|} \rightarrow \lambda_k |\Omega| \text{ and } \label{eq: |mu|}
   \frac{\mu_{j}}{\varepsilon_j |\log \varepsilon_j|} \stackrel{*}{\rightharpoonup} \mu \text{ in } \mathcal{M}(\Omega;\R^2).
  \end{equation}

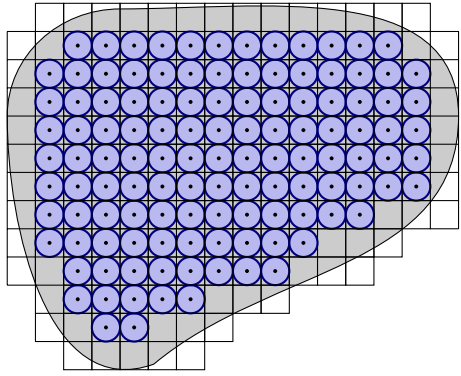
\begin{figure}
\centering
\begin{tikzpicture}[scale = 1.5]
\fill[gray!40!white] (0,0) to[out = 90,in = 180] (1,0.95) to[out=0, in=90] (4,0) to[out = 270, in=40] (1.3,-2.2) to[out=200,in=270] (0,0);
\draw (0,0) to[out = 90,in = 180] (1,0.95) to[out=0, in=90] (4,0) to[out = 270, in=40] (1.3,-2.2) to[out=200,in=270] (0,0);

\pgfmathsetmacro{\x}{0.25}
\pgfmathsetmacro{\y}{-1.25}
\pgfmathsetmacro{\c}{-1}
\pgfmathsetmacro{\v}{0.25}
\pgfmathsetmacro{\z}{0.5}
\foreach \a in {\y,\c,...,\v}{
\fill[blue!30!white,opacity=0.7] (\x+0.125,\a + 0.125) circle (0.125cm);
\draw[blue!50!black,thick]  (\x+0.125,\a + 0.125) circle (0.125cm);
\fill (\x+0.125,\a + 0.125) circle (0.5pt);
}
\draw[step =0.25] (\x,\y-0.75) grid (\x+0.25,\z+0.5);

\pgfmathsetmacro{\x}{0.5}
\pgfmathsetmacro{\y}{-1.75}
\pgfmathsetmacro{\c}{-1.5}
\pgfmathsetmacro{\v}{0.5}
\pgfmathsetmacro{\z}{0.75}
\foreach \a in {\y,\c,...,\v}{
\fill[blue!30!white,opacity=0.7] (\x+0.125,\a + 0.125) circle (0.125cm);
\draw[blue!50!black,thick]  (\x+0.125,\a + 0.125) circle (0.125cm);
\fill (\x+0.125,\a + 0.125) circle (0.5pt);
}
\foreach \a in {-2.25,-2,...,0.75}{
\draw (\x,\a) -- (\x+0.25,\a) -- (\x+0.25,\a+0.25) -- (\x,\a+0.25) -- cycle;
}

\pgfmathsetmacro{\x}{0.75}
\pgfmathsetmacro{\y}{-2}
\pgfmathsetmacro{\c}{-1.75}
\pgfmathsetmacro{\v}{0.5}
\pgfmathsetmacro{\z}{0.75}
\foreach \a in {\y,\c,...,\v}{
\fill[blue!30!white,opacity=0.7] (\x+0.125,\a + 0.125) circle (0.125cm);
\draw[blue!50!black,thick]  (\x+0.125,\a + 0.125) circle (0.125cm);
\fill (\x+0.125,\a + 0.125) circle (0.5pt);
}
\foreach \a in {-2.25,-2,...,0.75}{
\draw (\x,\a) -- (\x+0.25,\a) -- (\x+0.25,\a+0.25) -- (\x,\a+0.25) -- cycle;
}

\pgfmathsetmacro{\x}{1}
\pgfmathsetmacro{\y}{-2}
\pgfmathsetmacro{\c}{-1.75}
\pgfmathsetmacro{\v}{0.5}
\pgfmathsetmacro{\z}{0.75}
\foreach \a in {\y,\c,...,\v}{
\fill[blue!30!white,opacity=0.7] (\x+0.125,\a + 0.125) circle (0.125cm);
\draw[blue!50!black,thick]  (\x+0.125,\a + 0.125) circle (0.125cm);
\fill (\x+0.125,\a + 0.125) circle (0.5pt);
}
\foreach \a in {-2.25,-2,...,0.75}{
\draw (\x,\a) -- (\x+0.25,\a) -- (\x+0.25,\a+0.25) -- (\x,\a+0.25) -- cycle;
}

\pgfmathsetmacro{\x}{1.25}
\pgfmathsetmacro{\y}{-1.75}
\pgfmathsetmacro{\c}{-1.5}
\pgfmathsetmacro{\v}{0.5}
\pgfmathsetmacro{\z}{0.75}
\foreach \a in {\y,\c,...,\v}{
\fill[blue!30!white,opacity=0.7] (\x+0.125,\a + 0.125) circle (0.125cm);
\draw[blue!50!black,thick]  (\x+0.125,\a + 0.125) circle (0.125cm);
\fill (\x+0.125,\a + 0.125) circle (0.5pt);
}
\foreach \a in {-2.25,-2,-1.75,...,0.75}{
\draw (\x,\a) -- (\x+0.25,\a) -- (\x+0.25,\a+0.25) -- (\x,\a+0.25) -- cycle;
}

\pgfmathsetmacro{\x}{1.5}
\pgfmathsetmacro{\c}{-1.5}
\pgfmathsetmacro{\v}{0.5}
\pgfmathsetmacro{\y}{-1.75}
\pgfmathsetmacro{\z}{0.75}
\foreach \a in {\y,\c,...,\v}{
\fill[blue!30!white,opacity=0.7] (\x+0.125,\a + 0.125) circle (0.125cm);
\draw[blue!50!black,thick]  (\x+0.125,\a + 0.125) circle (0.125cm);
\fill (\x+0.125,\a + 0.125) circle (0.5pt);
}
\foreach \a in {-2.25,-2,-1.75,...,0.75}{
\draw (\x,\a) -- (\x+0.25,\a) -- (\x+0.25,\a+0.25) -- (\x,\a+0.25) -- cycle;
}

\pgfmathsetmacro{\x}{1.75}
\pgfmathsetmacro{\y}{-1.5}
\pgfmathsetmacro{\c}{-1.25}
\pgfmathsetmacro{\v}{0.5}
\pgfmathsetmacro{\z}{0.75}
\foreach \a in {\y,\c,...,\v}{
\fill[blue!30!white,opacity=0.7] (\x+0.125,\a + 0.125) circle (0.125cm);
\draw[blue!50!black,thick]  (\x+0.125,\a + 0.125) circle (0.125cm);
\fill (\x+0.125,\a + 0.125) circle (0.5pt);
}
\foreach \a in {-2,-1.75,-1.5,...,0.75}{
\draw (\x,\a) -- (\x+0.25,\a) -- (\x+0.25,\a+0.25) -- (\x,\a+0.25) -- cycle;
}

\pgfmathsetmacro{\x}{2}
\pgfmathsetmacro{\y}{-1.5}
\pgfmathsetmacro{\c}{-1.25}
\pgfmathsetmacro{\v}{0.5}
\pgfmathsetmacro{\z}{0.75}
\foreach \a in {\y,\c,...,\v}{
\fill[blue!30!white,opacity=0.7] (\x+0.125,\a + 0.125) circle (0.125cm);
\draw[blue!50!black,thick]  (\x+0.125,\a + 0.125) circle (0.125cm);
\fill (\x+0.125,\a + 0.125) circle (0.5pt);
}
\foreach \a in {-1.75,-1.5,...,0.75}{
\draw (\x,\a) -- (\x+0.25,\a) -- (\x+0.25,\a+0.25) -- (\x,\a+0.25) -- cycle;
}

\pgfmathsetmacro{\x}{2.25}
\pgfmathsetmacro{\y}{-1.5}
\pgfmathsetmacro{\c}{-1.25}
\pgfmathsetmacro{\v}{0.5}
\pgfmathsetmacro{\z}{0.75}
\foreach \a in {\y,\c,...,\v}{
\fill[blue!30!white,opacity=0.7] (\x+0.125,\a + 0.125) circle (0.125cm);
\draw[blue!50!black,thick]  (\x+0.125,\a + 0.125) circle (0.125cm);
\fill (\x+0.125,\a + 0.125) circle (0.5pt);
}
\foreach \a in {-1.75,-1.5,...,0.75}{
\draw (\x,\a) -- (\x+0.25,\a) -- (\x+0.25,\a+0.25) -- (\x,\a+0.25) -- cycle;
}

\pgfmathsetmacro{\x}{2.5}
\pgfmathsetmacro{\y}{-1.25}
\pgfmathsetmacro{\c}{-1}
\pgfmathsetmacro{\v}{0.5}
\pgfmathsetmacro{\z}{0.75}
\foreach \a in {\y,\c,...,\v}{
\fill[blue!30!white,opacity=0.7] (\x+0.125,\a + 0.125) circle (0.125cm);
\draw[blue!50!black,thick]  (\x+0.125,\a + 0.125) circle (0.125cm);
\fill (\x+0.125,\a + 0.125) circle (0.5pt);
}
\foreach \a in {-1.5,-1.25,...,0.75}{
\draw (\x,\a) -- (\x+0.25,\a) -- (\x+0.25,\a+0.25) -- (\x,\a+0.25) -- cycle;
}

\pgfmathsetmacro{\x}{2.75}
\pgfmathsetmacro{\y}{-1}
\pgfmathsetmacro{\c}{-0.75}
\pgfmathsetmacro{\v}{0.5}
\pgfmathsetmacro{\z}{0.75}
\foreach \a in {\y,\c,...,\v}{
\fill[blue!30!white,opacity=0.7] (\x+0.125,\a + 0.125) circle (0.125cm);
\draw[blue!50!black,thick]  (\x+0.125,\a + 0.125) circle (0.125cm);
\fill (\x+0.125,\a + 0.125) circle (0.5pt);
}
\foreach \a in {-1.5,-1.25,-1,...,0.75}{
\draw (\x,\a) -- (\x+0.25,\a) -- (\x+0.25,\a+0.25) -- (\x,\a+0.25) -- cycle;
}

\pgfmathsetmacro{\x}{3}
\pgfmathsetmacro{\y}{-1}
\pgfmathsetmacro{\c}{-0.75}
\pgfmathsetmacro{\v}{0.5}
\pgfmathsetmacro{\z}{0.75}
\foreach \a in {\y,\c,...,\v}{
\fill[blue!30!white,opacity=0.7] (\x+0.125,\a + 0.125) circle (0.125cm);
\draw[blue!50!black,thick]  (\x+0.125,\a + 0.125) circle (0.125cm);
\fill (\x+0.125,\a + 0.125) circle (0.5pt);
}
\foreach \a in {-1.5,-1.25,-1,...,0.75}{
\draw (\x,\a) -- (\x+0.25,\a) -- (\x+0.25,\a+0.25) -- (\x,\a+0.25) -- cycle;
}

\pgfmathsetmacro{\x}{3.25}
\pgfmathsetmacro{\y}{-0.75}
\pgfmathsetmacro{\c}{-0.5}
\pgfmathsetmacro{\v}{0.5}
\pgfmathsetmacro{\z}{0.75}
\foreach \a in {\y,\c,...,\v}{
\fill[blue!30!white,opacity=0.7] (\x+0.125,\a + 0.125) circle (0.125cm);
\draw[blue!50!black,thick]  (\x+0.125,\a + 0.125) circle (0.125cm);
\fill (\x+0.125,\a + 0.125) circle (0.5pt);
}
\foreach \a in {-1.25,-1,-0.75,...,0.75}{
\draw (\x,\a) -- (\x+0.25,\a) -- (\x+0.25,\a+0.25) -- (\x,\a+0.25) -- cycle;
}

\pgfmathsetmacro{\x}{3.5}
\pgfmathsetmacro{\y}{-0.75}
\pgfmathsetmacro{\c}{-0.5}
\pgfmathsetmacro{\v}{0.25}
\pgfmathsetmacro{\z}{0.5}
\foreach \a in {\y,\c,...,\v}{
\fill[blue!30!white,opacity=0.7] (\x+0.125,\a + 0.125) circle (0.125cm);
\draw[blue!50!black,thick]  (\x+0.125,\a + 0.125) circle (0.125cm);
\fill (\x+0.125,\a + 0.125) circle (0.5pt);
}
\foreach \a in {-1,-0.75,...,0.75}{
\draw (\x,\a) -- (\x+0.25,\a) -- (\x+0.25,\a+0.25) -- (\x,\a+0.25) -- cycle;
}

\draw[step = 0.25cm] (0,-1.5) grid (0.25,0.75);
\draw[step=0.25cm] (3.75,-1) grid (4,0.75);

\end{tikzpicture}
\caption{Sketch of the construction of the measures $\mu_j$, $\tilde{\mu}_j^{r_{\varepsilon_j}}$, and $\hat{\mu}_j^{r_{\varepsilon_j}}$ for a Burgers vector $\xi$ such that $\varphi(R,\xi) = \psi(R^T\xi)$: cover $\Omega$ with squares of side length $\sqrt{|\log \varepsilon_j|}^{-1}$. In every square that is included in $\Omega$ put a Dirac mass with weight $\xi$ (black dot) for $\mu_j$, a continuously distributed mass of $\xi$ on the circle of diameter $ \sqrt{|\log \varepsilon_j|}^{-1}$ (blue circle) for $\tilde{\mu}_j^{r_{\varepsilon_j}}$, and a measure of mass $\xi$ distributed on the boundary of that circle (dark blue) for $\hat{\mu}_j^{r_{\varepsilon_j}}$.}
\label{fig: constrmu}
\end{figure}

\noindent Note that by construction it holds $\mu_{j} \in X_{\varepsilon_j}$. \\
It is useful to combine the two summations in the definition of $\mu_{j}$ into
\begin{equation*}
\mu_{j} = \sum_{i=1}^{M_j} \varepsilon_j \xi_{i,j} \delta_{x_{i,j}}
\end{equation*}
for appropriate $\xi_{i,j} \in \mathbb{S}$ and $x_{i,j} \in \Omega$.
It follows that $M_j \leq C r_{\varepsilon_j}^{-2}$. \\
In \cite{BaBaSc80}, it is shown that for every $i=1, \dots, M_j$ there exists a strain field $\eta_i^{j}: \R^2 \rightarrow \R^{2\times2}$ of the form $\eta_i^{j} = \frac1{|x - x_{i,j}|} \Gamma_{R^T \xi_{i,j}}\left( \frac{x - x_{i,j}}{|x - x_{i,j}|} \right)$ solving
\begin{equation*}
 \begin{cases}
  \operatorname{curl} \eta_i^j = R^T \xi_{i,j} \delta_{x_{i,j}} &\text{in } \R^2, \\
  \operatorname{div} \mathcal{C} \eta_i^j = 0 &\text{in } \R^2.
 \end{cases}
\end{equation*}
The functions $\Gamma_{R^T \xi_{i,j}}: \R^2 \rightarrow \R^{2 \times 2}$ are uniformly bounded in $i$ and $j$. \\
Define
\begin{equation}\label{def: etaeps}
 \eta^{j} = \sum_{i=1}^{M_j} \varepsilon_j \eta_i^{j} \, \1_{B_{r_{\varepsilon_j}}(x_{i,j})}.
\end{equation}
Then  $\operatorname{curl }\eta^{j}$ equals $R^T\mu_{j}$ up to an error term arising from $\1_{B_{r_{\varepsilon_j}}(x_{i,j})}$, precisely
\begin{align}
\operatorname{curl} \eta^{j} &= \sum_{i=1}^{M_j} \varepsilon_j R^T\xi_{i,j} \delta_{x_{i,j}} - \varepsilon_j \eta_i^{j}(x) \, \frac{(x-x_{i,j})^{\perp}}{|(x-x_{i,j})|} \, d\mathcal{H}^1_{|\partial B_{r_{\varepsilon_j}}(x_{i,j})} \nonumber \\
&= R^T \mu_{j} - \sum_{i=1}^{M_j} \varepsilon_j \frac{ \eta_i^{j}(x) \,(x-x_{i,j})^{\perp}}{r_{\varepsilon_j}} \, d\mathcal{H}^1_{|\del B_{r_{\varepsilon_j}}(x_{i,j})} \nonumber \\
&=: R^T\mu_{j} - R^T\hat{\mu}_{j}^{r_{\varepsilon_j}}. \label{eq: defhatmu}
\end{align}
Note that $\hat{\mu}_{j}^{r_{\varepsilon_j}} \in H^{-1}(\Omega;\R^2)$. 
\\
Moreover, we define the auxiliary measure
\begin{equation}\label{eq: deftildemu}
 \tilde{\mu}_{j}^{r_{\varepsilon_j}} = R \sum_{i = 1}^{M_j} 2 \varepsilon_j \frac{ \eta_i^{j}(x) \, (x -x_{i,j})^{\perp} }{r_{\varepsilon_j}^2} \1_{B_{r_{\varepsilon_j}}(x_{i, j})} \, dx.
\end{equation}
For a sketch of the measures $\tilde{\mu}_j^{r_{\varepsilon_j}}$ and $\hat{\mu}_j^{r_{\varepsilon_j}}$, see Figure \ref{fig: constrmu}.\\
A straightforward computation shows that for all $i \in \{1, \dots, M_j \}$ it holds
\begin{equation*}
 \tilde{\mu}_{j}^{r_{\varepsilon_j}}(B_{r_{\varepsilon_j}}(x_{i,j})) = \hat{\mu}_{j}^{r_{\varepsilon_j}}(\partial B_{r_{\varepsilon_j}}(x_{i,j})) = \varepsilon_j \xi_{i,j}.
\end{equation*}
In \cite[Lemma 11]{GaLePo10}, it is also shown that 
\begin{gather}
 \frac{\hat{\mu}_{j}^{r_{\varepsilon_j}}}{\varepsilon_j |\log \varepsilon_j|} \stackrel{*}{\rightharpoonup} R^T\mu \text{ in } \mathcal{M}(\Omega;\R^2), \,
 \frac{\tilde{\mu}_{j}^{r_{\varepsilon_j}}}{\varepsilon_j |\log \varepsilon_j|} \stackrel{*}{\rightharpoonup} R^T\mu \text{ in } L^{\infty}(\Omega;\R^2), \nonumber \\
 \frac{\tilde{\mu}_{j}^{r_{\varepsilon_j}}}{\varepsilon_j |\log \varepsilon_j|} \rightarrow R^T\mu \text{ in } H^{-1}(\Omega;\R^2), \text{ and } 
 \frac{\tilde{\mu}_{j}^{r_{\varepsilon_j}} - \hat{\mu}_{j}^{r_{\varepsilon_j}}}{\varepsilon_j |\log \varepsilon_j|} \rightarrow 0 \text{ in } H^{-1}(\Omega;\R^2). \label{eq: lemma11}
\end{gather}
In order to define the recovery sequence, we introduce the auxiliary strain
 \begin{equation}\label{eq: defK}
  \tilde{K}_{\mu_{j}}^{r_{\varepsilon_j}} = \frac{\varepsilon_j}{r_{\varepsilon_j}^2} \sum_{i=1}^{M_j}   \eta^{j}_i\, |x - x_{i,j}|^2 \, \1_{B_{r_{\varepsilon_j}}(x_{i,j})}.
 \end{equation}
A straightforward calculation shows that $\operatorname{curl} \tilde{K}_{\mu_{j}}^{r_{\varepsilon_j}} = R^T (\tilde{\mu}_{j}^{r_{\varepsilon_j}} - \hat{\mu}_{j}^{r_{\varepsilon_j}})$. \\
Now, we define the approximating strains as
\begin{equation}
 \beta_{j} = R \left( Id + \varepsilon_j |\log \varepsilon_j| \beta + \eta^{j} - \tilde{K}_{\mu_{j}}^{r_{\varepsilon_j}} + \tilde{\beta}_{j} \right),
\end{equation}
where $\tilde{\beta}_{j}= \nabla w_{j}J$ for $  J = 
 \begin{pmatrix}
  0 & -1 \\ 1 & 0
 \end{pmatrix}$
 and $w_{j}$ the solution to
\begin{equation}\label{eq: defw}
 \begin{cases}
  - \Delta w_{j} = \varepsilon_j |\log \varepsilon_j| R^T \mu - R^T \tilde{\mu}^{r_{\varepsilon_j}}_{j} & \text{in } \Omega, \\
  w_ {j} \in H^1_0(\Omega;\R^2).
 \end{cases}
\end{equation}
Then $\beta_{j} \in \mathcal{A}\mathcal{S}_{\varepsilon_j}(\mu_{j})$. 
Indeed, one can show by a direct computation that $\eta^j$ and $K_{\mu_j}^{r_{\varepsilon_j}}$ are in $L^p(\Omega;\R^{2\times2})$; the function $\tilde{\beta}_j$ belongs to the space $L^2(\Omega;\R^{2\times2})$ by definition. 
Hence, each summand in the definition of $\beta_j$  is in the space $L^p(\Omega;\R^{2\times2})$. 
Furthermore,
\begin{equation*}
 \operatorname{curl} {\beta_{j}} = \varepsilon_j |\log \varepsilon_j| \mu + \mu_{j} - \hat{\mu}_{j}^{r_{\varepsilon_j}} - \tilde{\mu}_{j}^{r_{\varepsilon_j}} + \hat{\mu}_{j}^{r_{\varepsilon_j}} - \varepsilon_j |\log \varepsilon_j| \mu + \tilde{\mu}_{j}^{r_{\varepsilon_j}} 
 = \mu_{j}.
\end{equation*}
As in \cite[$\Gamma$-limsup inequality]{MuScZe14}, it can be shown for $\Omega_{\varepsilon_j}(\mu_j) = \Omega \setminus \bigcup_{ x \in \operatorname{supp }(\mu_j)} B_{\varepsilon_j}(x)$ that
\begin{enumerate}
 \item $\frac{\eta^{j}\1_{\Omega_{\varepsilon_j}(\mu_{j})}}{\varepsilon_j |\log \varepsilon_j|} \rightharpoonup 0$ in $L^2(\Omega;\R^{2\times2})$, \label{item: etaweak}
 \item $\frac{\tilde{K}_{\mu_{j}}^{r_{\varepsilon_j}}}{\varepsilon_j |\log \varepsilon_j|} \rightarrow 0$ in $L^2(\Omega;\R^{2\times2})$, \label{item: Kstrong}
 \item $\frac{\tilde{\beta}_{j}}{\varepsilon_j |\log \varepsilon_j|} \rightarrow 0$ in $L^2(\Omega;\R^{2\times2})$. \label{item: tildebetastrong}
\end{enumerate}
The boundedness in $L^2(\Omega;\R^{2\times2})$ of the function in \ref{item: etaweak}.~is a straightforward computation. 
The identification of the weak limit can be done in $L^p(\Omega;\R^{2\times2})$.
For \ref{item: Kstrong}.~notice that $|\tilde{K}_{\mu_{j}}^{r_{\varepsilon_j}}| \leq C \varepsilon_j \sqrt{|\log \varepsilon_j|}$.
In view of \eqref{eq: defw} and \eqref{eq: lemma11}, the last statement follows by classical elliptic estimates. \\
Furthermore, it can be proved that
\begin{enumerate}
\setcounter{enumi}{3}
 \item[4)] $\frac{\eta^{j}}{\varepsilon_j |\log \varepsilon_j|} \rightarrow 0$ in $L^p(\Omega;\R^{2\times2})$. \label{item: etastrong}
\end{enumerate}
In fact,
\begin{align*}
 \int_{\Omega} \left| \frac{\eta^{j}}{\varepsilon_j |\log \varepsilon_j|}\right|^p \, dx &= \frac{1}{|\log \varepsilon_j|^p} \sum_{i=1}^{M_j} \int_{B_{r_{\varepsilon_j}}(x_{i,j})} |\eta^{j}_i|^p \, dx\\
 &\leq \frac{C}{|\log \varepsilon_j|^p} \sum_{i=1}^{M_j} \int_0^{r_{\varepsilon_j}} r^{1 - p} \, dr \\
 &\leq C (2-p)^{-1} \frac{M_j}{|\log \varepsilon_j|^p} r_{\varepsilon_j}^{2 - p} \\
 &\leq  C (2-p)^{-1} |\log \varepsilon_j|^{-p} r_{\varepsilon_j}^{-p} \leq C (2-p)^{-1} |\log \varepsilon_j|^{-\frac p2} \rightarrow 0.
\end{align*}
Hence, $(\mu_{j},\beta_{j})$ converges to $(\mu,\beta,R)$ in the sense of Definition \ref{def: convcrit} with $R_{\varepsilon_j} = R$. \\ \newline
\noindent Next, we will show the $\limsup$-inequality for the energies.
For this purpose, fix $\alpha \in (0,1)$.
We split the energy as follows
\begin{equation} \label{eq: split}
 E_{\varepsilon_j}(\mu_{j},\beta_{j}) = \underbrace{\frac{1}{\varepsilon_j^2 |\log \varepsilon_j|^2} \int_{\Omega_{\varepsilon_j^{\alpha}}(\mu_{j})} W(\beta_{j}) \,dx}_{=: I^1_{\varepsilon_j}} +
 \underbrace{\frac{1}{\varepsilon_j^2 |\log \varepsilon_j|^2} \sum_{i=1}^{M_j} \int_{B_{\varepsilon_j^{\alpha}}(x_{i,j})} W(\beta_{j}) \,dx}_{=: I^2_{\varepsilon_j}}.
\end{equation}
First, we show that
\begin{align}\label{eq: Iepsilon1} 
 \limsup_{\varepsilon_j \to 0} I^1_{\varepsilon_j} \leq &\int_{\Omega} \frac12 \mathcal{C} \beta : \beta \,dx + \alpha \int_{\Omega} \varphi(R,\xi) \, dx  \\
 = &\int_{\Omega} \frac12 \mathcal{C} \beta : \beta \,dx + \alpha \int_{\Omega} \varphi\left(R, \frac{d \mu}{d |\mu|}\right) \, d|\mu|. \nonumber 
\end{align}
Using a second order Taylor expansion and the frame indifference  of the energy density $W$, we obtain similarly to the $\liminf$-inequality, \eqref{eq: liminflast}, that
\begin{align*}
 I^1_{\varepsilon_j} &= \frac{1}{\varepsilon_j^2 |\log \varepsilon_j|^2}  \int_{\Omega_{\varepsilon_j^{\alpha}}(\mu_{j})} \frac12 \mathcal{C}(\varepsilon_j |\log \varepsilon_j| \beta + \eta^{j} - \tilde{K}_{\mu_{j}}^{r_{\varepsilon_j}} + \tilde{\beta}_{j}) 
 :  ( \varepsilon_j |\log \varepsilon_j| \beta + \eta^{j} - \tilde{K}_{\mu_{j}}^{r_{\varepsilon_j}} + \tilde{\beta}_{j}) \, dx \\
 &+ \frac{1}{\varepsilon_j^2 |\log \varepsilon_j|^2} \int_{\Omega_{\varepsilon_j^{\alpha}}(\mu_{j})} \sigma(\varepsilon_j |\log \varepsilon_j| \beta + \eta^{j} - \tilde{K}_{\mu_{j}}^{r_{\varepsilon_j}} + \tilde{\beta}_{j}) \, dx,
\end{align*}
where $\frac{\sigma(F)}{|F|^2} \to 0$ as $F \to 0$.  \\ \\
By \ref{item: etaweak}.-- \ref{item: tildebetastrong}., all mixed terms and the quadratic terms involving $\tilde{K}_{\mu_{j}}^{r_{\varepsilon_j}}$ or $\tilde{\beta}_{j}$ in the first integral vanish in the limit.
In addition, from the non-negativity of $\mathcal{C} = \frac{\partial^2 W}{\partial^2 F}(Id)$ it follows that
\begin{equation*}
 \frac{1}{\varepsilon_j^2 |\log \varepsilon_j|^2}  \int_{\Omega_{\varepsilon_j^{\alpha}}(\mu_{j})} \frac12 \mathcal{C}(\varepsilon_j |\log \varepsilon_j| \beta) : (\varepsilon_j |\log \varepsilon_j| \beta) \, dx \leq
 \int_{\Omega} \frac12 \mathcal{C} \beta: \beta \, dx.
\end{equation*}
Hence, it remains to consider the term involving $\eta^{j}$.
Using the specific form of the $\eta_i^j$ for the second equality and Proposition \ref{prop: convpsi} (in particular \eqref{def: selfenergy})  for the inequality, we find that
\begin{align*}
 &\frac{1}{\varepsilon_j |\log \varepsilon_j|^2} \int_{\Omega_{\varepsilon_j^{\alpha}}(\mu_{j})} \frac12 \mathcal{C} \eta^{j} : \eta^{j} \, dx \\
 = &\frac1{|\log \varepsilon_j|^2} \sum_{i = 1}^{M_j}  \int_{B_{r_{\varepsilon_j}}(x_{i,j}) \setminus B_{\varepsilon_j^{\alpha}}(x_{i,j})}\frac12 \mathcal{C} \eta_i^{j} : \eta_i^{j} \, dx \\
 = &\frac{|\log \frac{\varepsilon_j^{\alpha}}{r_{\varepsilon_j}}|}{|\log \varepsilon_j|^2} \sum_{i = 1}^{M_j} \frac1{|\log \frac{\varepsilon_j^{\alpha}}{r_{\varepsilon_j}}|} \int_{B_{1}(x_{i,j}) \setminus B_{\frac{\varepsilon_j^{\alpha}}{r_{\varepsilon_j}}}(x_{i,j})}\frac12 \mathcal{C} \eta_i^{j} : \eta_i^{j} \, dx \\
 \leq  &\frac{|\log \frac{\varepsilon_j^{\alpha}}{r_{\varepsilon_j}}|}{|\log \varepsilon_j|^2} \sum_{i = 1}^{M_j} (\psi(R^T \xi_{i,j}) + o(1)). \\
 =  &\frac{\alpha + o(1)}{|\log \varepsilon_j|} \sum_{i = 1}^{M_j} (\psi(R^T \xi_{i,j}) + o(1)), \\
 \intertext{where $o(1) \to 0$ as $\varepsilon_j \to 0$ (recall that we deal only with finitely many values of $\xi_{i,j}$). By definition of $\mu_{j}^k$, this equals}
 = &\sum_{k = 1}^{M} (\alpha + o(1)) \frac{|\mu^k_{\varepsilon_j}|(\Omega)}{|\log \varepsilon_j|}   (\psi(R^T \xi_k)+ o(1)).
\end{align*}
Using \eqref{eq: |mu|}, this yields in the limit
\begin{equation*}
 \limsup_{\varepsilon_j \to 0} \frac{1}{\varepsilon_j^2 |\log \varepsilon_j|^2} \int_{\Omega_{\varepsilon_j^{\alpha}}(\mu_{j})} \frac12 \mathcal{C} \eta^{j} : \eta^{j} \, dx \leq 
  \alpha |\Omega| \sum_{k=1}^M \lambda_k \psi(R^T \xi_k) = \alpha \int_{\Omega} \varphi(R,\xi) \, dx.
\end{equation*}
To establish \eqref{eq: Iepsilon1} we still have to show that 
\begin{equation}
 \frac{1}{\varepsilon_j^2 |\log \varepsilon_j|^2} \int_{\Omega_{\varepsilon_j^{\alpha}}(\mu_{j})} \sigma(\varepsilon_j |\log \varepsilon_j| \beta + \eta^{j} - \tilde{K}_{\mu_{j}}^{r_{\varepsilon_j}} + \tilde{\beta}_{j}) \, dx
 \rightarrow 0. \label{eq: sigmavanish}
\end{equation}
First, we observe that for $x \in \Omega_{\varepsilon_j^{\alpha}}(\mu_{j})$ it holds
\begin{gather*}
 |\eta^{j}(x)| \leq  \sup_{i=1, \dots, M_j} \varepsilon_j | \1_{B_{r_{\varepsilon_j}}(x_{i,j})} \eta^{j}_i | \leq C \varepsilon_j^{1-\alpha} \\
 \text{ and } |\tilde{K}_{\mu_{j}}^{r_{\varepsilon_j}}(x)| \leq \frac{\varepsilon_j}{r_{\varepsilon_j}^2}\sup_{i=1, \dots, M_j} | \1_{B_{r_{\varepsilon_j}}(x_{i,j})} \eta^{j}_i\, |x - x_{i,j}|^2 | \leq C \frac{\varepsilon_j}{r_{\varepsilon_j}}.
\end{gather*}
Hence, $\1_{\Omega_{\varepsilon_j^{\alpha}}(\mu_{j})}(\eta^{j} + \tilde{K}_{\mu_{j}}^{r_{\varepsilon_j}})$ converges uniformly to zero. 
To compensate the lack of uniform convergence of $\tilde{\beta}_{j}$ and $\varepsilon_j |\log \varepsilon_j| \beta$, fix $L>0$ and define the set
\begin{equation*}
U^L_{\varepsilon_j} = \left\{ x \in \Omega_{\varepsilon_j^{\alpha}(\mu_{j})}: |\tilde{\beta}_{j}(x)| \leq | \eta^{j}(x) + \tilde{K}_{\mu_{j}}^{r_{\varepsilon_j}}(x)| \text{ and } |\beta(x)| \leq L \right\}.
\end{equation*}
Then $\1_{U^L_{\varepsilon_j}}(\varepsilon_j |\log \varepsilon_j| \beta + \eta^{j} + \tilde{K}_{\mu_{j}}^{r_{\varepsilon_j}} + \tilde{\beta}_{\varepsilon_j})$ converges uniformly to zero.
Set $\omega(t) = \sup_{|F| \leq t} |\sigma(F)|$ and notice that $\frac{\omega(t)}{t^2} \to 0$ as $t \to 0$.
By definition of $\omega$ it holds
\begin{align}
 &\frac{1}{\varepsilon_j^2 |\log \varepsilon_j|^2} \left| \int_{U^L_{\varepsilon_j}} \sigma(\varepsilon_j |\log \varepsilon_j| \beta + \eta^{j} - \tilde{K}_{\mu_{j}}^{r_{\varepsilon_j}} + \tilde{\beta}_{j}) \, dx\right| \nonumber \\
 \leq &\int_{U_{\varepsilon_j}^L} \frac{\omega(\varepsilon_j |\log \varepsilon_j| \beta + \eta^{j} - \tilde{K}_{\mu_{j}}^{r_{\varepsilon_j}} + \tilde{\beta}_{j})}{\left|\varepsilon_j |\log \varepsilon_j| \beta + \eta^{j} - \tilde{K}_{\mu_{j}}^{r_{\varepsilon_j}} + \tilde{\beta}_{j}\right|^2}
 \frac{\left|\varepsilon_j |\log \varepsilon_j| \beta + \eta^{j} - \tilde{K}_{\mu_{j}}^{r_{\varepsilon_j}} + \tilde{\beta}_{j}\right|^2}{\varepsilon_j^2 |\log \varepsilon_j|^2} \, dx \rightarrow 0 \label{eq: uepsilonvanish}
\end{align}
as the first term converges to zero uniformly and the second is bounded in $L^1$ by \ref{item: etaweak}.-- \ref{item: tildebetastrong}. \\
For the integral on $\Omega_{\varepsilon_j^{\alpha}}(\mu_{j}) \setminus U^L_{\varepsilon_j}$, we notice:
\begin{equation*}
 \omega(t) = \sup_{|F| \leq t} |\sigma(F)| =  \sup_{|F| \leq t} \left| W(Id + F) - \frac12 \mathcal{C} F: F \right| \leq C \sup_{|F| \leq t} |F|^2 \leq C t^2.
\end{equation*}
Thus,
\begin{align*}
 &\frac{1}{\varepsilon_j^2 |\log \varepsilon_j|^2} \left|\int_{\Omega_{\varepsilon_j^{\alpha}}(\mu_{j}) \setminus U_{\varepsilon_j}^L} \sigma(\varepsilon_j |\log \varepsilon_j| \beta + \eta^{j} - \tilde{K}_{\mu_{j}}^{r_{\varepsilon_j}} + \tilde{\beta}_{j}) \, dx \right|\\
 \leq & C \int_{\Omega_{\varepsilon_j^{\alpha}}(\mu_{j}) \setminus U^L_{\varepsilon_j}} \frac{\left| \varepsilon_j |\log \varepsilon_j| \beta + \eta^{j} - \tilde{K}_{\mu_{j}}^{r_{\varepsilon_j}} + \tilde{\beta}_{j} \right|^2}{\varepsilon_j^2 |\log \varepsilon_j|^2} \, dx \\
 \leq & C \int_{\{ |\beta| >L \}} \frac{|\tilde{\beta}_{j}|^2}{\varepsilon_j^2 |\log \varepsilon_j|^2} + |\beta|^2 \, dx \stackrel{j\to \infty}{\longrightarrow} \int_{\{|\beta| > L\}} |\beta|^2 \,dx \stackrel{L \to \infty}{\longrightarrow} 0,
\end{align*}
where we used \ref{item: tildebetastrong}.~for the convergence in $j$. 
Hence, we proved \eqref{eq: sigmavanish} which in turn finishes the proof of \eqref{eq: Iepsilon1}. \\ \newline
\noindent Next, we control $I_2^{\varepsilon_j}$ from \eqref{eq: split}.
Notice that
\begin{align*}
 &\frac{1}{\varepsilon_j^2 |\log \varepsilon_j|^2} \sum_{i=1}^{M_j} \int_{B_{\varepsilon_j^{\alpha}}(x_{i,j})} W(\beta_{j}) \,dx \\
 \leq &C \frac{1}{\varepsilon_j^2 |\log \varepsilon_j|^2} \sum_{i=1}^{M_j} \int_{B_{\varepsilon_j^{\alpha}}(x_{i,j})} \operatorname{dist}(\beta_{j},SO(2))^2 \wedge \operatorname{dist}(\beta_{j},SO(2))^p \,dx \\
 \leq &C \frac{1}{\varepsilon_j^2 |\log \varepsilon_j|^2} \sum_{i=1}^{M_j} \int_{B_{\varepsilon_j^{\alpha}}(x_{i,j})} \varepsilon_j^2 |\log \varepsilon_j|^2 |\beta|^2 + |\eta^{j}|^2 \wedge |\eta^{j}|^p  + |\tilde{K}_{\mu_{j}}^{r_{\varepsilon_j}}|^2 + |\tilde{\beta}_{j}|^2 \,dx.
\end{align*}
Due to \ref{item: Kstrong}.~and \ref{item: tildebetastrong}., the terms involving $\tilde{K}_{\mu_{j}}^{r_{\varepsilon_j}}$ and $\tilde{\beta}_{j}$ vanish in the limit.
Moreover, as 
\begin{equation*}
 \mathcal{L}^2\left(\bigcup_{i=1}^{M_j} B_{\varepsilon_j^{\alpha}}(x_{i,j})\right) = M_j \pi \varepsilon_j^{2{\alpha}} \leq C |\log \varepsilon_j| \varepsilon_j^{2{\alpha}} \to 0,
\end{equation*}
also the term involving $\beta$ vanishes in the limit.
Eventually, we estimate
\begin{align*}
 &\frac{1}{\varepsilon_j^2 |\log \varepsilon_j|^2} \sum_{i=1}^{M_j} \int_{B_{\varepsilon_j^{\alpha}}(x_{i,j})} |\eta^{j}|^2 \wedge |\eta^{j}|^p \, dx \\
 \leq &\frac{C}{\varepsilon_j^{2} |\log \varepsilon_j|^2} \sum_{i=1}^{M_j} \left(\int_{B_{\varepsilon_j^{\alpha}}(x_{i,j}) \setminus B_{\varepsilon_j}(x_{i,j})} |\eta^{j}|^2 \, dx
 + \int_{B_{\varepsilon_j}(x_{i,j})} |\eta^{j}|^p \, dx\right) \\
 \leq &\frac{C M_j}{|\log \varepsilon_j|^2} \int_{\varepsilon_j}^{\varepsilon_j^{\alpha}} r^{-1} \, dr
 + \frac{C M_j}{\varepsilon_j^{2-p} |\log \varepsilon_j|^2} \int_{0}^{\varepsilon_j} r^{1-p} \, dr \\
 \leq &C (1-{\alpha}) + C |\log \varepsilon_j|^{-1}.
\end{align*}
Hence, $\limsup_{\varepsilon_j \to 0} I_{\varepsilon_j}^2 \leq C(1-{\alpha})$.
Together with \eqref{eq: Iepsilon1}, this implies 
\begin{equation*}
 \limsup_{\varepsilon_j \to 0} E_{\varepsilon_j}(\beta_{j}, \mu_{j}) \leq \int_{\Omega} \frac12 \mathcal{C} \beta : \beta \,dx + \alpha \int_{\Omega} \varphi\left(R, \frac{d \mu}{d |\mu|}\right) + C(1-\alpha).
\end{equation*}
Letting $\alpha \to 1$ finishes step 1. \\
For step 2, it is useful to notice here that $\eta^j = \tilde{K}_{\mu_{j}}^{r_{\varepsilon_j}} = 0$  on $\partial \Omega$ and therefore we find by \eqref{eq: lemma11} and \ref{item: tildebetastrong}.~that (cf.~\cite[Theorem 2]{DaLi88})
\begin{align} \label{eq: tracevanishes}
 \left( \frac{R^T \beta_{j} - Id}{\varepsilon_j |\log \varepsilon_j|} - \beta \right) \cdot \tau &= 
 \left( \frac{\eta^{j} - \tilde{K}_{\mu_{j}}^{r_{\varepsilon_j}} + \tilde{\beta}_{j}}{\varepsilon_j |\log \varepsilon_j|}  \right) \cdot \tau \\ &=  \frac{\tilde{\beta}_{j}}{\varepsilon_j |\log \varepsilon_j|}   \cdot \tau\rightarrow 0 \text{ strongly in } H^{-\frac12}(\partial\Omega;\R^2), \nonumber
\end{align}
where $\tau$ denotes the unit tangent to $\partial \Omega$. \\ \\
\noindent \textbf{Step 2. }\emph{$\mu = \sum_{l=1}^L \xi^l \,d\mathcal{L}^2_{|\Omega^l}$ where $\xi^l \in \R^2$ and $\Omega^l \subseteq \Omega$ are pairwise disjoint Lipschitz-domains such that $\mathcal{L}^2\left( \Omega \setminus \bigcup_{l=1}^L \Omega^l \right) = 0$.} \\
We make use of the recovery sequence of step 1 on each $\Omega^l$. 
For this, we define $\beta^l = \beta \1_{\Omega^l}$ and $\mu^l= \mu_{|\Omega^l}$. 
For each $l=1, \dots , L$ let $(\mu^l_{j},\beta_{j}^l)_j$ be the recovery sequence from step 1 for $( \mu^l,\beta^l, R)$ on $\Omega^l$.
Now, define 
\begin{equation*}
 \tilde{\beta}_{j} = \sum_{l=1}^L \beta_{j}^l \1_{\Omega^l}.
\end{equation*}
Then
\begin{equation*}
 \operatorname{curl} \tilde{\beta}_{j} = \sum_{l=1}^L \mu^l_{j} - (\beta_{j}^l  \cdot \tau_{\partial \Omega^l}) \, d\mathcal{H}^1_{|\del \Omega^l \cap \Omega} \text{ in } \mathcal{D}'(\Omega),
\end{equation*}
where $\tau_{\partial \Omega^l}$ is the positively oriented unit tangent to $\partial \Omega^l$.
Note that for two neighboring regions $\Omega^l$ the corresponding tangents have opposite signs.
By \eqref{eq: tracevanishes}, we find
\begin{align*}
 &\left\|\frac{\operatorname{curl} \tilde{\beta}_{j} - \sum_{l=1}^L \mu^l_{j}}{\varepsilon_j |\log \varepsilon_j|}\right\|_{H^{-1}(\Omega;\R^2)} \\ =
& \left\| \sum_{l=1}^L  \left(\frac{\beta_{j}^l - R}{\varepsilon_j |\log \varepsilon_j|} - R\beta\right)  \cdot \tau_{\partial \Omega^l \cap \Omega} \, d\mathcal{H}^1_{|\del \Omega^l} \right\|_{H^{-1}(\Omega;\R^2)} \\ 
 \leq &C \sum_{l=1}^L \left\| R \left(\frac{R^T\beta_{j}^l - Id}{\varepsilon_j |\log \varepsilon_j|} - \beta\right)  \cdot \tau_{\partial \Omega^l} \right\|_{H^{-\frac12}(\del \Omega^l;\R^2)} \longrightarrow 0 \text{ as } j \to \infty.
\end{align*}
For the last inequality we used that for a Lipschitz domain $U$ the trace space of $H^1(U;\R^2)$ is $H^{\frac12}(\partial U;\R^2)$.
By this estimate in $H^{-1}$, we can find a sequence of functions $f_{j} \in L^2(\Omega;\R^{2\times2})$ such that $\operatorname{curl} f_{j} = \operatorname{curl} \tilde{\beta}_{j} - \sum_{l=1}^L \mu^l_{j}$ and 
$\frac{1}{\varepsilon_j |\log \varepsilon_j|} \| f_{j} \|_{L^2} \rightarrow 0$.
Now, define the recovery sequence as
\begin{equation*}
 \beta_{j} = \tilde{\beta}_{j} - f_{j} \text{ and } \mu_{j} = \sum_{l=1}^L \mu^l_{j}.
\end{equation*}
Then $\mu_{j} \in X_{\varepsilon_j}$ and $\beta_{j} \in \mathcal{A}\mathcal{S}_{\varepsilon_j}(\mu_{j})$.
From the construction of the $\mu_j^l$ in step 1 it follows that $\frac{\mu_j}{ \varepsilon_j |\log \varepsilon_j|} \stackrel{*}{\rightharpoonup} \mu$ in $\mathcal{M}(\Omega;\R^2)$.
Moreover, in the proof of step 1 it can be seen that although we subtract the vanishing sequence $f_{j}$ it still holds for all $l = 1,\dots ,L$
\begin{equation}
 \limsup_{\varepsilon_j \to 0} \frac{1}{\varepsilon_j^2 |\log \varepsilon_j|^2} \int_{\Omega^l} W(\beta_{j})\, dx  \leq \int_{\Omega^l} \mathcal{C} \beta: \beta \, dx + \int_{\Omega^l} \varphi\left(R, \frac{d\mu^l}{d|\mu^l|}\right) \, d|\mu^l|. \label{eq: localizedrec}
\end{equation}
Summing over \eqref{eq: localizedrec} finishes step 2.\\ \\
\noindent \textbf{Step 3. }\emph{$\mu \in H^{-1}(\Omega;\R^2) \cap \mathcal{M}(\Omega;\R^2)$.}\\
As our limit energy is the same, we can argue as in \cite[Theorem 12, step 3]{GaLePo10} to reduce the general case to step 2. 
Let us shortly sketch the argument for the sake of completeness. By reflection arguments and mollification, the authors show that there exists a sequence of smooth functions $\beta_j$ such that $\beta_j \rightarrow \beta$ in $L^2(\Omega;\R^{2\times 2})$, $\operatorname{curl }\beta_j \stackrel{*}{\rightharpoonup} \operatorname{curl }\beta$ in $\mathcal{M}(\Omega;\R^2)$, and $| \operatorname{curl }\beta_j|(\Omega) \rightarrow | \operatorname{curl }\beta|(\Omega)$.
The energy $E^{crit}$ is continuous with respect to this convergence.
Then, the authors carefully approximate $ \operatorname{curl }\beta$ by piecewise constant functions and correct the corresponding error in the $\operatorname{curl}$ by a vanishing sequence in $L^2(\Omega;\R^{2\times2})$.
\end{proof}

\begin{remark}
In \cite{MuScZe15}, the authors construct a recovery sequence $\beta_j$ which fulfills $\operatorname{det } \beta_j > 0$. 
This construction could also be used in our case. 
Most computations in the proof would remain the same.
Using this construction, we could weaken our assumptions on $W$ in the sense that we would need the upper bound $W(F) \leq C \operatorname{dist}(F,SO(2))^2 \wedge \operatorname{dist}(F,SO(2))^p$ only for $F$ such that $\operatorname{det}(F) > 0$.
\end{remark}

\section*{Acknowledgments}
The author is very grateful to Stefan M\"uller and Sergio Conti for bringing the problem to his attention and many fruitful discussions.  

\bibliographystyle{abbrv}
\bibliography{mixed-growth-arxiv}
\end{document}